\documentclass[12pt]{amsart}
\usepackage{tikz-cd}
\usetikzlibrary{matrix,arrows,decorations.pathmorphing}
\usepackage{amssymb}
\usepackage{microtype}
\usepackage{amsmath,amscd}
\usepackage{mathrsfs}
\usepackage{url}
\usepackage{cite}
\usepackage{fullpage}
\usepackage{hyperref}
\usepackage{verbatim}
\usepackage{comment}
\usepackage{fancyvrb}
\usepackage{fvextra}
\usepackage{mathrsfs}
\newtheorem{thm}{Theorem}[section]
\newtheorem{prop}[thm]{Proposition}
\newtheorem{lem}[thm]{Lemma}
\newtheorem{cor}[thm]{Corollary}

\usepackage{wasysym}
\usepackage{graphicx}
\usepackage{wrapfig}

\theoremstyle{definition}
\newtheorem{definition}[thm]{Definition}
\newtheorem{example}[thm]{Example}
\newtheorem{rem}[thm]{Remark}

\numberwithin{equation}{section}
\newcommand{\nf}{\mathrm{nf}}
\newcommand{\aV}{\mathcal{U}}
\newcommand{\U}{\mathcal{U}}
\newcommand{\W}{\mathcal{W}}
\newcommand{\aW}{\mathcal{X}}
\newcommand{\Wp}{\mathcal{X}}
\newcommand{\M}{\mathcal{M}}
\newcommand{\X}{\mathcal{Y}}
\newcommand{\Q}{\mathcal{Q}}
\newcommand{\T}{\mathcal{T}}

\newcommand{\J}{\mathcal{J}}

\newcommand{\B}{\mathcal{B}}

\newcommand{\qq}{\mathbb{Q}}

\newcommand{\I}{\mathcal{I}}
\newcommand{\C}{\mathcal{C}}
\newcommand{\Cp}{\mathcal{C}}
\newcommand{\Pp}{\mathcal{P}}
\newcommand{\Ep}{\mathcal{E}}
\newcommand{\p}{\mathbb{P}}
\newcommand{\pp}{\mathbb{P}}

\renewcommand{\H}{\mathcal{H}}
\newcommand{\Hp}{\mathcal{H}}

\newcommand{\F}{\mathcal{F}}
\renewcommand{\P}{\mathcal{P}}
\newcommand{\E}{\mathcal{E}}
\renewcommand{\O}{\mathcal{O}}

\newcommand{\Mg}{\M_g}

\renewcommand{\tilde}{\widetilde}
\DeclareMathOperator{\rank}{rank}

\DeclareMathOperator{\GL}{GL}
\DeclareMathOperator{\SL}{SL}
\DeclareMathOperator{\Pic}{Pic}

\DeclareMathOperator{\Supp}{Supp}

\DeclareMathOperator{\Sym}{Sym}

\DeclareMathOperator{\PGL}{PGL}
\DeclareMathOperator{\BSL}{BSL}

\DeclareMathOperator{\codim}{codim}

\DeclareMathOperator{\Proj}{Proj}

\newcommand{\hannah}[1]{{\color{teal} ($\spadesuit$ Hannah: #1)}}
\newcommand{\sam}[1]{{\color{red} ($\spadesuit$ Sam: #1)}}

\begin{document}
\title{Chow rings of low-degree Hurwitz spaces}

\author{Samir Canning, Hannah Larson}
\thanks{During the preparation of this article, S.C. was partially supported by NSF RTG grant DMS-1502651. H.L. was supported by the Hertz Foundation and NSF GRFP under grant DGE-1656518. This work will be part of S.C.'s and H.L.'s Ph.D. theses.}
\email{srcannin@ucsd.edu}
\email{hlarson@stanford.edu}
\maketitle


\begin{abstract}
While there is much work and many conjectures surrounding the intersection theory of the moduli space of curves, relatively little is known about the intersection theory of the Hurwitz space $\Hp_{k, g}$ parametrizing smooth degree $k$, genus $g$ covers of $\pp^1$. Let $k = 3, 4, 5$. We prove that the rational Chow rings of $\Hp_{k,g}$ stabilize in a suitable sense as $g$ tends to infinity. In the case $k = 3$, we completely determine the Chow rings for all $g$. 
We also prove that the rational Chow groups of the simply branched Hurwitz space $\Hp^s_{k,g} \subset \Hp_{k,g}$ are zero in codimension up to roughly $g/k$. In \cite{CL}, results developed in this paper are used to prove that the Chow rings of $\M_7, \M_8,$ and $\M_9$ are tautological.
\end{abstract}

\section{Introduction}
Intersection theory on the moduli space of curves $\M_g$ has received much attention since Mumford's famous paper \cite{Mum}, in which he introduced the Chow ring of $\M_g$. Based on Harer's result \cite{H} that the cohomology of the moduli space of curves is independent of the genus $g$ in degrees small relative to $g$, Mumford conjectured that the stable cohomology ring is isomorphic to $\qq[\kappa_1, \kappa_2, \kappa_3, \ldots]$. Madsen--Weiss \cite{MW} later proved Mumford's conjecture. 
It is unknown whether there is an analogous stabilization result in the Chow ring of $\M_g$. Upon restricting attention to the tautological ring, however, more is known.

The \emph{tautological subring} $R^*(\M_g) \subseteq A^*(\M_g)$ is defined to be the subring of the rational Chow ring generated by the kappa classes. There are many conjectures concerning the relations and structure of the tautological ring. Prominent among them is Faber's conjecture \cite[Conjecture 1]{F}, which states that the tautological ring should be Gorenstein with socle in codimension $g-2$ and generated by the first $\lfloor g/3\rfloor$ kappa classes with no relations in degree less than $\lfloor g/3\rfloor$. 
The Gorenstein part of Faber's conjecture is unknown, although it has been shown to hold when $g\leq 23$ by a direct computer calculation of Faber. The second portion of Faber's conjecture has been proved:
Ionel \cite{I} proved that the tautological ring is generated by $\kappa_1, \kappa_2, \ldots, \kappa_{\lfloor g/3 \rfloor}$, and Boldsen \cite{B} proved that there are no relations among the $\kappa$-classes in degrees less than $\lfloor g/3 \rfloor$. In other words, there is a surjection
\begin{equation} \label{forM}
\qq[\kappa_1, \kappa_2, \ldots, \kappa_{\lfloor g/3 \rfloor}] \twoheadrightarrow R^*(\M_g),
\end{equation}
which is an isomorphism in degrees less than $\lfloor g/3 \rfloor$.

In this paper, we study the Chow rings of low-degree Hurwitz spaces. Our main theorem is a stabilization result of a similar flavor to \eqref{forM}.
Let $\Hp_{k,g}$ be the Hurwitz stack parametrizing degree $k$, genus $g$ covers of $\pp^1$ up to automorphisms of the target.
Let $\Cp$ be the universal curve and $\Pp$ the universal $\pp^1$-fibration over the Hurwitz space $\Hp_{k,g}$:
\begin{center}
\begin{tikzcd}
\Cp \arrow{r}{\alpha} \arrow{rd}[swap]{f} & \Pp \arrow{d}{\pi} \\
& \Hp_{k,g}.
\end{tikzcd}
\end{center}
We define the \emph{tautological subring} of the Hurwitz space $R^*(\Hp_{k,g}) \subseteq A^*(\Hp_{k,g})$ to be the subring generated by classes of the form $f_*(c_1(\omega_f)^i \cdot \alpha^*c_1(\omega_\pi)^j) = \pi_*(\alpha_*(c_1(\omega_f)^i)  \cdot c_1(\omega_\pi)^j)$.
Let $\Ep^\vee$ be the cokernel of the map $\O_{\Pp} \to \alpha_* \O_{\Cp}$ (the universal ``Tschirnhausen bundle"). Set $z =  -\frac{1}{2}c_1(\omega_\pi) ``= c_1(\O_{\Pp}(1))"$.
Our theorem will be stated in terms of the tautological classes $c_2 = -\pi_*(z^3) \in A^2(\Hp_{k,g})$ and
\[a_i  = \pi_*(c_i(\Ep) \cdot z) \in A^i(\Hp_{k,g}) \qquad \text{and} \qquad a_i' = \pi_*(c_i(\Ep))  \in A^{i-1}(\Hp_{k,g}).\]

When $k=3,4,5$, our main theorem gives a minimal set of generators for $R^*(\Hp_{k,g})$ and determines all relations among them in degrees up to roughly $g/k$. In contrast with the case of $\Mg$ in \eqref{forM}, the tautological ring of $\Hp_{k,g}$ does \emph{not} require a growing number of generators as $g$ increases.
In degree $3$, we determine the full Chow ring of $\Hp_{3,g}$.
When $k = 3, 5$, our results imply that the dimensions of the Chow groups of $\Hp_{k,g}$ are independent of $g$ for $g$ sufficiently large. In degree $4$, \emph{factoring covers} --- i.e. covers $C \to \pp^1$ that factor as a composition of two double covers $C \to C' \to \pp^1$ --- present a difficulty. We instead obtain stabilization results for the Chow groups of $\Hp_{4,g}^{\mathrm{nf}} \subseteq \Hp_{4,g}$, the open substack parametrizing non-factoring covers, or equivalently covers whose monodromy group is not contained in the dihedral group $D_4$.

\begin{thm}\label{main}
Let $g\geq 2$ be an integer. 
\begin{enumerate}
    \item The rational Chow ring of $\Hp_{3,g}$ is
    \[
    A^*(\Hp_{3,g}) = R^*(\Hp_{3,g}) =\begin{cases}\qq & \text{if $g=2$} \\
    \qq[a_1]/(a_1^2) & \text{if $g=3,4,5$} \\
    \qq[a_1]/(a_1^3) & \text{if $g\geq 6$.} \end{cases}
    \]
    \item Let $r_i = r_i(g)$ be defined as in Section \ref{ach}. For each $g$ there is a map
    \[\frac{\qq[a_1, a_2', a_3']}{\langle r_1, r_2, r_3, r_4 \rangle} \twoheadrightarrow R^*(\Hp_{4,g}) \subseteq A^*(\Hp_{4,g}) \to A^*(\Hp_{4,g}^{\mathrm{nf}}),\] such that the composition is an isomorphism in degrees up to $\frac{g+3}{4} - 4$.
    Furthermore, the dimension of the Chow group $A^i(\Hp^{\mathrm{nf}}_{4,g})$ is independent of $g$ for $g>4i+12$. When $g>4i+12$, the dimensions are given by
    \[
    \dim A^i(\Hp^{\mathrm{nf}}_{4,g}) = \dim R^i(\H_{4,g}) =\begin{cases}
    2 & i=1,4
    \\ 
    4 & i=2
    \\
    3 & i=3
    \\
    1 &i\geq 5.
    \end{cases}
    \]
    \item Let $r_i = r_i(g)$ be as defined in Section \ref{ach5}. There is a map
    \[\frac{\qq[a_1, a_2', a_2, c_2]}{\langle r_1, r_2, r_3, r_4, r_5 \rangle} \twoheadrightarrow R^*(\Hp_{5,g}) \subseteq A^*(\Hp_{5,g}) \]
    such that the composition is an isomorphism in degrees $\leq \frac{g+4}{5} - 16$.
    Furthermore, the dimension of the Chow group $A^i(\Hp_{5,g})$ is independent of $g$ for $g>5i+76$. When $g>5i+76$, the dimensions are given by
    \[
    \dim A^i(\Hp_{5,g})= \dim R^i(\H_{5,g}) = \begin{cases}
    2 & i=1, i\geq 7
    \\
    5 & i=2
    \\
    6 & i=3
    \\
    7 & i=4
    \\
    4 & i=5
    \\
    3 & i=6.
    \end{cases}
    \]
\end{enumerate}
\end{thm}



\begin{rem}
Angelina Zheng recently computed the rational cohomology of $\Hp_{3,5}$ in \cite{Z}, and, in subsequent work \cite{Z2}, finds the stable rational cohomology of $\Hp_{3,g}$. Together, our results prove that the cycle class map is injective. The corresponding statement for $\Mg$ is unknown, but when $g\leq 6$ it follows from the fact that the tautological ring is the entire Chow ring. 
\end{rem}

\begin{rem} \label{vanr}
Note that Theorem \ref{main}(2) implies that the restriction map $R^i(\Hp_{4,g}) \to R^i(\Hp_{4,g}^{\mathrm{nf}})$ is an isomorphism for $i < \frac{g+3}{4} - 4$. This implies an interesting vanishing result: Any tautological class of codimension $i < \frac{g+3}{4} - 4$ supported on the locus of factoring covers is zero.
\end{rem}

\begin{rem}
In \cite{B3, B4, B5, BD4, BD5}, Bhargava famously applied structure theorems for degree $3, 4$, and $5$ covers to counting number fields. As in Bhargava's work, our techniques rely on special aspects of structure theorems that do not seem to extend to covers of degree $k \geq 6$.
Our need to throw out factoring covers in order to obtain asymptotic results for the full Chow ring seems to parallel the fact that, when quartic covers are counted by discriminant, the $D_4$ covers constitute a positive proportion of all covers \cite[Theorem 4]{BD4}.
\end{rem}

\begin{rem}
 Ellenberg-Venkatesh-Westerland \cite{EVW} have studied stability in the \emph{homology} of Hurwitz spaces of $G$ covers (which in particular separates out factoring covers). Like the work of Harer and Madsen-Weiss, their techniques are topological. On the other hand, ours are algebro-geometric: they are about the Chow groups rather than (co)homology and they work in characteristic $p>5$ without the use of a comparison theorem.
\end{rem}

\begin{rem}
For $g$ suitably large, our proof of Theorem \ref{main} (2) shows that $\dim R^i(\Hp_{4,g}) \leq 1$ for all $i \geq 5$, and similarly in (3) that
$\dim R^i(\Hp_{5,g}) \leq 2$ for all $i \geq 7$. Hence,  $R^*(\Hp_{4,g})$ and $R^*(\Hp_{5,g})$ are \emph{not} Gorenstein because there cannot be a perfect pairing for dimension reasons. On the other hand, $A^*(\Hp_{3,g})=R^*(\Hp_{3,g})$ is Gorenstein.  \end{rem}

Our method of proof is to study a large open substack $\Hp_{k,g}^\circ \subset \Hp_{k,g}$, which can be represented as an open substack of a vector bundle $\Wp_{k,g}^\circ$ over a certain moduli stack of vector bundles on $\pp^1$. The fact that the moduli space admits such a description comes from the structure theorms of degree $3, 4, 5$ covers and is precisely what is so special about these low-degree cases.
We then determine the Chow ring of $\Hp_{k,g}^\circ$ via excision on the complement of $\Hp_{k,g}^\circ$ inside $\Wp_{k,g}^\circ$. This complement is a ``discriminant locus" parametrizing singular covers and maps that are not even finite. The stability of the Chow groups we find fits in with the philosophy of Vakil--Wood \cite{VW} about discriminants and suggests some possible variations on their theme. The key point, which is reflected in the ampleness assumptions in some of the conjectures from \cite{VW}, is that the covers we parametrize correspond to sections of a vector bundle that becomes ``more positive" as the genus of the curve grows.
We compute generators for the Chow ring of the discriminant locus by constructing a resolution whose Chow ring we can compute. See Figure \ref{sumfig} in Section \ref{rels-step2} for a picture summarizing our method.

We also give formulas in Section \ref{app-sec} that express other natural classes on $\Hp_{k,g}$ --- namely the $\kappa$-classes pulled back from $\M_g$ and the classes corresponding to covers with certain ramification profiles --- in terms of the generators from Theorem \ref{main}. We give two applications of these formulas. First, we show that for $k = 4, 5$, ``the push forward of tautological classes on $\Hp_{k,g}$ are tautological on $\M_g$." (The case $k = 3$ already follows from Patel--Vakil's result that $A^*(\Hp_{3,g}) = R^*(\Hp_{3,g})$ is generated by $\kappa_1$ when $g > 3$, and and all classes on $\M_3$ are tautological.)
Note that for $k > 3$, there are tautological classes on $\Hp_{k,g}$ that are \emph{not} pullbacks of tautological classes on $\M_g$: Theorem \ref{main} implies $\dim R^1(\Hp_{k,g}) > 1$, so it cannot be spanned by the pullback of $\kappa_1$. 
Hence, our claim regarding pushforwards is not a priori true.
To set the stage for the theorem, let $\beta: \Hp_{k,g} \to \M_g$ be the forgetful morphism. Define $\M_g^k \subset \M_g$ to be the locus of curves of gonality $\leq k$. There is a proper morphism $\beta': \Hp_{k,g}\setminus \beta^{-1}(\Mg^{k-1}) \to \M_g \smallsetminus \M_g^{k-1}$. We define a class to be tautological on $\M_g \smallsetminus \M_g^{k-1}$ if it is the restriction of a tautological class on $\M_g$.

\begin{thm}\label{taut}
Let $g\geq 2$ be an integer and $k\in\{3,4,5\}$. The $\beta'$ push forward of classes in $R^*(\Hp_{k,g})$ are tautological
on $\M_g \smallsetminus \M_g^{k-1}$.
\end{thm}

\begin{rem}
Theorem \ref{taut} is a key tool in recent work of the authors \cite{CL}, which proves
that the Chow rings of $\M_7, \M_8$ and $\M_9$ are tautological.
Because the tautological ring has been computed in these cases by Faber \cite{F}, this work settles the next open case in the program suggested by Mumford \cite{Mum} of determining the Chow ring of $\Mg$ for small $g$.
\end{rem}

\begin{rem}
We emphasize that 
when $k=4$, there can be non-tautological classes in low codimension supported on the locus of factoring covers. In particular, the fundamental class of the bielliptic locus on $\M_{12}$ is not tautological by a theorem of van Zelm \cite{VZ}, so Theorem \ref{taut} implies $R^*(\Hp_{4,g}) \neq A^*(\Hp_{4, g})$ for $g=12$.
\end{rem}

The second application of our formulas is to vanishing results for the Chow groups of 
the simply branched Hurwitz space
$\Hp^s_{k,g} \subseteq \Hp_{k,g}$. The Hurwitz space Picard rank conjecture \cite[Conjecture 2.49]{HM} says that $A^1(\Hp^s_{k,g}) = \Pic(\Hp^s_{k,g})\otimes \qq=0$ . This conjecture is known for $k\leq 5$ \cite{DP}, and for $k>g-1$ \cite{Mu}. 
In the cases $k =2, 3$, the stronger vanishing result $A^i(\Hp_{k,g}^s) = 0$ holds for all $i > 0$. 
The following theorem 
provides further evidence for a generalization of the Hurwitz space Picard rank conjecture to higher codimension cycles.

\begin{thm}\label{GPRC}
Let $g\geq 2$ be an integer. The rational Chow groups of the simply-branched Hurwitz space satisfy
    \begin{align*}
    A^i(\Hp^s_{4,g}) &=0 \qquad \text{for  $1\leq i < \frac{g+3}{4} - 4$}  \\
    A^i(\Hp^s_{5,g})&=0 \qquad \text{for $1\leq i < \frac{g+4}{5} - 16$.}
    \end{align*}
\end{thm}

The paper is structured as follows. In Section \ref{conventions}, we introduce some notational conventions and some basic ideas from (equivariant) intersection theory that we will use throughout the paper. We prove a lemma, the ``Trapezoid Lemma", which establishes a useful set up where one can determine all relations coming from certain excisions
with an appropriate resolution.
In Section \ref{partsec}, we introduce certain bundles of principal parts, which will be used throughout the remainder of the paper. Loosely speaking, these bundles help detect singularities and ramification behavior. As we shall see in the later sections of the paper, constructing a suitable principal parts bundle often requires geometric insights and can be somewhat involved.
In Sections \ref{RE3}, \ref{re4}, and \ref{re5}, we use principal parts bundles and the Trapezoid Lemma to produce relations among tautological classes in $A^*(\Hp_{3,g})$, $A^*(\Hp_{4,g})$, and $A^*(\Hp_{5,g})$, respectively. From these calculations, we obtain the proof Theorem \ref{main}.
Finally, in Section \ref{app-sec}, we rewrite the $\kappa$-classes and classes that parametrize covers with certain ramification behavior in terms of our preferred generators. These calculations allow us to prove Theorems \ref{taut} and \ref{GPRC}.
\par Several of the calculations in this paper were using the Macaulay2 \cite{M2} package Schubert2 \cite{S2}. All of the code used in this paper is provided in a Github repository \cite{github}. Whenever there is a reference to a calculation done with a computer, one can find the code to perform that calculation in the Github repository.

\subsection*{Acknowledgments} We are grateful to our advisors, Elham Izadi and Ravi Vakil, respectively, for the many helpful conversations. In addition, we are grateful to Anand Patel, who pointed out the need for special arguments for $\Hp_{3,g}$ when $g$ is small. We thank Maxwell da Paixão de Jesus Santos for his correspondence, which inspired us to formulate Theorem 1.7. We thank Aaron Landesman and Andrea Di Lorenzo for their comments and insights.

\section{Conventions and some intersection theory} \label{conventions}
We will work over an algebraically closed field of characteristic $0$ or characteristic $p>5$. All schemes in this paper will be taken over this fixed field.
\subsection{Projective and Grassmann bundles} \label{pandg}
We follow the subspace convention for projective bundles: given a scheme (or stack) $X$ and a vector bundle $E$ of rank $r$ on $X$, set
\[
\p E:=\Proj(\Sym ^{\bullet} E^{\vee}),
\]
so we have the tautological inclusion
\[
\mathcal{O}_{\p E}(-1)\hookrightarrow \gamma^*E,
\]
where $\gamma :\p E\rightarrow X$ is the structure map. Set $\zeta:=c_1(\O_{\p E}(1))$. With this convention, the Chow ring of $\p E$ is given by
\begin{equation} \label{pbt}
A^*(\p E)=A^*(X)[\zeta]/\langle \zeta^r + \zeta^{r-1} c_1(E) + \ldots + c_r(E)\rangle.
\end{equation}
We call this the \emph{projective bundle theorem}.
Note that $1, \zeta, \zeta^2, \ldots, \zeta^{r-1}$ form a basis for $A^*(\pp E)$ as an $A^*(X)$-module. Since
\[ \gamma_* \zeta^i = \begin{cases} 0 & \text{if $i \leq r-2$} \\ 1 & \text{if $i = r-1$,} \end{cases}\]
this determines the $\gamma_*$ of all classes from $\p E$.

More generally, we define the Grassmann bundle $G(n, E)$ of $n$-dimensional subspaces in $E$, which is equipped with a tautological sequence
\[0 \rightarrow S \rightarrow \gamma^* E \rightarrow Q\rightarrow 0
\]
where $\gamma: G(n, E) \rightarrow X$ is the structure map and $S$ has rank $n$. The relative tangent bundle of $G(n, E) \to X$ is $\H om(S, Q)$. The Chow ring $A^*(G(n, E))$ is generated as an $A^*(X)$-\textit{algebra} by the classes $\zeta_i = c_i(Q)$. Of particular interest to us will be Grassmann bundles $A^*(G(2, E))$ when the rank of $E$ is either $4$ or $5$. If the rank of $E$ is $4$, $A^*(G(2, E))$ is generated as a $A^*(X)$-\textit{module} by $\zeta_1^i\zeta_2^j$ for $0\leq i\leq 2$, $0\leq j\leq 2$, $0\leq i+j\leq 2$. If the rank of $E$ is $5$, $A^*(G(2, E))$ is generated as a $A^*(X)$ \textit{module} by $\zeta_1^i\zeta_2^j\zeta_3^k$ for $0\leq i\leq 2$, $0\leq j\leq 2$, $0\leq k\leq 2$ and $0\leq i+j+k\leq 2$. See \cite{GSS} for a much more general discussion on the Chow rings of flag bundles. In particular, these bases seem to be the preferred ones of the Macaulay2 \cite{M2} package Schubert2 \cite{S2}, which is what we use for calculations in this paper.

\subsection{The Trapezoid Lemma}

Let $\tau: V \to B$ be a rank $r$ vector bundle. If $\sigma$ is a section of $V$ which vanishes in codimension $r$, then the vanishing locus of $\sigma$ has fundamental class $c_r(V) \in A^r(B)$.
The identity induces a section of $\tau^*V$ on the total space of $V$ whose vanishing locus is the zero section. Thus, a special case of this fact is that the zero section in the total space of a vector bundle has class $c_r(\tau^*V) = \tau^*c_r(V) \in A^r(V) \cong \tau^* A^r(B)$. More generally, suppose $\rho: X \to B$ is another vector bundle on $B$ and we are given a map of vector bundles $\phi: X \to V$ over $B$.
Composing $\phi$ after the section induced by the identity on the total space of $X$ defines a section of $\rho^*V$ on the total space of $X$. We call the vanishing locus $K$ of this section the \emph{preimage under $\phi$ of the zero section in $V$}.
If $\phi$ is a surjection of vector bundles, then $K$ is simply the total space of the kernel subbundle.
If $K$ has codimension $r$ inside the total space of $W$, then its fundamental class is $[K] = c_r(\rho^*V) = \rho^*c_r(V) \in A^r(X) \cong \rho^*A^r(B)$.

A basic tool we shall use repeatedly is the following.
\begin{lem}[``Trapezoid push forwards"] \label{trap}
Suppose $\widetilde{B} \to B$ is proper (e.g. a tower of Grassmann bundles). Let $X$ be a vector bundle on $B$ and let $V$ be a vector bundle of rank $r$ on $\widetilde{B}$. Suppose that we are given a map of vector bundles $\phi: \sigma^*X \rightarrow V$ on $\widetilde{B}$. Let $K \subset \sigma^*X$ be the preimage under $\phi$ of the zero section in $V$, and suppose that $K$ has codimension $r$.
We call this a trapezoid diagram:
\begin{center}
\begin{tikzcd}
&K \arrow{dr}[swap]{\rho''} \arrow{r}{\iota} &\sigma^* X \arrow{d}{\rho'} \arrow{r}{\sigma'} & X \arrow{d}{\rho} \\
&&\widetilde{B} \arrow{r}[swap]{\sigma} & B.
\end{tikzcd}
\end{center}
The image of $(\sigma' \circ\iota)_*: A_*(K) \to A_*(X)$ contains the ideal generated by $\rho^*(\sigma_*(c_r(V) \cdot \alpha_i))$ as $\alpha_i \in A^*(\widetilde{B})$ ranges over generators for $A^*(\widetilde{B})$ as a $A^*(B)$-module. Equality holds if $\phi$ is a surjection.
In other words, we have a surjective map of rings
\[A^*(B)/ \langle \sigma_*(c_r(V) \cdot \alpha_i)) \rangle \rightarrow A^*(X \smallsetminus \sigma'(\iota(K))),\]
which is an isomorphism when $\phi$ is a surjection of vector bundles.
\end{lem}

\begin{proof}
The pullback maps $(\rho')^*$ and $\rho^*$ are isomorphisms on Chow rings.  The fundamental class of $K$ in $\sigma^*X$ is $(\rho')^*c_r(V)$, since it is defined by the vanishing of a section of $(\rho')^*V$. Consider classes in $A^*(K)$ of the form $(\rho'')^*\alpha$, where $\alpha\in A^*(\tilde{B})$. The effect of $(\sigma' \circ \iota)_*$ on such classes is
\begin{equation} \label{effect}
\sigma'_* \iota_* (\rho'')^* \alpha = \sigma'_* \iota_* \iota^* (\rho')^* \alpha = \sigma'_*([K] \cdot (\rho')^*\alpha) = \sigma'_*(\rho')^*(c_r(V) \cdot \alpha) = \rho^*\sigma_*(c_r(V) \cdot \alpha).
\end{equation}
The last step uses that flat pull back and proper push forward commute in fiber diagram.

If $\alpha = \sum_i (\sigma^*\beta_i) \cdot \alpha_i$, then the projection formula gives
\[\rho^*\sigma_*(c_r(V) \cdot \alpha) = \sum_i \rho^*(\beta_i) \cdot \rho^*(\sigma_*(c_r(V) \cdot \alpha_i)).\]
If $K$ is a vector bundle, then \emph{every} class in $A^{*}(K)$ has the form $(\rho'')^* \alpha$ for some $\alpha \in A^{*}(\widetilde{B})$. 
Thus, if $K$ is a vector bundle, the image of $(\sigma' \circ \iota)_*$ is generated over $A^*(X) \cong \rho^*A^*(B)$ by the classes $\rho^*(\sigma_*(c_r(V) \cdot \alpha_i))$, as $\alpha_i$ runs over generators for $A^*(\widetilde{B})$ as a $A^*(B)$-module.
\end{proof}

\subsection{The Hurwitz space} \label{hssec}
Given a scheme $S$, an $S$ point of the \emph{parametrized Hurwitz scheme} $\H_{k,g}^\dagger$ is the data of a finite, flat map $C \to \pp^1 \times S$, of constant degree $k$ so that the composition $C \to \pp^1 \times S \to S$ is smooth with geometrically connected fibers. (We do not impose the condition that a cover $C \to \pp^1$ be simply branched.)

The \emph{unparametrized Hurwitz stack} is the $\PGL_2$ quotient of the parametrized Hurwitz scheme. 
There is also a natural action of $\SL_2$ on $\H_{k,g}^\dagger$ (via $\SL_2 \subset \GL_2 \to \PGL_2$). The natural map $[\H_{k,g}^\dagger/\SL_2] \to [\H_{k,g}^\dagger/\PGL_2]$ is a $\mu_2$ banded gerbe.  It is a general fact that \emph{with rational coefficients}, the pullback map along a gerbe banded by a finite group is an isomorphism \cite[Section 2.3]{PV}. In particular, since we work with rational coefficients throughout, $A^*([\H_{k,g}^\dagger/\PGL_2]) \cong A^*([\H^\dagger_{k,g}/\SL_2])$. It thus suffices to prove all statements for the $\SL_2$ quotient $[\H^\dagger_{k,g}/\SL_2]$, which we denote by $\H_{k,g}$ from now on.

Explicitly, $\H_{k,g} = [\H_{k,g}^\dagger/\SL_2]$ is the stack whose objects over a scheme $S$ are families $(C\rightarrow P\rightarrow S)$ where $P = \pp V \rightarrow S$ is the projectivization of a rank $2$ vector bundle $V$ with trivial determinant, $C\rightarrow P$ is a finite flat finitely presented morphism of constant degree $k$, and the composition $C\rightarrow S$ has smooth fibers of genus $g$.
The benefit of working with $\H_{k,g}$ is that the $\SL_2$ quotient is equipped with a universal $\pp^1$-bundle $\P \to \H_{k, g}$ that has a relative degree one line bundle $\O_{\P}(1)$ (a $\pp^1$-fibration does not).
Working with this $\pp^1$-bundle simplifies our intersection theory calculations.

We shall also work with $\Hp_{k,g}^{\nf}$, the open substack of $\Hp_{k,g}$ parametrizing covers that do not factor through a lower genus curve. When $k$ is prime, $\Hp_{k,g}^{\nf}=\Hp_{k,g}$.
In Section \ref{app-sec} of the paper, we will consider the open substack $\Hp^{s}_{k,g}\subset\Hp_{k,g}$, which parametrizes covers that are simply branched. Note that $\Hp^{s}_{k,g} \subseteq \H_{k,g}^{\mathrm{nf}}$.

\section{Relative bundles of principal parts} \label{partsec}
In this section, we collect some background on bundles of principal parts, which will be used to produce relations among tautological classes in Sections \ref{RE3}, \ref{re4}, \ref{re5}, and to compute classes of certain ramification strata in Section \ref{app-sec}.
For the basics, we follow the exposition in Eisenbud-Harris \cite{EH}. 

\subsection{Basic properties} Let $b:Y\rightarrow Z$ be a smooth proper morphism. Let $\Delta_{Y/Z}\subset Y \times_Z Y$ be the relative diagonal. With $p_1$ and $p_2$ the projection maps, we obtain the following commutative diagram:
\[
\begin{tikzcd}
\Delta_{Y/Z} \arrow[rd] \arrow[rrd, bend left] \arrow[rdd, bend right] &                                      &                  \\
                                                                 & Y \times_{Z} Y \arrow[d, "p_2"'] \arrow[r, "p_1"] & Y \arrow[d, "b"] \\
                                                                 & Y \arrow[r, "b"']                    & Z.               
\end{tikzcd}
\]
\begin{definition}
Let $\W$ be a vector bundle on $Y$ and let $\mathcal{I}_{\Delta_{Y/Z}}$ denote the ideal sheaf of the diagonal in $Y \times_Z Y$. The bundle of relative $m^{\text{th}}$ order principal parts $P^m_{Y/Z}(\W)$ is defined as
\[
P^m_{Y/Z}(\W)=p_{2*}(p_1^*\W\otimes \mathcal{O}_{Y \times_Z Y}/\mathcal{I}_{\Delta_{Y/Z}}^{m+1}).
\]
\end{definition}
The following explains all the basic properties of bundles of principal parts that we need.
Parts (1) and (2) are Theorem 11.2 in \cite{EH}. Let $m\Delta_{Y/Z}$ be the closed subscheme of $Y\times_Z Y$ defined by the ideal sheaf $\mathcal{I}_{\Delta_{Y/Z}}^m$. Part (3) below follows because the restriction of $p_2$ to the thickened diagonal $m\Delta_{Y/Z} \to Y$ is finite, so the push forward is exact.
\begin{prop}\label{parts}
With notation as above,
\begin{enumerate}
    \item The quotient map $p_1^*\W\rightarrow p_1^*\W\otimes \mathcal{O}_{Y \times_Z Y}/\mathcal{I}_{\Delta_{Y/Z}}^{m+1}$ pushes forward to a map
    \[
    b^*b_*\W\cong p_{2*}p_1^*\W\rightarrow P^{m}_{Y/Z}(\W),
    \]
    which, fiber by fiber, associates to a global section $\delta$ of $\W$ a section $\delta'$ whose value at $z\in Z$ is the restriction of $\delta$ to an $m^{\text{th}}$ order neighborhood of $z$ in the fiber $b^{-1}b(z)$.
    \item \label{filtration} $P^0_{Y/Z}(\W)=\W$. For $m>1$, the filtration of the fibers $P^m_{Y/Z}(\W)_y$ by order of vanishing at $y$ gives a filtration of $P^m_{Y/Z}(\W)$ by subbundles that are kernels of the natural surjections
    $P^m_{Y/Z}(\W)\rightarrow P^k_{Y/Z}(\W)$ for $k<m$. The graded pieces of the filtration are identified by the exact sequences
    \[
    0\rightarrow \W\otimes \Sym^m(\Omega_{Y/Z})\rightarrow P^m_{Y/Z}(\W)\rightarrow P^{m-1}_{Y/Z}(\W)\rightarrow 0.
    \]
    \item A short exact sequence $0 \rightarrow K \to \W \to \W'\rightarrow 0$ of vector bundles on $Y$ induces an exact sequence of principal parts bundles 
    \[0 \to P^m_{Y/Z}(K) \to P^m_{Y/Z}(\W) \to P^m_{Y/Z}(\W') \to 0\]
\end{enumerate}
\end{prop}
We will need to know when the map from part (1) is surjective.
\begin{lem}\label{veryample}
Suppose $\W$ is a relatively very ample line bundle on $Y$. Then the evaluation map
$
b^*b_*\W\rightarrow P^1_{Y/Z}(\W)$
is surjective. 
\end{lem}
\begin{proof}
The statement can be checked fiber by fiber over $Z$. Then, it follows from the fact that very ample line bundles separate points and tangent vectors.
\end{proof}

Together with the above lemma, the following two lemmas will help us establish when evaluation maps are surjective in our particular setting.

\begin{lem} \label{line}
Let $E = \O(e_1) \oplus \cdots \oplus \O(e_r)$ be a vector bundle on $\pp^1$ with $e_1 \leq \cdots \leq e_r$ and let $\gamma: \pp E^\vee \to \pp^1$ be the projectivization. The line bundle $L =\gamma^*\O_{\pp^1}(a) \otimes  \O_{\pp E^\vee}(m)$ is very ample if and only if $m \geq 1$ and $a + me_1 \geq 1$, equivalently if and only if $h^1(\pp^1, \gamma_*L \otimes \O_{\pp^1}(-2)) = 0$.
\end{lem}
\begin{proof}
First, note that $L$ is the pullback of $\O(1)$ under a degree $m$ relative Veronese embedding $\pp E^\vee \hookrightarrow \pp(\O(a) \otimes \Sym^mE)^\vee$. 
The $\O(1)$ on the projective bundle $\pp(\O_{\pp^1}(a) \otimes \Sym^mE)^\vee$ is very ample if and only if all summands of $\O_{\pp^1}(a) \otimes \Sym^mE = \gamma_*L$ have positive degree (see  \cite[Section 9.1.1]{EH}). These summands have degrees of the form $a + e_{i_1} + \ldots + e_{i_m}$, all of which are at least $a + me_1$.
\end{proof}

\begin{lem}  \label{famjet}
Suppose $\E$ is a vector bundle on a $\pp^1$-bundle $\pi: \P \to B$ and let $\gamma: \pp \E^\vee \to \P$ be the projectivization. Suppose $\W = (\gamma^*A) \otimes \O_{\pp \E^\vee}(m)$ for some $m \geq 1$ and vector bundle $A$ on $\P$.
If $R^1\pi_*[\gamma_*\W \otimes \O_{\P}(-2)] = 0$, then the evaluation map
\[(\pi \circ \gamma)^*(\pi \circ \gamma)_* \W \to P^1_{\pp \E^\vee/B}(\W)\]
is surjective.
\end{lem}
\begin{proof}
It suffices to check surjectivity in each of the fibers over $B$, so we are reduced to the case that $B$ is a point.
Now we may assume $A$ splits as a sum of line bundles, say $A \cong \O(a_1) \oplus \cdots \oplus \O(a_r)$. By cohomology and base change, we have $h^1(\pp^1, \gamma_*\W \otimes \O_{\pp^1}(-2)) = 0$, which implies $h^1(\pp^1, \gamma_*(\gamma^*\O(a_i) \otimes \O_{\pp E^\vee}(m)) \otimes \O(-2)) = 0$ for each $i$. By Lemma \ref{line}, we have that $\W$ is a sum of very ample line bundles (over $B$). The bundle of principal parts respects direct sums, so the evaluation map is surjective by Lemma \ref{veryample}. 
\end{proof}

The following lemma should be thought of as saying ``pulled back sections have vanishing vertical derivatives."

\begin{lem}\label{partspullback}
Let $X\xrightarrow{a}Y\xrightarrow{b}Z$ be a tower of schemes with $a$ and $b$ smooth, and let $\W$ be a vector bundle on $Y$. 
For each $m$ there is a natural map $a^*P^m_{Y/Z}(\W)\rightarrow P^m_{X/Z}(a^*\W)$. This map fits in an exact sequence
\[
0\rightarrow a^*P^m_{Y/Z}(\W)\rightarrow P^m_{X/Z}(a^*\W)\rightarrow F_m\rightarrow 0,
\]
where $F_1\cong \Omega_{X/Y}\otimes a^*\W$ and $F_m$ for $m>1$ is filtered as
\[ 
0\rightarrow \Sym^{m-1}\Omega_{X/Z}\otimes \Omega_{X/Y}\otimes a^*\W\rightarrow F_m\rightarrow F_{m-1}\rightarrow 0.
\]
In particular, the evaluation map 
\[
b^*b_*\W\rightarrow P^m_{Y/Z}(\W)
\]
gives rise to a composition
\[
a^*b^*b_*\W\rightarrow a^*P^m_{Y/Z}(\W)\rightarrow P^m_{X/Z}(a^*\W),
\]
which, fiber by fiber, gives the Taylor expansion of sections of $\W$ along the ``horizontal" pulled back directions.
\end{lem}
\begin{proof}
We begin by constructing the map $a^*P^m_{X/Z}(\W)\rightarrow P^m_{Y/Z}(a^*\W)$. Consider the following commutative diagram:
\[
\begin{tikzcd}
  & X \arrow[d, "a"'] \arrow[ld] & X\times_{Z} X \arrow[l, "p_2"'] \arrow[r, "p_1"] \arrow[d, "a\times a"] & X \arrow[d, "a"] \arrow[rd] &   \\
Z & Y \arrow[l, "b"]             & Y\times_{Z} Y \arrow[r, "q_1"'] \arrow[l, "q_2"]                        & Y \arrow[r, "b"']           & Z
\end{tikzcd}
\]
Let $\Delta_{Y}\subset Y\times_{Z} Y$ denote the relative diagonal, and similarly for $\Delta_{X}\subset X\times_{Z} X$. By definition, we have
\[
a^*P^m_{Y/Z}(\W)=a^*(q_{1*}(\mathcal{O}_{Y\times_{Z} Y}/\I^{m+1}_{\Delta_Y}\otimes q_2^*\W)).
\]
The natural transformation of functors $a^*q_{1*}\rightarrow p_{1*}(a\times a)^*$ induces a map
\[
a^*P^m_{Y/Z}(\W)\rightarrow p_{1*}((a\times a)^*(\mathcal{O}_{Y\times_{Z} Y}/\I^{m+1}_{\Delta_Y})\otimes (a\times a)^*q_2^*\W)).
\]
The transform $(q_2\circ (a\times a))^*\rightarrow (a\circ p_2)^*$ induces a map
\[
p_{1*}((a\times a)^*(\mathcal{O}_{Y\times_{Z} Y}/\I^{m+1}_{\Delta_Y})\otimes (a\times a)^*q_2^*\W))\rightarrow p_{1*}((a\times a)^*(\mathcal{O}_{Y\times_{Z} Y}/\I^{m+1}_{\Delta_Y})\otimes p_2^*a^*\W).
\]
The natural morphism of sheaves $\mathcal{O}_{Y\times_{Z} Y}\rightarrow (a\times a)_*\mathcal{O}_{X\times_Z X}$ induces a map on quotients $\mathcal{O}_{Y\times_{Z} Y}/\I^{m+1}_{\Delta_Y}\rightarrow (a\times a)_*(\mathcal{O}_{X\times_Z X}/\I^{m+1}_{\Delta_X}).$
By adjunction, we obtain a map
\[
(a\times a)^*(\mathcal{O}_{Y\times_{Z} Y}/\I^{m+1}_{\Delta_Y}) \rightarrow \mathcal{O}_{X\times_Z X}/\I^{m+1}_{\Delta_X}.
\]
Then we have a morphism
\[
p_{1*}((a\times a)^*(\mathcal{O}_{Y\times_{Z} Y}/\I^{m+1}_{\Delta_Y})\otimes p_2^*a^*\W)\rightarrow p_{1*}(\mathcal{O}_{X\times_{Z} X}/\I_{\Delta_X}^{m+1}\otimes p_2^*(a^*\W))=P^m_{X/Z}(a^*\W).
\]
By construction, the maps $a^*P^m_{Y/Z}(\W)\rightarrow P^m_{X/Z}(a^*\W)$ are compatible with the filtrations on the fibers by order of vanishing, so we obtain an induced map on the graded pieces of the filtrations:
\[
\begin{tikzcd}
0 \arrow[r] & \Sym^m (a^*\Omega_{Y/Z})\otimes a^*\W \arrow[r] \arrow[hook, d] & a^*P^m_{Y/Z}(\W) \arrow[r] \arrow[d] & a^*P^{m-1}_{Y/Z}(\W) \arrow[r] \arrow[d] & 0 \\
0 \arrow[r] & \Sym^m (\Omega_{X/Z})\otimes a^*\W \arrow[r]             & P^m_{X/Z}(a^*\W) \arrow[r]           & P^{m-1}_{X/Z}(a^*\W) \arrow[r]              & 0
\end{tikzcd}
\]
When $m = 1$, the right vertical map is the identity on $a^* \W$. Hence, 
$a^*P^1_{Y/Z}(\W)\rightarrow P^1_{X/Z}(a^*\W)$ is injective. By the snake lemma, the cokernel is isomorphic to the cokernel of the left vertical map, which in turn is
$\Omega_{X/Y}\otimes a^*\W$ because $a$ and $b$ are smooth and $\W$ is locally free.
For $m > 1$, we may assume by induction that
the right vertical map is injective, hence the center vertical map is injective.
The filtration of the cokernel $F_m$ of the center vertical map follows by induction and the snake lemma.
\end{proof}

\subsection{Directional refinements} \label{dir-ref}
Much of the exposition in this subsection is based on unpublished notes of Ravi Vakil. Suppose we have a tower $X \xrightarrow{a} Y \xrightarrow{b} Z$ and $a^* \Omega_{Y/Z}$ admits a filtration on $X$
\begin{equation} \label{basicfilt}
0 \rightarrow \Omega_y \rightarrow a^* \Omega_{Y/Z} \rightarrow \Omega_x \rightarrow 0.
\end{equation}
For example, take $X=\p(\Omega_{Y/Z})$ or $G(n, \Omega_{Y/Z})$ with the filtration given by the tautological sequence.
First, suppose $\Omega_x$ and $\Omega_y$ are rank $1$.
The filtration \eqref{basicfilt} is the same as saying we can choose local coordinates $x, y$ at each point of $Y$ where $y$ is well-defined up to $(x, y)^2$, and $x$ is only defined modulo $y$. The vanishing of $y$ defines a distinguished ``$x$-direction" on the tangent space $T_{Y/Z}$ at each point, which is dual to the quotient $a^*\Omega_{Y/Z} \to \Omega_x$.

The goal of this section is to define principal parts bundles that measure certain parts of a Taylor expansion with respect to these local coordinates. These principal parts bundles will be indexed by \textit{admissible sets} $S$ of monomials in $x$ and $y$ (defined below). 
If $x^iy^j \in S$, then $P^S_{Y/Z}(\W)$ will keep track of the coefficient of $x^iy^j$ in the Taylor expansion of a section of $\W$. For example,
 $S = \{1, x\}$ will correspond to a quotient of $a^*P^1_{Y/Z}(\W)$ that measures only derivatives in the $x$-direction. The set
$S = \{1, x, y, x^2, xy, y^2\}$ corresponds to the pullback of the usual second order principal parts. It is helpful to visualize these sets with diagrams as below, where we place a dot at coordinate $(i, j)$ if $x^i y^j \in S$.
\begin{center}
\begin{tikzpicture}[scale = .6]
\draw[->] (0, 0) -- (1, 0);
\node at (1.2, 0) {$i$}; 
\draw[->] (0, 0) -- (0, -1);
\node at (0, -1.3) {$j$};
\filldraw[color=white] (0, -2) circle (5pt);
\end{tikzpicture}
\hspace{1.5in}
\begin{tikzpicture}[scale = .6]
 \filldraw (0, 0) circle (5pt);
\filldraw (1, 0) circle (5pt);
\node at (0.5, 1) {$\{1, x\}$};
\filldraw[color=white] (0, -2) circle (5pt);
\end{tikzpicture}
\hspace{1.5in}
\begin{tikzpicture}[scale = .6]
 \filldraw (0, 0) circle (5pt);
\filldraw (0, -1) circle (5pt);
 \filldraw (1, 0) circle (5pt);
\filldraw (0, -2) circle (5pt);
 \filldraw (2, 0) circle (5pt);
\filldraw (1, -1) circle (5pt);
\node at (1.3, 1) {$\{1,x, y, x^2, xy, y^2\}$};
\end{tikzpicture}
\end{center}

More generally, if $\Omega_x$ and $\Omega_y$ have any ranks, the quotient $\Omega_x$ is dual to a distinguished subspace of $T_{Y/Z}$.
The construction below will build bundles $P^S_{Y/Z}(\W)$ such that if $x^i y^j \in S$, then $P^S_{Y/Z}(\W)$ tracks the coefficients of all monomials corresponding to $\Sym^i \Omega_x \otimes \Sym^j \Omega_y$. In other words, $P^S_{Y/Z}(\W)$ will admit a filtration with quotients $\Sym^i \Omega_x \otimes \Sym^j \Omega_y \otimes \W$ for each $(i, j)$ such that $x^iy^j \in S$. Each dot in the diagram corresponds to a piece of this filtration.
Only diagrams of certain shapes are allowed.

\begin{definition}
A set $S$ is \textit{admissible} if the following hold
\begin{itemize}
    \item If $x^i y^j \in S$, then $x^{i-1} y^{j} \in S$ (if $i - 1 \geq 0$). That is, for each dot in the diagram, the dot to its left is also in the diagram if possible.
    \item If $x^i y^j \in S$, then $x^{i-2} y^{j+1}$ (if $i - 2 \geq 0$). That is, for each dot in the diagram, the dot two to the left and one down is also in the diagram if possible.
\end{itemize}
Equivalently, the diagram associated to $S$ is built, via intersections and unions, from triangular collections of lattice points bounded by the axes and a line of slope $1$ or slope $\frac{1}{2}$.
\end{definition}

To build the principal parts bundles $P^S_{Y/Z}(W)$, let us consider the diagram
\begin{center}
\begin{tikzcd}
\widetilde{\Delta} \arrow{d}{\iota}  \arrow{r} &\Delta \arrow{d} \\
X \times_Z Y \arrow{d}[swap]{\tilde{p}_2} \arrow{r}{\tilde{a}} & Y \times_Z Y \arrow{d}[swap]{p_2}\arrow{r}{p_1} & Y \arrow{d}{b} \\
X \arrow{r}{a} &Y \arrow{r}{b} & Z, 
\end{tikzcd}
\end{center}
where $\Delta = \Delta_{Y/Z} \subset Y \times_Z Y$ is the diagonal and all squares are fibered squares. The composition of vertical maps give isomorphisms $\Delta \cong Y$ and $\widetilde{\Delta} \cong X$.
There is an identification $\iota_* \Omega_{\widetilde{\Delta}/Z} \cong \mathcal{I}_{\widetilde{\Delta}}/\mathcal{I}_{\widetilde{\Delta}}^2$.
Using \eqref{basicfilt} and the isomorphism $\widetilde{\Delta} \cong X$, we obtain an injection
\[\iota_* \Omega_y \to \iota_*a^* \Omega_{Y/Z} \to \iota_*\Omega_{X/Z} \cong \I_{\widetilde{\Delta}}/\I_{\widetilde{\Delta}}^2, \]
which determines a subsheaf $\mathcal{J} \subset \I_{\widetilde{\Delta}} =: \I$.
The sheaf $\I$ corresponds to the monomials $\{x^iy^j : i + j \geq 1\}$ (see \eqref{Ipow} below). The subsheaf $\mathcal{J}$ corresponds to the monomials $\{x^iy^j : i + j \geq 2 \text{ or } j \geq 1\}$ (see \eqref{Jpow} below). The condition $i + j \geq 2$ says $\I^2 \subset \mathcal{J}$. The condition $j \geq 1$ says $\mathcal{J} \subset \mathcal{I}$ and it ``picks out our $y$-coordinate(s) to first order."

In the next paragraph, we will explain how to construct an ideal $\I_S$, via intersections and unions of $\I$ and $\J$, corresponding to monomials not in $S$.
Our refined principal parts bundles will then be defined as
\[P^S_{Y/Z}(\W) := \tilde{p}_{2*}\left( \tilde{a}^* p_1^*  \W \otimes \O_{X \times_Z Y}/\mathcal{I}_S \right), \]
The bundle $P^S_{Y/Z}(\W)$ is defined on $X$ and will be a quotient of $a^*P^m_{Y/Z}(\W)$ for $m = \max\{i + j : x^iy^j \in S\}$. In particular, there are restricted evaluation maps
\begin{equation} \label{res-ref} a^*b^*b_* \W \to a^*P^m_{Y/Z}(\W) \to P^S_{Y/Z}(\W),
\end{equation}
which we think of as Taylor expansions only along certain directions specified by $S$.

To start, we shall have $\I_{\{1, x, y\}} := \I$ and $\I_{\{1, x\}} := \J$. Powers of these ideals correspond to regions below lines of slope $1$ and $\frac{1}{2}$ respectively. 
\begin{equation} \label{Ipow}
\begin{tikzpicture}[scale = .6]
 \draw[dotted] (0, 0) circle (5pt);
\filldraw (0, -1) circle (5pt);
 \filldraw (1, 0) circle (5pt);
\filldraw (0, -2) circle (5pt);
 \filldraw (2, 0) circle (5pt);
\filldraw (1, -1) circle (5pt);
 \filldraw (3, 0) circle (5pt);
\filldraw (2, -1) circle (5pt);
 \filldraw (1, -2) circle (5pt);
\filldraw (0, -3) circle (5pt);
\filldraw (2, -2) circle (5pt);
\filldraw (4, 0) circle (5pt);
\filldraw (1, -3) circle (5pt);
\filldraw (3, -1) circle (5pt);
\filldraw (0, -4) circle (5pt);
\node at (3, -3) {$\ddots$};
\node at (1.5, 1) {$\I$};
\end{tikzpicture}
\hspace{.8in}
\begin{tikzpicture}[scale = .6]
 \draw[dotted] (0, 0) circle (5pt);
\draw[dotted] (0, -1) circle (5pt);
 \draw[dotted] (1, 0) circle (5pt);
\filldraw (0, -2) circle (5pt);
 \filldraw (2, 0) circle (5pt);
\filldraw (1, -1) circle (5pt);
 \filldraw (3, 0) circle (5pt);
\filldraw (2, -1) circle (5pt);
 \filldraw (1, -2) circle (5pt);
\filldraw (0, -3) circle (5pt);
\filldraw (2, -2) circle (5pt);
\filldraw (4, 0) circle (5pt);
\filldraw (1, -3) circle (5pt);
\filldraw (3, -1) circle (5pt);
\filldraw (0, -4) circle (5pt);
\node at (3, -3) {$\ddots$};
\node at (1.5, 1) {$\I^2$};
\end{tikzpicture}
\hspace{.8in}
\begin{tikzpicture}[scale = .6]
\node at (1.5, 1) {$\I^3$};
 \draw[dotted] (0, 0) circle (5pt);
\draw[dotted] (0, -1) circle (5pt);
 \draw[dotted] (1, 0) circle (5pt);
\draw[dotted] (0, -2) circle (5pt);
 \draw[dotted] (2, 0) circle (5pt);
\draw[dotted] (1, -1) circle (5pt);
 \filldraw (3, 0) circle (5pt);
\filldraw (2, -1) circle (5pt);
 \filldraw (1, -2) circle (5pt);
\filldraw (0, -3) circle (5pt);
\filldraw (2, -2) circle (5pt);
\filldraw (4, 0) circle (5pt);
\filldraw (1, -3) circle (5pt);
\filldraw (3, -1) circle (5pt);
\filldraw (0, -4) circle (5pt);
\node at (3, -3) {$\ddots$};
\end{tikzpicture}
\end{equation}

\begin{equation} \label{Jpow}
\begin{tikzpicture}[scale = .6]
 \draw[dotted] (0, 0) circle (5pt);
\draw[dotted] (1, 0) circle (5pt);
 \filldraw (0, -1) circle (5pt);
\filldraw (0, -2) circle (5pt);
 \filldraw (2, 0) circle (5pt);
\filldraw (1, -1) circle (5pt);
 \filldraw (3, 0) circle (5pt);
\filldraw (2, -1) circle (5pt);
 \filldraw (1, -2) circle (5pt);
\filldraw (0, -3) circle (5pt);
\filldraw (2, -2) circle (5pt);
\filldraw (4, 0) circle (5pt);
\filldraw (1, -3) circle (5pt);
\filldraw (3, -1) circle (5pt);
\filldraw (0, -4) circle (5pt);
\node at (3, -3) {$\ddots$};
\node at (1.5, 1) {$\J$};
\end{tikzpicture}
\hspace{.8in}
\begin{tikzpicture}[scale = .6]
 \draw[dotted] (0, 0) circle (5pt);
\draw[dotted] (0, -1) circle (5pt);
 \draw[dotted] (1, 0) circle (5pt);
\filldraw (0, -2) circle (5pt);
 \draw[dotted] (2, 0) circle (5pt);
\filldraw (1, -1) circle (5pt);
 \filldraw (3, 0) circle (5pt);
\filldraw (2, -1) circle (5pt);
 \filldraw (1, -2) circle (5pt);
\filldraw (0, -3) circle (5pt);
\filldraw (2, -2) circle (5pt);
\filldraw (4, 0) circle (5pt);
\filldraw (1, -3) circle (5pt);
\filldraw (3, -1) circle (5pt);
\filldraw (0, -4) circle (5pt);
\node at (3, -3) {$\ddots$};
\node at (1.5, 1) {$\J^2$};
\end{tikzpicture}
\hspace{.8in}
\begin{tikzpicture}[scale = .6]
\node at (1.5, 1) {$\J^3$};
 \draw[dotted] (0, 0) circle (5pt);
\draw[dotted] (0, -1) circle (5pt);
 \draw[dotted] (1, 0) circle (5pt);
\filldraw (0, -2) circle (5pt);
\filldraw (4, 0) circle (5pt);
\filldraw (3, -1) circle (5pt);
 \draw[dotted] (2, 0) circle (5pt);
\draw[dotted] (1, -1) circle (5pt);
 \draw[dotted] (3, 0) circle (5pt);
\filldraw (2, -1) circle (5pt);
 \filldraw (1, -2) circle (5pt);
\filldraw (0, -3) circle (5pt);
\filldraw (2, -2) circle (5pt);
\filldraw (2, -2) circle (5pt);
\filldraw (4, 0) circle (5pt);
\filldraw (1, -3) circle (5pt);
\filldraw (0, -4) circle (5pt);
\node at (3, -3) {$\ddots$};
\end{tikzpicture}
\end{equation}

To say that $S$ is admissible is to say that $\I_S$ is built by taking unions and intersections such half planes, which corresponds to intersections and unions of $\I$ and $\J$. We list below the principal parts bundles we require in the remainder of the paper and their associated ideal $\I_S$.
\begin{enumerate}
    \item $S = \{1, x\}$ with $\I_S = \J$, which we call the bundle of \emph{restricted principal parts}.
    \item $S = \{1, x, y, x^2\}$ with $\I_S = \J^2$ will arise in triple point calculations.
    \item $S = \{1, x, y, x^2, xy\}$ with $\I_S = \I^3 + \J^3$ arises when finding quadruple points in a pencil of conics.
    \item $S = \{1, x, y, x^2, xy, x^3\}$ with $\I_S = \J^3$ will arise in finding quadruple points in pentagonal covers.
\end{enumerate}
Diagrams corresponding to these sets appear at the end of the next subsection. Given two admissible sets $S \subset S'$, there is a natural surjection $P^{S'}_{Y/Z}(\W) \to P^{S}_{Y/Z}(\W)$, which corresponds to truncating Taylor series. This determines the order(s) that the terms $\Sym^i \Omega_x \otimes \Sym^j \Omega_y \otimes \W$ corresponding to $x^i y^j \in S'$ may appear as quotients in a filtration: a term corresponding to $x^i y^j \in S'$ is a well-defined subbundle of $P^{S'}_{Y/Z}(\W)$ if $S' \smallsetminus x^i y^j$ is an admissible set.

\subsection{Bundle-induced refinements}
Now suppose that $a^* \W$ admits a filtration on $X$ by
\begin{equation} \label{otherfilt} 0 \rightarrow K \rightarrow a^*\W \rightarrow \W' \rightarrow 0,
\end{equation}
where $\W'$ is a vector bundle, and hence so is $K$.
Exactness of principal parts for $X$ over $Z$ gives an exact sequence
\[0 \rightarrow P^m_{X/Z}(K) \rightarrow P^m_{X/Z}(a^*\W) \rightarrow P^m_{X/Z}(\W') \rightarrow 0.\]
We are interested in the restriction of this filtration to $a^* P^m_{Y/Z}(\W) \subset P^m_{X/Z}(a^*\W)$. 
First, we need the following fact.

\begin{lem}
The intersection of the two subbundles
\begin{equation} \label{twosubs} P^m_{X/Z}(K) \subset P^m_{X/Z}(a^*\W) \qquad \text{and} \qquad a^* P^m_{Y/Z}(\W) \subset P^m_{X/Z}(a^*\W)
\end{equation}
is a subbundle. 
\end{lem}
\begin{proof}
We proceed by induction. For $m=0$, the claim is just that $K$ is a subbundle of $a^*\W$. The question is local, so we can assume that the vanishing order filtration exact sequences
\[
0\rightarrow \Sym^m \Omega_{X/Z}\otimes a^*\W\rightarrow P^m_{X/Z}(a^*\W)\rightarrow P^{m-1}_{X/Z}(a^*\W)\rightarrow 0,
\]
are split. By induction and the (locally split) exact sequences,
\[
0\rightarrow \Sym^m\Omega_{X/Z}\otimes K\rightarrow P^m_{X/Z}(K)\rightarrow P^{m-1}_{X/Z}(K)\rightarrow 0
\]
and
\[
0\rightarrow a^*\Sym^m \Omega_{Y/Z}\otimes a^*\W\rightarrow a^*P^m_{Y/Z}(\W)\rightarrow a^*P^{m-1}_{Y/Z}(\W)\rightarrow 0
\]
it suffices to show that the intersection of $\Sym^m\Omega_{X/Z}\otimes K$ and $a^*\Sym^m\Omega_{Y/Z}\otimes a^*\W$ is a subbundle of $\Sym^m\Omega_{X/Z}\otimes a^*\W$. But this intersection is given by $a^*\Sym^m\Omega_{Y/Z}\otimes K$, which is a subbundle.
\end{proof}

\begin{definition} \label{Qdef}
We define $\underline{P}^m_{Y/Z}(K)$ to be the intersection of the two subbundles in \eqref{twosubs}. This subbundle tracks principal parts of $K$ in the directions of $Y/Z$. We include the underline to remind ourselves that this bundle is defined on $X$ since $K$ is defined on $X$. We define $Q^m_{Y/Z}(\W')$ to be the cokernel of $\underline{P}^m_{Y/Z}(K) \to a^* P^m_{Y/Z}(\W)$.
\end{definition}
When $K=a^*K'$ for a bundle $K'$ on $Y$, the bundle $\underline{P}^m_{Y/Z}(K)$ is just the bundle $a^*P^m_{Y/Z}(K')$.
\par The vanishing order filtrations from Proposition \ref{parts} of $P^m_{X/Z}(K)$ and $a^*P^m_{Y/Z}(\W)$ restrict to a vanishing order filtration on $\underline{P}^m_{Y/Z}(K)$, which in turn induces a vanishing order filtration on $Q^m_{Y/Z}(\W')$. We describe this for $m = 1$ below for future use.

\begin{lem} \label{Qequip}
The bundle $Q^1_{Y/Z}(\W')$ is equipped with a surjection $a^*P^1_{Y/Z}( \W) \to Q^1_{Y/Z}(\W')$ and a filtration
\[0 \rightarrow a^*\Omega_{Y/Z} \otimes \W' \to Q^1_{Y/Z}(\W') \rightarrow \W' \rightarrow 0\]
A section $\O_Y \xrightarrow{\delta} \W$ on $X$ induces a section $\O_X \xrightarrow{\delta'} a^*P^1_{Y/Z}(\W) \to Q^1_{Y/Z}(\W')$ that records the values and ``horizontal derivatives" of $\delta$ in the quotient $\W'$.
\end{lem}
\subsection{Directional and bundle-induced refinements} The principal parts bundles constructed in this subsection will not be needed until Section \ref{app-sec}.
Here, we suppose that we have filtrations as in \eqref{basicfilt} and \eqref{otherfilt}.
We have an inclusion $\underline{P}^m_{Y/Z}(K) \hookrightarrow a^*P^m_{Y/Z}(\W)$ as well as a quotient $a^*P^m_{Y/Z}(\W) \to P^S_{Y/Z}(\W)$. We define $\underline{P}^S_{Y/Z}(K)$ to be image of the composition \[\underline{P}^m_{Y/Z}(K) \hookrightarrow a^*P^m_{Y/Z}(\W) \to P^S_{Y/Z}(\W),\]
which tracks the principal parts of $K$ in the $Y/Z$ directions specified by $S$.

Given two admissible sets $S \subset S'$, there is a quotient $\underline{P}^{S'}_{Y/Z}(K) \to \underline{P}^{S}_{Y/Z}(K)$. Let $V \subset \underline{P}^{S'}_{Y/Z}(K)$ be the kernel.
We define $P^{S \subset S'}_{Y/Z}(\W \to \W')$ to be the cokernel of the composition
\[V \hookrightarrow \underline{P}^{S'}_{Y/Z}(K) \hookrightarrow P^{S'}_{Y/Z}(\W).\]
The bundle $P^{S \subset S'}_{Y/Z}(\W \to \W')$ tracks the principal parts associated to $S$ on $\W$ and then the principal parts associated to the rest of $S'$ but just in the $\W'$ quotient.
We visualize $P^{S \subset S'}_{Y/Z}(\W \to \W')$ by a decorated diagram of shape $S'$, where the dots are filled in the subshape $S$ and half filled (representing values just in $\W'$) in the remainder $S' \smallsetminus S$ (colored in blue below).
A preview of the cases we shall need later are pictured below.
\begin{enumerate}
   \item[(6.4A)] \label{6T} $S = \{1, x\}$ and $S' = \{1, x, y, x^2\}$, for triple points in Section \ref{formulas5}.
   \begin{center}
   \begin{tikzpicture}[scale = .6]
     \draw[white] (3, 0) circle (5pt);
 \filldraw (0, 0) circle (5pt);
\draw (0, -1) circle (5pt);
 \filldraw (1, 0) circle (5pt);
 \draw (2, 0) circle (5pt);
 \begin{scope}[xshift=2cm]
 \fill[blue] (0,0) -- (90:1ex) arc (90:270:1ex) -- cycle;
 \end{scope}
  \begin{scope}[yshift=-1cm]
 \fill[blue] (0,0) -- (90:1ex) arc (90:270:1ex) -- cycle;
 \end{scope}
\end{tikzpicture}
   \end{center}
   \item[(6.4B)] \label{U4} $S = \{1, x\}$ and $S' = \{1, x, y, x^2, xy\}$, for quadruple points in Lemma \ref{4quad}.
      \begin{center}
   \begin{tikzpicture}[scale = .6]
     \draw[white] (3, 0) circle (5pt);
 \filldraw (0, 0) circle (5pt);
\draw (0, -1) circle (5pt);
 \filldraw (1, 0) circle (5pt);
 \draw (2, 0) circle (5pt);
 \begin{scope}[xshift=2cm]
 \fill[blue] (0,0) -- (90:1ex) arc (90:270:1ex) -- cycle;
 \end{scope}
  \begin{scope}[yshift=-1cm]
 \fill[blue] (0,0) -- (90:1ex) arc (90:270:1ex) -- cycle;
 \end{scope}
   \begin{scope}[yshift=-1cm]
 \fill[blue] (0,0) -- (90:1ex) arc (90:270:1ex) -- cycle;
 \end{scope}
 \draw (1, -1) circle (5pt);
    \begin{scope}[yshift=-1cm, xshift = 1cm]
 \fill[blue] (0,0) -- (90:1ex) arc (90:270:1ex) -- cycle;
 \end{scope}
\end{tikzpicture}
   \end{center}
 \item[(6.4C)] \label{U5} $S = \{1, x\}$ and $S' = \{1, x, y, x^2, xy, x^3\}$, for quadruple points in Lemma \ref{classU5}.
 \begin{center}
    \begin{tikzpicture}[scale = .6]
 \filldraw (0, 0) circle (5pt);
\draw (0, -1) circle (5pt);
 \filldraw (1, 0) circle (5pt);
 \draw (2, 0) circle (5pt);
 \begin{scope}[xshift=2cm]
 \fill[blue] (0,0) -- (90:1ex) arc (90:270:1ex) -- cycle;
 \end{scope}
  \begin{scope}[yshift=-1cm]
 \fill[blue] (0,0) -- (90:1ex) arc (90:270:1ex) -- cycle;
 \end{scope}
   \begin{scope}[yshift=-1cm]
 \fill[blue] (0,0) -- (90:1ex) arc (90:270:1ex) -- cycle;
 \end{scope}
 \draw (1, -1) circle (5pt);
    \begin{scope}[yshift=-1cm, xshift = 1cm]
 \fill[blue] (0,0) -- (90:1ex) arc (90:270:1ex) -- cycle;
 \end{scope}
  \draw (3, 0) circle (5pt);
    \begin{scope}[xshift = 3cm]
 \fill[blue] (0,0) -- (90:1ex) arc (90:270:1ex) -- cycle;
 \end{scope}
 \end{tikzpicture}
 \end{center}
\end{enumerate}
Revisiting Definition \ref{Qdef},  $Q^1_{Y/Z}(\W') = P^{\varnothing \subset \{1, x, y\}}(\W \to \W')$ would be represented by
\begin{center}
\begin{tikzpicture}[scale = .6]
   \draw[white] (-1.6, 0) circle (5pt);
 \draw (0, 0) circle (5pt);
\draw (0, -1) circle (5pt);
 \draw (1, 0) circle (5pt); \begin{scope}[xshift=1cm]
 \fill[blue] (0,0) -- (90:1ex) arc (90:270:1ex) -- cycle;
 \end{scope}
  \begin{scope}[yshift=-1cm]
 \fill[blue] (0,0) -- (90:1ex) arc (90:270:1ex) -- cycle;
 \end{scope}
 \fill[blue] (0,0) -- (90:1ex) arc (90:270:1ex) -- cycle;
\draw[white] (3, 0) circle (5pt);
\end{tikzpicture}
\end{center}

\section{The Chow ring in degree $3$} \label{RE3}

\subsection{Set up}
We begin by recalling the linear algebraic data associated to a degree $3$ cover as developed by Miranda and Casnsati--Ekedahl \cite{M,CE}. For more details in our context, see \cite{BV} and \cite[Section 3.1]{part1}. Given a degree $3$, genus $g$ cover, $\alpha: C \to \pp^1$, define $E_\alpha := (\alpha_*\O_C/\O_{\pp^1})^\vee$, which is a rank $2$, degree $g+2$ vector bundle on $\pp^1$. There is a natural embedding
$C \subset \pp E_{\alpha}^\vee$ and $C$ is the zero locus of a section of 
\[H^0(\pp E^{\vee}_\alpha, \gamma^*\det E_\alpha^\vee \otimes \O_{\pp E_\alpha^{\vee}}(3)) \cong H^0(\pp^1, \det E_\alpha^\vee \otimes \Sym^3 E_\alpha),\]
where $\gamma: \pp E^{\vee}_\alpha \to \pp^1$ is the structure map.
Conversely, given a globally generated, rank $2$, degree $g + 2$ vector bundle $E$ on $\pp^1$ with $\Sym^3 E \otimes \det E^\vee$ globally generated, the vanishing of a general section $\delta \in H^0(\pp E^{\vee}, \gamma^* \det E^\vee \otimes \O_{\pp E^{\vee}}(3))$ defines a smooth, genus $g$ triple cover $\alpha: C = V(\delta) \subset \pp E^{\vee} \to \pp^1$ such that $E_\alpha \cong E$. First, let us give a characterization of which sections do not yield covers parametrized by $\H_{3,g}$.

\begin{lem}\label{3limits}
Let $E$ be a rank $2$, degree $g+2$ vector bundle on $\p^1$ such that $\Sym^3 E\otimes \det E^{\vee}$ is globally generated. Let $\delta \in H^0(\pp E, \gamma^* \det E^\vee \otimes \O_{\pp E_\alpha}(3))$.
Suppose that the zero locus $C = V(\delta) \subseteq \p E^{\vee}$ is not a smooth, irreducible genus $g$ triple cover of $\p^1$. Then there exists a point $p\in C$ such that $\dim T_p C =2$. 
\end{lem}
\begin{proof}
If $\delta = 0$, then $C$ is $2$-dimensional and the claim follows. Now suppose $\delta \neq 0$.
We will show that $C$ is connected, which implies that if $C$ fails to be an irreducible triple cover, it must have a point with $2$ dimensional tangent space.
If $\Sym^3 E \otimes \det E^\vee$ is globally generated, then both summands of $E = \O(e_1) \oplus \O(e_2)$ have degree at least $\frac{g+2}{3}$.
Hence,
$h^0(\pp^1, E^\vee) = 0$. If $C \to \pp^1$ is finite we have $h^0(C, \O_C) = h^0(\pp^1, \alpha_* \O_C) = h^0(\pp^1, \O_{\pp^1}) + h^0(\pp^1, E_\alpha^\vee) = 1$, so $C$ is connected. Now suppose $C$ has a positive dimensional fiber over $\pp^1$. 
Any curve in the class $\O_{\pp E^\vee}(3) \otimes \gamma^* \det E^\vee$ has a component that meets every fiber, thus $C$ is again connected.
\begin{center}
\includegraphics[width=2.2in]{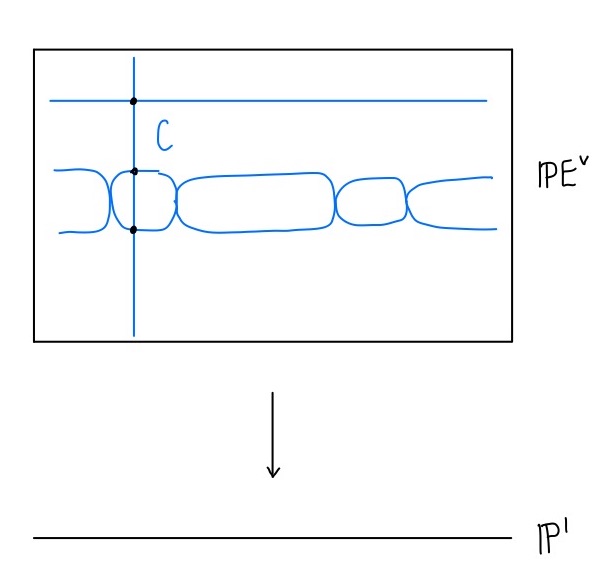}
\end{center}
\end{proof}

We now recall some notation and constructions from \cite{part1}.
The association of $\alpha:C \to \pp^1$ with $E_\alpha$ gives rise to a map of $\H_{3,g}$ to the moduli stack $\B_{3,g}$ of rank $2$, degree $g + 3$, globally generated vector bundles on $\pp^1$-bundles. Let $\pi: \P \to \B_{3,g}$ be the universal $\pp^1$-bundle and let $\E$ be the universal rank $2$ vector bundle on $\P$.
Continuing the notation of \cite{part1}, let $z = c_1(\O_{\P}(1))$ and define 
classes $a_i \in A^i(\B_{3,g})$ and $a_i' \in A^{i-1}(\B_{3,g})$ by the formula
\[c_i(\E) = \pi^*a_i + \pi^*a_i'z.\]
We also define $c_2 = -\pi_*(z^3) \in A^2(\B_{3,g})$, which is the pullback of the universal second Chern class on $\BSL_2$. By \cite[Theorem 4.4]{part1}, 
\begin{equation} \label{b3fact}
\text{$a_1, a_2, a_2', c_2$ generate $A^*(\B_{3,g})$ and satisfy no relations in degrees up to $g + 2$}.
\end{equation}

Now, let $\gamma: \pp \E^\vee \to \P$ and define $\mathcal{W} := \O_{\pp \E^\vee}(3) \otimes \gamma^* \det \E^\vee$, which is a line bundle on $\pp \E^\vee$.
Consider the bundle $\U_{3,g} := \gamma_* \mathcal{W} =  \Sym^3 \E \otimes \det \E^\vee$. 
We define 
\begin{equation} \label{bpdef} \B_{3,g}' := \B_{3,g} \smallsetminus \Supp R^1 \pi_* \U_{3,g}(-1).
\end{equation}
Equivalently, $\B_{3,g}'$ is the open locus where $\U_{3,g}$ is globally generated on the fibers of $\pi$.
By the theorem on cohomology and base change
\[\mathcal{X}'_{3,g}:= \pi_* \U_{3,g}|_{\B_{3,g}'} \]
is a vector bundle.
In \cite[Lemma 5.1]{part1}, we showed that the map $\H_{3,g} \to \B_{3,g}$ factors through an open embedding $\H_{3,g} \to \mathcal{X}'_{3,g}$.
Hence, the Chow ring of $\H_{3,g}$ is generated by the pullbacks of the classes $a_1,a_2',a_2,c_2$ from $\mathcal{B}'_{3,g}$.
We must determine the relations 
among these classes that come from excising $\Supp R^1 \pi_* \U_{3,g}(-1)$ from $\B_{3,g}$ and 
from excising
\[\Delta_{3,g}:= \mathcal{X}'_{3,g} \smallsetminus \H_{3,g}.\]
In other words, we shall compute the Chow ring $A^*(\H_{3,g})$ by computing the image of the left-hand map in the excision sequence
\[
A^{*-1}(\Delta_{3,g})\rightarrow A^*(\aW_{3,g}')\rightarrow A^*(\H_{3,g})\rightarrow 0.
\]

\subsection{Resolution and excision}

We begin by constructing a space $\tilde{\Delta}_{3,g}$, which corresponds to triple covers (or worse) with a marked singular point.
By forgetting the marked point, we will obtain a proper surjective morphism $\tilde{\Delta}_{3,g}\rightarrow \Delta_{3,g}$ by Lemma \ref{3limits}. Because our Chow rings are taken with rational coefficients, pushforward induces a surjection on Chow groups $A^*(\tilde{\Delta}_{3,g})\rightarrow A^*(\Delta_{3,g})$.  Thus, it will suffice to describe the image of $A^{*-1}(\widetilde{\Delta}_{3,g}) \to A^*(\mathcal{X}_{3,g}')$.

To build $\tilde{\Delta}_{3,g}$, we use the machinery of bundles of relative principal parts.
 By Proposition \ref{parts} part (1), there is an evaluation map
\begin{equation} \label{eval3}
\gamma^*\pi^*\aW_{3,g}' =(\pi \circ \gamma)^*(\pi \circ \gamma)_*\W \rightarrow P^1_{\p\E^{\vee}/\B'_{3,g}}(\W).
\end{equation}
A geometric point of $\gamma^*\pi^* \mathcal{X}_{3,g}'$ is the data of $(E, \delta, p)$ where $E$ is a geometric point of $\B_{3,g}'$, $\delta$ a section of $\O_{\pp E}(3) \otimes \gamma^* \det E^\vee$, and $p$ a point of $\pp E^\vee$.
Such a point lies in the kernel of the evaluation map \eqref{eval3} precisely when $\delta(p) = 0$ and the first order derivatives of $\delta$ also vanish at $p$, which is to say $V(\delta) \subset \pp E^\vee$ has $2$-dimensional tangent space at $p$. A similar description works in arbitrary families.
We define $\tilde{\Delta}_{3,g}$ to be the preimage of the zero section of \eqref{eval3} so we obtain a ``trapezoid" diagram:
\begin{equation}\label{3diagram}
\begin{tikzcd}
{\tilde{\Delta}_{3,g}} \arrow[r, "i", hook] \arrow[rd, "\rho''"'] & \gamma^*\pi^*\aW_{3,g}' \arrow[r, "\gamma'"] \arrow[d, "\rho'"] & \pi^*\aW_{3,g}' \arrow[r, "\pi'"] \arrow[d] & \aW_{3,g}' \arrow[d, "\rho"] \\
                                                                  & \p\E^{\vee} \arrow[r, "\gamma"']                       & \P \arrow[r, "\pi"']               & {\B_{3,g}'.} 
\end{tikzcd}
\end{equation}

We can thus determine information about the Chow ring of $\H_{3,g} = \aW_{3,g}' \smallsetminus (\pi' \circ \gamma' \circ i)(\widetilde{\Delta}_{3,g})$ using the Trapezoid Lemma \ref{trap}.
\begin{lem} \label{trig-result}
The rational Chow ring of $\H_{3,g}$ is a quotient of $\qq[a_1]/(a_1^3)$. Moreover, 
\begin{enumerate}
    \item For all $g \geq 3$, we have $A^1(\H_{3,g}) = \qq a_1$.
    \item For all $g \geq 6$, we have $A^2(\H_{3,g}) = \qq a_1^2$.
\end{enumerate}
\end{lem}
\begin{proof}
Let $z = c_1(\O_{\P}(1))$ and $\zeta = c_1(\O_{\pp \E^\vee}(1))$, so $z^i \zeta^j$ for $0 \leq i, j \leq 1$ form a basis for $A^*(\pp \E^\vee)$ as a $A^*(\B_{3,g}')$ module.
Let $I$ be the ideal generated by
$(\pi \circ \gamma)_*(c_3(P^1_{\pp \E^\vee/\B_{3,g}'}(\W)) \cdot z^i\zeta^j)$ for $0 \leq i, j \leq 1$.
We compute expressions for these push forwards in terms of $a_1, a_2, a_2', c_2$, and we find $\qq[a_1, a_2', a_2, c_2]/I \cong \qq[a_1]/(a_1^3)$. The code to do the above computations is provided at \cite{github}. For example, when $i = j = 0$, because $\tilde{\Delta}_{3,g} \to \Delta_{3,g}$ is generically one-to-one, this allows us to find
\begin{equation} \label{3codim1} [\Delta_{3,g}] = \pi'_* \gamma'_* [\widetilde{\Delta}_{3,g}] = \rho^*(\pi \circ \gamma)_*(c_3(P^1_{\pp \E^\vee/\B_{3,g}'}(\W)) = (8g + 12)a_1 - 9a_2'.
\end{equation}
By the Trapezoid Lemma \ref{trap}, we have that $A^*(\H_{3,g})$ is a quotient of $A^*(\B_{3,g}')/I$. 
Since $A^*(\B_{3,g}')$ is a quotient of $\qq[a_1, a_2, a_2', c_2]$, it follows that $A^*(\H_{3,g})$ is a quotient of $\qq[a_1]/(a_1^3)$.

First, note that the complement of $\B_{3,g}'$ inside $\B_{3,g}$ is the union of splitting loci where $E = \O(e_1) \oplus \O(e_2)$ for $3e_1 < g+2$ (see \cite[Section 4.2]{part1} for a review of splitting loci in our context). The codimension of the $(e_1, e_2)$ splitting locus with $e_1 \leq e_2$ is $\max\{0, e_2 - e_1 - 1\}$.
Using this, one readily checks that the complement of $\B_{3,g}'$ has codimension at least $2$ for $g \geq 3$ and at least $3$ for $g \geq 6$. Thus, by \eqref{b3fact}, for $g \geq 3$, the only relations in codimension $1$ come from the push forwards of classes on $\widetilde{\Delta}_{3,g}$. Further, for $g \geq 6$, the only relations in codimension $2$  come from the push forwards of classes supported on $\widetilde{\Delta}_{3,g}$.

To prove (1) and (2), it suffices to show that $I$ already accounts for all such relations in codimension $1$ when $g \geq 3$ and for all such relations in codimension $2$ when $g \geq 6$.
Precisely, let $\mathcal{Z} \subset \pp \E^\vee$ be the locus where the map \eqref{eval3} fails to be surjective on fibers.
We will show that
\begin{equation} \label{exeq1} A^0(\widetilde{\Delta}_{3,g}) = A^0(\widetilde{\Delta}_{3,g} \smallsetminus \rho''^{-1}(\mathcal{Z})) \cong \rho''^*A^0(\pp \E^\vee \smallsetminus \mathcal{Z}) = \rho''^*A^0(\pp \E^\vee)
\end{equation}
and when $g \neq 4$, that
\begin{equation} \label{exeq2}
A^1(\widetilde{\Delta}_{3,g}) = A^1(\widetilde{\Delta}_{3,g} \smallsetminus \rho''^{-1}(\mathcal{Z})) \cong \rho''^*A^1(\pp \E^\vee \smallsetminus \mathcal{Z}) =\rho''^*A^1(\pp \E^\vee).
\end{equation}
The middle isomorphism follows in both cases from the fact that $\widetilde{\Delta}_{3,g} \smallsetminus \rho''^{-1}(\mathcal{Z})$ is a vector bundle over $\pp \E^\vee \smallsetminus \mathcal{Z}$. To show the other equalities we use excision.

We claim that the map \eqref{eval3} always has rank at least $2$.
To see this, consider the diagram
\begin{equation} \label{rk2}
\begin{tikzcd}
\gamma^*\pi^* \pi_* \gamma_* \W \arrow{r} \arrow{d} & P^1_{\pp \E^\vee/\B_{3,g}'}(\W) \arrow{d} \\
\gamma^* \gamma_* \W \arrow{r} & P^1_{\pp \E^\vee/\P}(\W) 
\end{tikzcd}
\end{equation}
The left vertical map is a surjection because $\gamma_* \W$ is relatively globally generated along $\P \to \B_{3,g}'$ (by definition of $\B_{3,g}'$, see \eqref{bpdef}); the bottom horizontal map is surjective by Lemma \ref{veryample} because
$\W$ is relatively very ample on $\pp \E^\vee$ over $\P$. Thus, the top horizontal map must have rank at least $2 = \mathrm{rank} (P^1_{\pp \E^\vee/\P}(\W))$.
It follows that 
\begin{equation} \label{cocomp}
\codim(\rho''^{-1}(\mathcal{Z}) \subset \widetilde{\Delta}_{3,g}) = \codim(\mathcal{Z} \subset \pp \E^\vee) - 1.
\end{equation}
By the argument in Lemma \ref{famjet}, $\mathcal{Z}$ is the locus where $\W$ fails to induce a relative embedding on $\pp \E^\vee$ over $\B_{3,g}'$. 
By Lemma \ref{line}, the restriction to a fiber over $\B_{3,g}'$, say $\W|_{\p E^\vee}\cong\mathcal{O}_{\pp E^\vee}(3)\otimes \gamma^*\mathcal{O}_{\p^1}(-g-2)$ fails to be very ample if and only if
$E \cong \O(e_1) \oplus \O(e_2)$
with
$3e_1 \leq g+2$.
Moreover, in this case, the linear system fails to induce an embedding precisely along the directrix of $\pp E^\vee$.
By definition of $\B_{3,g}'$, we always have $3e_1 \geq g+2$. Thus, $\gamma(\mathcal{Z})$ is contained in at most one splitting locus, which is nonempty if and only if $g \equiv 1 \pmod 3$.
In particular:
\begin{enumerate}
\item if $g = 4$, then $\gamma(\mathcal{Z})$ is the splitting locus $(e_1, e_2) = (2, 4)$, which has codimension $1$ 
\item if $g = 7$, then $\gamma(\mathcal{Z})$ is the splitting locus $(e_1, e_2) = (3, 6)$, which has codimension $2$
\item if $g \neq 4, 7$, then $\gamma(\mathcal{Z})$ has codimension at least $3$
\end{enumerate}
Since the directrix has codimension $1$, it follows that
\[\codim(\mathcal{Z} \subset \pp \E^\vee) = \begin{cases} 2 & \text{if $g = 4$} \\ 3 & \text{if $g = 7$} \\ \geq 4 & \text{otherwise.}\end{cases} \]
By \eqref{cocomp}, we see then that $\rho''^{-1}(\mathcal{Z})$ has suitably high codimension so that \eqref{exeq1} is satisfied for all $g$ and \eqref{exeq2} is satisfied for $g \neq 4$.
\end{proof}

This completes the proof of Theorem \ref{main}(1) when $g \geq 6$.

\subsection{Low genus calculations} \label{lgcalcs}
The lemmas in this section show that the remaining Chow groups not already determined by Lemma \ref{trig-result} vanish. This is due to certain geometric phenomena that occur in low codimension when the genus is small.

\begin{lem} \label{g2} 
When $g = 2$, we have $A^*(\H_{3,2}) = 0$.
\end{lem}
\begin{proof}
When $g = 2$, the complement of $\B_{3,2}' \subset \B_{3,2}$ is the $(1, 3)$ splitting locus, which has codimension $1$. As a consequence, $a_1$ and $a_2'$ satisfy a relation on $\B_{3,2}'$. Using  \cite[Lemma 5.1]{L}, we calculate the class of the $(1, 3)$ splitting locus as the 
degree $1$ piece of a ratio of total Chern classes below, which we compute with the code \cite{github}:
\[0 = s_{1,3} = \left[\frac{c((\pi_* \E(-2) \otimes \pi_* \O_{\P}(1))^\vee)}{c((\pi_* \E(-1))^\vee)}\right]_1 = a_2'-2a_1\]
on $\B_{3,2}'$.
Specializing \eqref{3codim1} to $g = 2$, we also have the additional relation $0 = [\Delta_{3,2}] = 28a_1 - 9a_2'$ in $A^1(\H_{3,g})$, so we conclude $a_1 = a_2' = 0$ and hence, the Chow ring is trivial.
\end{proof}

\begin{lem} \label{van3}
For $g = 3, 4, 5$, we have $A^2(\H_{3,g}) = 0$.
\end{lem}
\begin{proof}
We first explain the case $g = 3$. 
Here, the complement of $\B_{3,3}'$ inside $\B_{3,3}$ is the closure of the splitting locus $(e_1, e_2) = (1, 4)$, which has codimension $2$. 
 The universal formulas for classes of splitting loci \cite{L} say that the class of this unbalanced splitting locus is the degree $2$ piece of a ratio of total Chern classes, which we computed in the code \cite{github},
\[s_{1,4} = \left[\frac{c((\pi_*\E(-2)\otimes \pi_*\O_{\P}(1))^\vee)}{c((\pi_*\E(-1))^\vee)}\right]_2 = 3a_1^2+\frac{1}{2}a_2-\frac{5}{2}a_1a_2'+\frac{1}{2}a_2'^2+3c_2.\] 
It follows that $A^*(\H_{3,3})$ is a quotient of $\qq[a_1, a_2, a_2', c_2]/(I + \langle s_{1,4} \rangle)$. We checked in the code \cite{github} that the codimension $2$ piece of this ring is zero.

The case $g = 5$ is very similar so we explain it next. The complement of $\B_{3,5}'$ inside $\B_{3,5}$ is the closure of the splitting locus $(e_1, e_2) = (2, 5)$, which has codimension $2$. The class of this splitting locus is computed similarly:
\[s_{2,5} = \left[\frac{c((\pi_*\E(-3)\otimes \pi_*\O_{\P}(1))^\vee)}{c((\pi_*\E(-2))^\vee)}\right]_2 = 6a_1^2+\frac{1}{2}a_2-\frac{7}{2}a_1a_2'+\frac{1}{2}a_2'^2+6c_2.\]
Therefore, $A^*(\H_{5,3})$ is a quotient $\qq[a_1, a_2, a_2', c_2]/(I + \langle s_{2,5}\rangle)$, whose codimension $2$ piece we also checked is zero \cite{github}.

In the case $g = 4$, our additional relation will come from $\rho''^{-1}(\mathcal{Z}) \subset \widetilde{\Delta}_{3,4}$, which has codimension $1$, and whose push forward therefore determines a class that is zero in $A^2(\H_{3,4})$.
By \eqref{rk2}, we have that $\rho''^{-1}(\mathcal{Z})$ is the transverse intersection of $\rho'^{-1}(\mathcal{Z})$ with the kernel subbundle of $\gamma^*\pi^*\pi_*\gamma_* \W \to P^1_{\pp \E^\vee/\P}(\W)$. That is, our possible additional relation is given by
\begin{equation} \label{ss}
s := \pi'_* \gamma'_* i_* [\rho''^{-1}(\mathcal{Z})] = \gamma'_* \pi'_* (\rho'^*[\mathcal{Z}] \cdot \rho'^*c_2(P^1_{\pp \E^\vee/\P}(\W))) = \rho^*\gamma_* \pi_*([\mathcal{Z}] \cdot c_2(P^1_{\pp \E^\vee/\P}(\W))).
\end{equation}
It remains to compute $[\mathcal{Z}]$, which we do now.
Let $\Sigma = \gamma(\mathcal{Z}) \subset \B_{3,4}'$ be the $(2, 4)$ splitting locus. 
Using the formulas for classes of splitting loci \cite{L}, we compute
\[[\Sigma] = s_{1,4} = \left[\frac{c(((\pi_*\E(-3) \otimes \pi_*\O_{\P}(1))^\vee)}{c(((\pi_*\E(-2))^\vee)}\right]_1 =  a_2' - 3a_1. \]

Over $\Sigma$, there is a sequence
\begin{equation} \label{onSigma}
0 \rightarrow \pi^*\mathcal{M}(-2) \rightarrow \E^\vee|_{\Sigma}  \rightarrow \pi^*\mathcal{N}(-4) \rightarrow 0
\end{equation}
for line bundles $\mathcal{M}$ and $\mathcal{N}$ on $\Sigma$. 
Let $m = c_1(\mathcal{M})$ and $n = c_1(\mathcal{N})$.
The directrix over $\Sigma$ is $\mathcal{Z} = \pp (\pi^*\mathcal{M}(-2)) \subset \pp \E^\vee|_{\Sigma}$. By \cite[Proposition 9.13]{EH}, the fundamental class of $\mathcal{Z}$ inside $\pp \E^\vee|_{\Sigma}$ is $\zeta + c_1(\pi^*\mathcal{N}(-4)) = \zeta+ n - 4z$. Considering Chern classes in the exact sequence \eqref{onSigma}, we learn (recall $a_1' = g+2 = 6$)
\begin{align*} -a_1|_{\Sigma} - 6z = c_1(\E^\vee|_{\Sigma}) &= m - 4z + n -2z &\quad &\Rightarrow &\quad m+n &= -a_1|_{\Sigma}
\intertext{and}
a_2|_{\Sigma} + (a_2'|_{\Sigma}) \cdot z = c_2(\E^\vee|_{\Sigma}) &= (m - 4z)(n - 2z) \\
&= mn -c_2 - (2m + 4n)z &\quad &\Rightarrow &\quad 2m+4n &= -a_2'|_{\Sigma}.
\end{align*}
In particular, $n =\left. \left(a_1 - \frac{a_2'}{2}\right)\right|_{\Sigma}$. Hence, the fundamental class of $\mathcal{Z}$ inside  all of $\pp \E^\vee$ is 
\[[\mathcal{Z}] = (\zeta + a_1 - \tfrac{a_2'}{2} - 4z) \cdot [\Sigma] = (\zeta + a_1 - \tfrac{a_2'}{2} - 4z)(a_2' - 3a_1).\]
This allows us to compute $s$ in \eqref{ss}, and
our code confirms that the codimension $2$ piece of $\qq[a_1, a_2, a_2', c_2]/(I + \langle s \rangle)$ is zero \cite{github}.
\end{proof}

Together, Lemmas \ref{trig-result}, \ref{g2} and \ref{van3} determine the rational Chow ring of $\H_{3,g}$ for all $g$:
 \[A^*(\H_{3,g}) = \begin{cases} \qq & \text{if $g = 2$} \\ \qq[a_1]/(a_1^2) & \text{if $g = 3, 4, 5$} \\ \qq[a_1]/(a_1^3) & \text{if $g \geq 6$.} \end{cases}\]
This completes the proof of Theorem \ref{main}(1).

\section{The Chow ring in degree $4$} \label{re4}
\subsection{Set up}
We begin by briefly recalling the linear algebraic data associated to a degree $4$ cover, as developed by Casnati--Ekedahl \cite{CE}. For more details in our context, see \cite[Section 3.2]{part1}. Given a degree $4$ cover $\alpha: C \to \pp^1$, we associate two vector bundles on $\pp^1$:
\[E_\alpha := (\alpha_*\O_C/\O_{\pp^1})^\vee = \ker(\alpha_*\omega_\alpha \to \O_{\pp^1}) \qquad \text{and} \qquad F_\alpha := \ker(\Sym^2 E_\alpha \to \alpha_*\omega_{\alpha}^{\otimes 2}).\]
The first is rank $3$ and the second is rank $2$.
If $C$ has genus $g$, then both bundles have degree $g+3$.
Geometrically, the curve $C$ is embedded in $\gamma:\p E^{\vee}_{\alpha}\rightarrow \p^1$ as the zero locus of a section
\[
\delta_{\alpha} \in H^0(\p E_{\alpha}^{\vee},\O_{\p E^{\vee}_{\alpha}}(2)\otimes \gamma^*F_{\alpha}^{\vee}).
\]
In each fiber of $\gamma$, the four points are the base locus of a pencil of conics parametrized by $F_\alpha$.

Conversely, given vector bundles $E, F$ of ranks $3$ and $2$, both of degree $g + 3$, we wish to characterize when a section
\[\delta \in H^0(\pp E^\vee, \O_{\pp E^\vee}(2) \otimes \gamma^* F^{\vee})\]
fails to produce a smooth degree $4$, genus $g$ cover.

\begin{lem}\label{4limits}
Suppose $E, F, \delta$ are as above with $F^\vee \otimes \Sym^2 E$ globally generated.
If the zero locus $C = V(\delta)$ is not an irreducible, smooth quadruple cover of $\p^1$, then there is a point $p\in C$ such that $\dim T_p C \geq 2$.
\end{lem}
\begin{proof}
If $C$ is connected or has a component of dimension at least $2$ then the lemma is immediate. Suppose $C$ is $1$-dimensional and disconnected. We first rule out the case in which $C$ has at least $2$ connected components, both mapping finitely onto $\p^1$. In this case, $\alpha_* \O_C$ has  more than one $\O$ factor; then $E$ has a degree $0$ summand, so $\Sym^2 E \otimes F^\vee$ would have a negative degree summand, which we are assuming is not the case. 

Next suppose $C$ has a component $C_0$ which does not map finitely onto $\p^1$. Then $C_0$ must be contained in a fiber of $\gamma:\p  E^{\vee}\rightarrow \p^1$. The restriction of the zero locus of $\delta$ to a fiber is the intersection of two (possibly singular) conics in $\p^2$. The only way for such an intersection to have a $1$-dimensional component is for the conics to have a common component $C_0$. Hence, some fiber of $C$ is equal to $C_0$ union a finite subscheme of length less than $4$ (length $1$ if $C_0$ is a line, empty if $C_0$ is a conic).
\begin{center}
\includegraphics[width=4.5in]{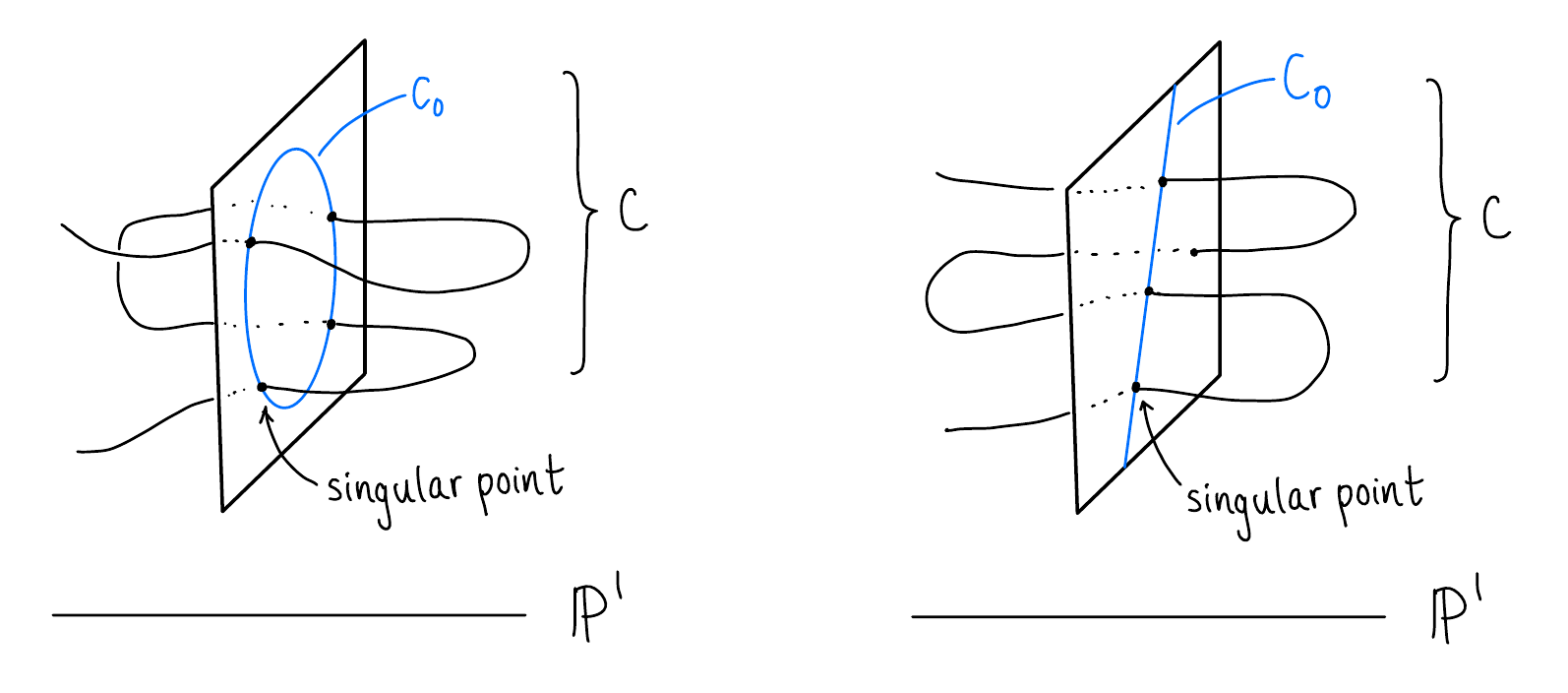}
\end{center}
Since the generic fiber consists of $4$ points, some of those $4$ points must specialize to $C_0$, which means $C$ is singular at those points on $C_0$ (and $C$ is connected).
\end{proof}

The association of $\alpha: C \to \pp^1$ with the pair $(E_\alpha, F_\alpha)$ gives rise to map $\H_{4,g}$ to the moduli stack $\B_{4,g}$ of pairs of vector bundles on $\pp^1$-bundles, as defined in \cite[Definition 5.2]{part1}. 
Let $\pi: \P \to \B_{4,g}$ be the universal $\pp^1$-bundle. Let $\E$ be the universal rank $3$ vector bundle on $\P$, and $\F$ the universal rank $2$ bundle on $\P$.
Continuing the notation of \cite{part1}, let $z = c_1(\O_{\P}(1))$ and define 
classes $a_i, b_i \in A^i(\B_{4,g})$ and $a_i', b_i' \in A^{i-1}(\B_{4,g})$ by the formula
\[c_i(\E) = \pi^*a_i + \pi^*a_i'z \qquad \text{and} \qquad c_i(\F) = \pi^*b_i + \pi^*b_i' z.\]
(Note that there is a ``determinant compatibility" condition which implies $a_1 = b_1$, see \cite[Equation 5.4]{part1}.)
We also define $c_2 = -\pi_*(z^3) \in A^2(\B_{4,g})$, which is the pullback of the universal second Chern class on $\BSL_2$. By \cite[Equation 5.5]{part1}, 
\begin{gather} \label{b4fact}
\text{$a_1, a_2, a_3, a_2', a_3', b_2, b_2', c_2$ generate $A^*(\B_{4,g})$ and satisfy} \\
\text{no relations in degrees up to $g + 3$}. \notag
\end{gather}
We call the pullbacks of $\E$ and $\F$ to $\H_{4,g}$ the \emph{CE bundles} (these are the bundles appearing in the Casnati--Ekedahl resolution for the universal curve).
We call the pullbacks to $\H_{4,g}$ of the associated classes in \eqref{b4fact} the
\emph{CE classes}. By \cite[Theorem 3.10]{part1}, the CE classes are tautological and generate the tautological ring.

Up to this point, the set up has been quite similar to degree $3$. However, unlike in degree $3$, the full Hurwitz stack $\H_{4,g}$ \emph{cannot} be realized as an open substack of a vector bundle over an open substack of $\B_{4,g}$. 
This is why we were unable to determine the full Chow ring of $\H_{4,g}$ with our techniques. We now proceed in two steps. First, in Section \ref{rels-step1} we shall produce several relations among CE classes on $\H_{4,g}$ using principal parts bundles. Then, in Section \ref{rels-step2}, we shall define an open substack $\H_{4,g}^\circ \subset \H_{4,g}$, which does lie inside a vector bundle over an open substack $\B_{4,g}^\circ \subset \B_{4,g}$, and use it to demonstrate that we have found all relations in degrees up to roughly $g/4$. It may help to think of $\H_{4,g}^\circ$ as an open substack that is ``large enough to witness the independence of many CE classes."


\subsection{Relations among CE classes} \label{rels-step1}
Let $\E$ and $\F$ be the CE bundles on the universal $\pp^1$-bundle $\pi: \P \to \H_{4,g}$. Let $\gamma: \pp \E^\vee \to \P$ be the structure map. 
We define a rank $2$ vector bundle on $\pp \E^\vee$ by $\W := \O_{\pp \E^\vee}(2) \otimes \F^\vee$. By the Casnati--Ekedahl theorem in degree 4, see \cite[Equation 3.7]{part1} or \cite{CE}, the universal curve $\C \subset \pp \E^\vee$ determines a global section $\delta^{\mathrm{\mathrm{univ}}}$ of $\W$, whose vanishing locus is $V(\delta^{\mathrm{\mathrm{univ}}}) = \C \subset \pp \E^\vee$. 

The global section $\delta^{\mathrm{\mathrm{univ}}}$ induces a global section $\delta^{\mathrm{\mathrm{univ}}}{'}$ of the principal parts bundle $P_{\pp \E^\vee/\H_{4,g}}^1(\W)$ on $\pp \E^\vee$, which records the value and derivatives of $\delta^{\mathrm{\mathrm{univ}}}$.
Now consider the tower
\begin{center}
    \begin{tikzcd}
{G(2,T_{\p\E^{\vee}/\H_{4,g}})} \arrow[r, "a"'] & \p\E^{\vee} \arrow[r, "\gamma"']              & \P \arrow[r, "\pi"']               & {\H_{4,g}},
    \end{tikzcd}
\end{center}
where
 $G(2, T_{\pp \E^\vee/ \H_{4,g}})$ parametrizes $2$ dimensional subspaces of the vertical tangent space of $\pp \E^\vee$ over $\H_{4,g}$. 
 Dualizing the tautological sequence on $G(2,T_{\p\E^{\vee}/\H_{4,g}^\circ})$ we obtain a filtration 
\[0 \rightarrow \Omega_y \rightarrow a^* \Omega_{\p\E^{\vee}/\H_{4,g}} \to \Omega_x \rightarrow 0,\]
where $\Omega_y$ is rank $1$ and $\Omega_x$ is rank $2$. Let $P_{\pp \E^\vee/\H_{4,g}}^{\{1, x\}}(\W)$ be the bundle of restricted principal parts as defined in Section \ref{dir-ref}.

 On $G(2, T_{\pp \E^\vee/\H_{4,g}})$, we obtain a global section, call it $\delta^{\mathrm{\mathrm{univ}}}{''}$, of $P_{\pp \E^\vee/\H_{4,g}}^{\{1, x\}}(\W)$ by composing the section $a^* \delta^{\mathrm{\mathrm{univ}}}{'}$ with the quotient $a^* P_{\pp \E^\vee/\H_{4,g}}^1(\W) \to P_{\pp \E^\vee/\H_{4,g}}^{\{1, x\}}(\W)$.
 The vanishing locus of $\delta^{\mathrm{\mathrm{univ}}}{''}$
is the space of pairs $(p, S)$ where $p \in V(\delta^{\mathrm{\mathrm{univ}}}) \subset \pp \E^\vee$ and $S$ is a two-dimensional subspace of the tangent space of the fiber of $V(\delta^{\mathrm{univ}}) \to \H_{4,g}$ through $p$. But $V(\delta^\mathrm{univ}) = \C \to \H_{4,g}$ is smooth of relative dimension $1$. Thus, $\delta^{\mathrm{\mathrm{univ}}}{''}$ must be non-vanishing on $G(2, T_{\pp \E^\vee/\H_{4,g}})$. 

Since $P_{\pp \E^\vee/\H_{4,g}}^{\{1, x\}}(\W)$ has a non-vanishing global section, its top Chern class, $c_6(P_{\pp \E^\vee/\H_{4,g}}^{\{1,x\}}(\W))$, must be $0$. Moreover, the push forward of this class times any class on $G(2, T_{\pp \E^\vee/\H_{4,g}})$ is also zero. Such relations are generated by the following classes.
\begin{lem} \label{CErels3}
Let $\tau = c_1(\Omega_y^\vee)$ where $\Omega_y^\vee$ is the tautological quotient bundle on $G(2, T_{\pp \E^\vee/ \H_{4,g}})$. Let $\zeta = \O_{\pp \E^\vee}(1)$ and $z = c_1(\O_{\P}(1))$. Then all classes of the form
\begin{equation} \label{rel4first}
\pi_*\gamma_*a_*(c_6(P^{\{1, x\}}_{\p\E^{\vee}/\H_{4,g}}(\W)) \cdot \tau^i \zeta^j z^k) 
\end{equation}
are zero in $R^*(\H_{4,g}) \subseteq A^*(\H_{4,g})$.
\end{lem}
It is straightforward for a computer to compute such push forwards as polynomials in the CE classes. We describe the ideal these push forwards generate in Section \ref{ach}

\subsection{All relations in low codimension} \label{rels-step2}
We now recall the construction of our ``large open" substack $\H_{4,g}^{\circ}\subset \H_{4,g}$.
We start with $\B_{4,g}$, the moduli space of pairs of vector bundles $E$ of rank $3$, degree $g+3$ and $F$ of rank $2$, degree $g+3$ on $\pp^1$-bundles together with an isomorphism of their determinants (see \cite[Section 5.2]{part1}). Now, working over $\B_{4,g}$, let $\E$ and $\F$ be the universal bundles on $\pi:\P\rightarrow \B_{4,g}$ and let $\gamma: \pp \E^\vee \to \P$ be the structure map. Define $\mathcal{W} := \gamma^* \F^\vee \otimes \O_{\pp \E^\vee}(2)$, and let $\U_{4,g} := \gamma_* \mathcal{W} = \F^\vee \otimes \Sym^2 \E$.
 We consider an open substack $\B^{\circ}_{4,g}\subset \B_{4,g}$, defined by a certain positivity condition for the bundle $\aV_{4,g}$
    \begin{equation} \label{defb4c}
    \B^{\circ}_{4,g}:=\B_{4,g}\smallsetminus \Supp R^1\pi_*(\aV_{4,g}(-2)).
    \end{equation}
    Let $\H^{\circ}_{4,g}$ denote the base change of $\H_{4,g}\rightarrow \B_{4,g}$ along the open embedding $\B_{4,g}^{\circ}\hookrightarrow \B_{4,g}$.

\begin{figure}[h!]
\includegraphics[width=6in]{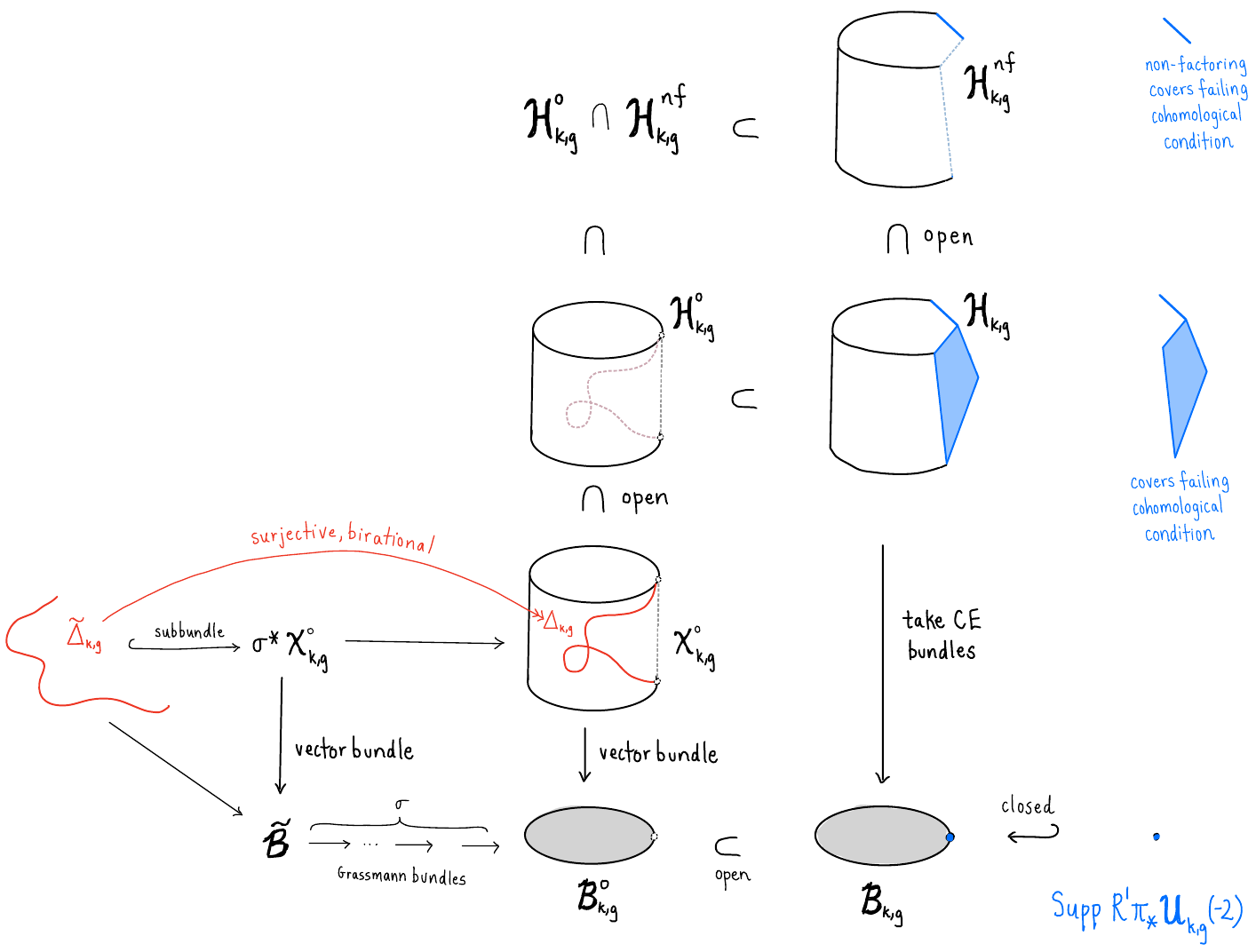}
\caption{Summary of the method}
\label{sumfig}
\end{figure}

\begin{rem}
We note that the complement of $\mathcal{H}^\circ_{4,g} \subset \H_{4,g}$ (represented in blue in the right of Figure \ref{sumfig}) contains covers that factor through a curve of low genus (see \cite[p. 21-22]{part1}). Thus, the codimension of the complement of $\H^\circ_{4,g} \subset \H_{4,g}$ is $2$. However, upon restricting to non-factoring covers, the codimension of the complement of $\H^{\circ}_{4,g} \cap \H_{4,g}^{\nf} \subset \H_{4,g}^{\nf}$ has codimension at least $\frac{g+3}{4} - 4$ \cite[Lemma 5.5]{part1}. In this sense, $\H_{4,g}^{\mathrm{nf}}$ and $\H_{4,g}^\circ$ are ``good approximations" to each other. This is what allows us to find stabilization results for $\dim A^i(\H_{4,g}^{\mathrm{nf}})$.
\end{rem}

Over $\B^{\circ}_{4,g}$, we see that $\mathcal{X}_{4,g}^{\circ}:=\pi_*\aV_{4,g}|_{\B^{\circ}_{4,g}}$ is a vector bundle whose fibers correspond to sections of $\aV_{4,g}$. 
The open $\H_{4,g}^\circ$ is contained in the open $\H_{4,g}'$ of \cite[Lemma 5.3]{part1}, so
that lemma implies $\H^{\circ}_{4,g} \to \B_{4,g}^\circ$
 factors through an open embedding in $\mathcal{X}^{\circ}_{4,g}$. We define
    \[
    \Delta_{4,g}:=\mathcal{X}^{\circ}_{4,g}\smallsetminus \H^{\circ}_{4,g},
    \]
represented in red in the middle column of Figure \ref{sumfig}.
Now we wish to use the excision
to determine the Chow ring of $\H_{4,g}^\circ$ in degrees up to $\frac{g+3}{4} - 4$. Note that $A^*(\mathcal{X}^\circ_{4,g}) \cong A^*(\B_{4,g}^\circ)$, and we have already determined the latter in degrees up to $\frac{g+3}{4} - 4$ by \cite[Equation 5.6]{part1}.

The next step is to construct a space $\widetilde{\Delta}_{4,g}$ (pictured in red on the far left of Figure \ref{sumfig}), which surjects properly onto $\Delta_{4,g}$. With rational coefficients, the push forward $\widetilde{\Delta}_{4,g} \to \Delta_{4,g}$ will be surjective on Chow groups.
Thus, pushing forward all classes from $\tilde{\Delta}_{4,g}$ will produce all relations needed to describe $\H_{4,g}^\circ$ as a quotient of $A^*(\mathcal{X}^\circ_{4,g}) \cong A^*(\B_{4,g}^\circ)$.

Each geometric point of $\aW_{4,g}^\circ$ corresponds to 
a tuple $(E, F, \delta)$ where $E, F$ are vector bundles on $\pp^1$ and $\delta \in H^0(\pp E^\vee, F^\vee \otimes \O_{\pp E^\vee}(2))$.
We now use restricted bundles of relative principal parts for $\p \E^{\vee}\rightarrow \B_{4,g}^\circ$ to define a space parametrizing triples
\[((E, F, \delta) \in \aW_{4,g}^\circ, \  p \in V(\delta), \ S \subset T_p V(\delta) \text{ of dimension 2}).\]
Let $a:G(2,T_{\p\E^{\vee}/\B_{4,g}^\circ})\rightarrow \p\E^{\vee}$ be the Grassmann bundle of $2$-planes in the relative tangent bundle. Dualizing the tautological sequence on $G(2,T_{\p\E^{\vee}/\B_{4,g}^\circ})$ we obtain a filtration 
\[0 \rightarrow \Omega_y \rightarrow a^* \Omega_{\p\E^{\vee}/\B_{4,g}^\circ} \to \Omega_x \rightarrow 0,\]
where $\Omega_y$ is rank $1$ and $\Omega_x$ is rank $2$.
Using the bundle of restricted principal parts constructed in Section \ref{dir-ref}, we obtain an evaluation map
\begin{equation} \label{eval4}
a^*\gamma^*\pi^*\pi_*\gamma_*\W\cong a^*\gamma^*\pi^*\aW_{4,g}^\circ \rightarrow P^1_{\pp \E^\vee/\B_{4,g}^\circ}(\W) \rightarrow P^{\{1,x\}}_{\p\E^{\vee}/\B^\circ_{4,g}}(\W),
\end{equation}
which we claim is surjective.
The rightmost map from principal parts to restricted principal parts is always a surjection. Thus, it suffices to show that the map 
$
\gamma^*\pi^*\aW_{4,g}^\circ \rightarrow P^1_{\p\E^{\vee}/\B^{\circ}_{4,g}}(\W)
$
is surjective. By definition of $\B_{4,g}^\circ$ (see \eqref{defb4c}), we have $R^1\pi_*[(\gamma_* \W) \otimes \O_{\P}(-2)] = 0$, so the surjectivity follows from Lemma \ref{famjet}.

We define $\tilde{\Delta}_{4,g}$ to be the kernel bundle of \eqref{eval4}. We have the following ``trapezoid" diagram:
\begin{equation}\label{4diagram}
    \begin{tikzcd}
{\tilde{\Delta}_{4,g}} \arrow[r, hook] \arrow[rd, "\rho''"'] & a^*\gamma^*\pi^*\aW_{4,g}^\circ \arrow[r, "a'"] \arrow[d, "\rho'"] & \gamma^*\pi^*\aW_{4,g}^\circ \arrow[r, "\gamma'"] \arrow[d] & \pi^*\aW_{4,g}^\circ \arrow[r, "\pi'"] \arrow[d] & \aW_{4,g}^\circ \arrow[d, "\rho"] \\
                                                             & {G(2,T_{\p\E^{\vee}/\B^\circ_{4,g}})} \arrow[r, "a"'] & \p\E^{\vee} \arrow[r, "\gamma"']              & \P \arrow[r, "\pi"']               & {\B_{4,g}^\circ} 
\end{tikzcd}
\end{equation}

\begin{prop}\label{4image}
Let $\tau = c_1(\Omega_y^\vee)$ where $\Omega_y^\vee$ is the tautological quotient line bundle on $G(2, T_{\pp \E^\vee/ \B_{4,g}^\circ})$. Let $\zeta = \O_{\pp \E^\vee}(1)$ and $z = c_1(\O_{\P}(1))$. 
Let $I$ be the ideal generated by 
\begin{equation} \label{relsgen}
\pi_*\gamma_*a_*(c_6(P^{\{1, x\}}_{\p\E^{\vee}/\B^\circ_{4,g}}(\W)) \cdot \tau^i \zeta^j z^k) \qquad \text{for } 0 \leq i, j \leq 2, \ 0 \leq k \leq 1.
\end{equation}
Then $A^*(\H_{4,g}^\circ) \cong A^*(\B_{4,g}^\circ)/I$. Together with $a_1 = b_1$, the classes in \eqref{rel4first} therefore generate all relations among the CE classes on $\H_{4,g}$ in degrees less than $\frac{g+3}{4} - 4$.
\end{prop}
\begin{proof}
By Lemma \ref{4limits}, $\widetilde{\Delta}_{4,g}$ surjects onto $\Delta_{4,g}$, so we may apply the Trapezoid Lemma \ref{trap}.
 Since $T_{\pp \E^\vee/\B_{4,g}^\circ}$ has rank $3$, the Grassmann bundle $G(2, T_{\pp \E^\vee/\B_{4,g}^\circ})$ is just the projectivization of $T_{\pp \E^\vee/\B_{4,g}^\circ}^\vee$; hence its Chow ring is generated as a module over $A^*(\pp \E^\vee)$ by $\tau^i$ for $0 \leq i \leq 2$. Similarly $A^*(\pp \E^\vee)$ is generated as a module over $A^*(\P)$ by $\zeta^j$ for $0 \leq j \leq 2$ and $A^*(\P)$ is generated as a module over $A^*(\B_{4,g}^\circ)$ by $z^k$ for $0 \leq k \leq 1$. Thus, the Trapezoid Lemma \ref{trap} implies that the classes in \eqref{relsgen} generate all relations among the pullbacks of classes on $\B_{4,g}^\circ$. In particular, setting $i=j=k=0$, we obtain 
 \begin{equation} \label{4delta} [\Delta_{4,g}] = \pi'_* \gamma'_*a'_* [\widetilde{\Delta}_{4,g}] = \rho^*(\pi \circ \gamma\circ a)_*(c_6(P^1_{\pp \E^\vee/\B_{4,g}^{\circ}}(\W))) = (8g + 20)a_1 - 8a_2' - b_2'.
\end{equation}
 
 To see the second claim, note that the classes in \eqref{relsgen} pullback to the classes in \eqref{rel4first}.
 By \cite[Equation 5.6]{part1}, the generators $a_1 = b_1, a_2, a_2', a_3, a_3', b_2, b_2', c_2$ of $A^*(\B_{4,g}^\circ)$ satisfy no relations in codimension less than $\frac{g+3}{4}-4$ (besides $a_1 = b_1$). Since one can only obtain more relations under restriction $A^*(\H_{4,g}) \to A^*(\H_{4,g}^\circ)$, we have found all relations among CE classes in degrees less than $\frac{g+3}{4}-4$.
\end{proof}

\subsection{Presentation of the ring and stabilization} \label{ach}
We use the code \cite{github} compute the classes in \eqref{rel4first}. Let $I$ be the ideal they generate in the $\qq$-algebra on the CE classes.
It turns out that modulo $I$, all CE classes are expressible in terms of $a_1, a_2', a_3'$. In particular,
\begin{equation} \label{cealg}
\qq[c_2, a_1, a_2, a_3, a_2', a_3', b_2', b_2]/I \cong \qq[a_1, a_2', a_3']/\langle r_1, r_2, r_3, r_4 \rangle,
\end{equation}
where
{\small
\begin{align*}
r_1 &= (2g^3 + 9g^2 + 10g)a_1^3 - (8g^2 + 24g + 8)a_1a_3' \\
r_2 &= (12g^3 + 42g^2 + 36g)a_1^2a_2' - (22g^3 + 121g^2 + 187g + 66)a_1a_3' - (24g^2 + 24g)a_2'a_3' \\
r_3 &= (432g^3 + 1512g^2 + 1296g)a_1a_2'^2 - (1450g^3 + 8001g^2 + 13115g + 5442)a_1a_3' \\
&\qquad \qquad \qquad \qquad \qquad \qquad \qquad \ \  - (1584g^3 + 5544g^2 + 3936g)a_2'a_3' \\
r_4 &= (14344g^6 + 165692g^5 + 747682g^4 + 1636869g^3 + 1719009g^2 + 677844g - 540)a_1^2a_3' \\
&- (17280g^4 + 112320g^3 + 224640g^2 + 129600g)a_2'^2a_3' + (352g^5 + 1440g^4 + 1448g^3 + 120g^2)a_3'^2.
\end{align*}}

\vspace{-.2in}
\begin{rem}
In contrast with the degree $3$ case, brute force computations show that there is no presentation of the Chow ring whose relations do not involve $g$.
\end{rem}

\begin{cor} \label{rcor}
Suppose $g\geq 2$.
\begin{enumerate}
    \item $R^1(\H_{4,g})$ is spanned by $\{a_1,a_2'\}$.
    \item $R^2(\H_{4,g})$ is spanned by $\{a_1^2,a_1a_2',a_2'^2,a_3'\}$.
    \item $R^3(\H_{4,g})$ is spanned by $\{a_1a_3',a_2'^3,a_2'a_3'\}$.
    \item $R^4(\H_{4,g})$ is spanned by $\{a_2'^4,a_3'^2\}$.
    \item For $i\geq 5$, $R^i(\H_{4,g})$ is spanned by $\{a_2'^i\}$.
\end{enumerate}
For $g > 4i + 12$, the spanning set of $R^i(\H_{4,g})$ given above is a basis.
\end{cor}
\begin{proof}
Our code \cite{github} verifies that the lists above are bases for $\qq[a_1, a_2', a_3']/\langle r_1, r_2, r_3, r_4 \rangle$ in degrees $i \leq 9$. In particular, for $5 \leq i \leq 10$, every monomial in $a_1, a_2', a_3'$ of degree $i$ is a multiple of $a_2'^i$.
By inspection, $a_2'^i$ is not in the ideal $\langle r_1, r_2, r_3, r_4 \rangle$
for any $i$, so $a_2'^i$ is non-zero for all $i$. 
For $i \geq 11$, every monomial of degree $i$ in
$a_1, a_2', a_3'$ is expressible as a product of monomials having degrees between $5$ and $10$. It follows that every monomial of degree $i \geq 11$ is a multiple of $a_2'^i$.

Proposition \ref{4image} states that $I$ provides all relations among the CE classes in degrees less than $\frac{g+3}{4}-4$. That is, the left-hand side of \eqref{cealg} maps to $R^*(\H_{4,g})$ isomorphically in degrees $i < \frac{g+3}{4}-4$.
Hence, a basis for the degree $i$ piece of $\qq[a_1, a_2', a_3']/\langle r_1, r_2, r_3, r_4 \rangle$ is a basis for $R^i(\H_{4,g})$ when $i < \frac{g+3}{4}-4$, equivalently when $g > 4i + 12$.
\end{proof}

\begin{proof}[Proof of Theorem \ref{main}(2)]
Consider the equation
\[\frac{\qq[a_1, a_2',
a_3']}{\langle r_1, r_2, r_3, r_4 \rangle} \rightarrow R^*(\H_{4,g}) \rightarrow R^*(\H_{4,g}^\circ) = A^*(\H_{4,g}^\circ).\]
The first map exists and is surjective by Proposition \ref{CErels3} and the presentation \eqref{cealg}. Meanwhile, Lemma \ref{4image} establishes that the composition is an isomorphism in degrees less than $\frac{g+3}{4}-4$. Therefore, the first map can have no kernel in codimension less than $\frac{g+3}{4}-4$.

Finally,
in \cite[Lemmas 5.5 and 5.8]{part1}, we showed that
$\H_{4,g}^\circ$ and $\H_{4,g}^{\nf}$ are ``good approximations of each other" in the sense that the codimension the complement of $\H^\circ_{4,g} \cap \H^{\nf}_{4,g} \subset \H^{\nf}_{4,g}$
and of $\H^\circ_{4,g} \cap \H^{\nf}_{4,g} \subset \H^\circ_{4,g}$ are both
at least $\frac{g+3}{4} - 4$. Therefore, by excision there is an isomorphism $A^i(\H_{4,g}^\circ) \cong A^i(\H_{4,g}^{\nf})$.
In particular, we have $\dim A^i(\H_{4,g}^{\nf}) = \dim A^i(\H_{4,g}^\circ) = \dim R^i(\H_{4,g})$.
The calculation of  $\dim R^i(\H_{4,g})$
follows from Corollary \ref{rcor}.
\end{proof}


\section{The Chow ring in degree $5$} \label{re5}
\subsection{Set up}
We begin by recalling the linear algebraic data associated to degree $5$ covers, as developed by Casnati \cite{C}. For more details in our context, see \cite[Section 3.3]{part1}. To a degree $5$, cover $\alpha: C \to \pp^1$, we again associate two vector bundles on $\pp^1$:
\[E_\alpha := (\alpha_*\O_C/\O_{\pp^1})^\vee = \ker(\alpha_*\omega_\alpha \to \O_{\pp^1}) \qquad \text{and} \qquad F_\alpha := \ker(\Sym^2 E_\alpha \to \alpha_*\omega_{\alpha}^{\otimes 2}).\]
If $C$ has genus $g$, then $E_\alpha$ has degree $g+4$, and rank $4$, while $F_\alpha$ has degree $2g+8$ and rank $5$.
Geometrically, the curve $C$ is embedded in $\gamma:\p( E^{\vee}_{\alpha}\otimes \det E_{\alpha})\rightarrow \p^1$, which further maps to $\p(\wedge^2 F_{\alpha})$ via an associated section
\[
\eta\in H^0(\p^1,\H om(E_{\alpha}^{\vee}\otimes \det E_{\alpha},\wedge^2 F_{\alpha})).
\]
The curve $C$ is obtained as the intersection of the image of $\p( E^{\vee}_{\alpha}\otimes \det E_{\alpha})$ with the  Grassmann bundle $G(2,F_{\alpha}) \subset \pp(\wedge^2 F_\alpha)$.

Conversely, suppose we are given a rank $4$, degree $g+4$ vector bundle $E$ and a rank $5$, degree $2g + 8$ vector bundle $F$ on $\pp^1$. We write $E' := E^\vee \otimes \det E$ and $\gamma: \pp E' \to \pp^1$. We characterize which sections $\eta$ fail to produce a smooth degree $5$, genus $g$ cover.
Let
\[\Phi : H^0(\pp^1, \H om(E^\vee \otimes \det E, \wedge^2 F) \xrightarrow{\sim} H^0( \pp E', \gamma^* \wedge^2 F \otimes \O_{\pp E'}(1)). \]

\begin{lem} \label{whats_sing}
Let $E$ and $F$ be as above, with $\H om(E', \wedge^2 F)$ globally generated.
Suppose we have a map $\eta: E' \to \wedge^2 F$.
\begin{enumerate}
    \item If $\eta$ is not injective on fibers then the subscheme $D(\Phi(\eta)) \subset \pp E'$ cut by the $4 \times 4$ Pfaffians of $\Phi(\eta)$ is not smooth of dimension $1$.
    \item  If $\eta: E' \rightarrow \wedge^2 F$ is injective on fibers, the intersection $C = \eta(\pp E') \cap G(2,F)$ fails to be a smooth, irreducible genus $g$, degree $5$ cover of $\pp^1$ if and only if there exists $p \in C$ so that $\dim T_p C \geq 2$.
\end{enumerate}

\end{lem}
\begin{proof}
(1) Suppose $\eta(e_1) = 0$ for $e_1$ a vector in the fiber of $E'$ over $0 \in \pp^1$, where $\pp^1$ has coordinate $t$. We can choose coordinates $X_1, X_2, X_3, X_4$ on $\pp E'$ so that $\mathrm{span}(e_1) \in \pp E'|_0 \subset \pp E'$ is defined by vanishing of $t$ and $X_2, X_3, X_4$. Since $\eta(e_1)$ vanishes at $t = 0$, all entries of a matrix representative $M_\eta$ for $\Phi(\eta)$ as in \cite[Equation 5.13]{part1} would have coefficient of $X_1$ divisible by $t$. In particular, the quadrics $Q_i$ that define the Pfaffian locus $C = D(\Phi(\eta))$ of $\eta$ lie in the ideal $(t) + (X_2, X_3, X_4)^2$. Hence, $T_p C$ contains the entire vertical tangent space of $\pp E' \to \pp^1$, and therefore has dimension at least $3$.

(2) If $\eta(\pp E') \cap G(2, F) \subset \pp(\wedge^2 F)$ is connected, or has a component of dimension $\geq 2$, then we are done, so we suppose $\dim C = 1$. The general fiber of $C$ over $\pp^1$ consists of $5$ points. The global generation of $\H om(E', \wedge^2 F)$ implies all summands of $E$ have positive degree, so $h^0(\pp^1, E^\vee) = 0$. Hence, if $C$ has the right codimension in each fiber, then $h^0(C, \O_C) = h^0(\pp^1, E^\vee) + 1 = 1$ so $C$ is connected.

Now suppose that $C$ has a component $C_0$ that is contained in a fiber. We claim $C$ is connected (and thus has a two dimension tangent space at some point on $C_0$). Suppose that the fiber over $x \in \pp^1$ is the union of a one dimensional component $C_0$ together with a finite scheme $\Gamma$. The image $\eta(\pp E'|_x)$ is the intersection of 
six hyperplanes $H_i$ in the fiber $\pp(\wedge^2 \F)|_x \cong \pp^9$.
Thus the fiber of $C$ over $x$ is the intersection of six hyperplanes $H_i$ and the Grassmannian $G(2,F|_x)$ in its Pl\"ucker embedding. Because the Pl\"ucker embedding is nondegenerate, we can arrange it so that $H_1 \cap \cdots \cap H_5 \cap G(2, F|_x)$ has pure dimension $1$, i.e. 
the excess dimension appears only after intersecting with $H_6$; see \cite[Section 13.3.6]{EH} for a similar argument due to Vogel. 

To obtain the excess component $C_0$ in the final intersection, we must have that
\[H_1 \cap \cdots \cap H_5 \cap G(2, F|_x) = C_0\cup \Phi\]
with $C_0 \subset H_6$.
Note that the reducible curve $C_0\cup \Phi$ must have degree $5 = \deg G(2, F|_x)$, so each component has degree at most $4$. Therefore, the finite scheme $\Gamma = \Phi \cap H_6$ has degree at most $4$. Because the general fiber of $C$ over $\pp^1$ consists of a degree five zero dimensional subscheme, it follows that some of the five points in the general fiber must specialize into $C_0$, and the intersection $\eta(\pp E') \cap G(2,F)$ is singular there.
\end{proof}

The association of $\alpha:C\rightarrow \p^1$ with the pair $(E_{\alpha}, F_{\alpha})$ gives rise to a map $\H_{5,g}\rightarrow \B_{5,g}$, where $\B_{5,g}$ is the moduli stack of pairs of vector bundles on $\p^1$-bundles, as defined in \cite[Definition 5.10]{part1}. Let $\pi: \P \to \B_{5,g}$ be the universal $\pp^1$-bundle and let $\E$ be the universal rank $4$ vector bundle on $\P$.
Continuing the notation of \cite{part1}, let $z = c_1(\O_{\P}(1))$ and define 
classes $a_i, b_i \in A^i(\B_{5,g})$ and $a_i', b_i' \in A^{i-1}(\B_{5,g})$ by the formula
\[c_i(\E) = \pi^*a_i + \pi^*a_i'z \qquad \text{and} \qquad c_i(\F) = \pi^*b_i + \pi^* b_i' z.\]
(Note that there is a ``determinant compatibility condition" which implies $2a_1 = b_1$, see \cite[p. 25]{part1}.)
We also define $c_2 = -\pi_*(z^3) \in A^2(\B_{5,g})$, which is the pullback of the universal second Chern class on $\BSL_2$. 

By \cite[Equation 5.10]{part1}
\begin{gather} \label{b5fact}
\text{$a_1, a_2, a_3, a_4, a_2', a_3', a_4', b_2, b_3, b_4, b_5, b_2', b_3', b_4', b_5', c_2$ generate $A^*(\B_{5,g})$ and satisfy} \\
\text{no relations in degrees up to $g + 4$}. \notag
\end{gather}

We call the pullbacks of $\E$ and $\F$ to $\H_{5,g}$ the \emph{CE bundles}, just like in the degree $4$ case. Similarly, the pullbacks of the classes appearing in \ref{b5fact} to $\H_{5,g}$ are called the \emph{CE classes}. By \cite[Theorem 3.10]{part1}, the CE classes are tautological and generate the tautological ring.

In order to prove Theorem \ref{main}(3), we proceed in two steps, just like we did in degree $4$. First, in Section \ref{construction}, we construct a certain bundle of principal parts, and it to find relations among the CE classes. In Section \ref{allrel}, we define an open substack $\H^{\circ}_{5,g}\subset \H_{5,g}$, which is an open substack of a vector bundle over $\B^\circ_{5,g}\subset \B_{5,g}$, and use it to demonstrate that we have found all relations in degrees up to roughly $g/5$. Just like in degree $4$, the method is summarized by Figure \ref{sumfig}, but this time there are no factoring covers, so one can ignore the top row.

\subsection{The construction of the bundle of principal parts and relations}\label{construction}
In this section, we will perform a construction that starts with the data $(\P\rightarrow B, \E, \F, \eta)$ associated to degree $5$ covers
and produces a vector bundle called $RQ^1_{\p\E'/B}(\W')$ whose sections help us detect when the associated subscheme $D(\Phi(\eta)) \subset \pp \E'$ defined by the vanishing of Pfaffians fails to be smooth of relative dimension $1$ over $B$. The formation of this bundle commutes with base change. We will use this construction to produce relations among CE classes in the Chow ring of $\H_{5,g}$. 

Suppose we are given the data $(\P\rightarrow B, \E, \F, \eta)$ where $\P\rightarrow B$ is a $\p^1$-bundle, $\E$ is a rank $4$ vector bundle on $\P$, $\F$ is a rank $5$ vector bundle on $\P$, and $\eta\in H^0(\P,\H om(\E^{\vee}\otimes \det \E,\wedge^2\F))$. Set $\E'=\E^{\vee}\otimes \det \E$. Furthermore, we will assume that $\eta:\E'\rightarrow \wedge^2\F$ is injective with locally free cokernel. It thus induces an inclusion $\p\eta:\p\E'\rightarrow \p(\wedge^2 \F)$.

To set up this construction, let
$\mathcal{Y} := G(2, \F) \times_{\P} \pp \E'$ and let
$p_1: \mathcal{Y} \to G(2, \F)$ and $p_2: \mathcal{Y} \to \pp \E'$ be the projection maps, so we have the diagram below.
\begin{center}
    \begin{tikzcd}
\X \arrow[dd, bend right = 90, looseness=2, swap, "p_2"] \arrow[r, "p_1"] \arrow[d, hook] & {G(2,\F)} \arrow[d, hook, "i"] \\
{\small \pp \E' \times_{\P} \pp(\wedge^2 \F)} \arrow{d}[swap]{q_2} \arrow{r}{q_1} & \p(\wedge^2 \F) \arrow{d}{\epsilon}\\
\p\E' \arrow[r, "\gamma"']           & \P \arrow[d, "\pi"]                               &                 \\
                                     & {B}                                &                
\end{tikzcd}
\end{center}
These spaces come equipped with tautological sequences, which we label as follows. On $G(2, \F)$, we have an exact sequence
\[
0\rightarrow \mathcal{T}\rightarrow i^*\epsilon^*\F\rightarrow \mathcal{R}\rightarrow 0,
\]
where $\T$ is rank $2$ and $\mathcal{R}$ is rank $3$.
Meanwhile, on $\p(\wedge^2 \F)$, we have an exact sequence
\begin{equation} \label{pftaut}
0 \rightarrow \O_{\pp(\wedge^2 \F)}(-1) \rightarrow \epsilon^* (\wedge^2 \F) \rightarrow \U_9 \rightarrow 0
\end{equation}
where $\U_9$ is the tautological rank $9$ quotient bundle. Noting that the Pl\"ucker embedding satisfies $i^*\O_{\p(\wedge^2 \F)}(-1)=\wedge^2 \mathcal{T}$, the restriction of \eqref{pftaut} to $G(2, \F)$ takes the form
\begin{equation} \label{u9}
0\rightarrow \wedge^2 \mathcal{T}\rightarrow i^*\epsilon^* (\wedge^2 \F)\rightarrow i^*\mathcal{U}_9\rightarrow 0.
\end{equation}
It follows that the map $i^*\epsilon^*(\wedge^2 \F) \rightarrow \wedge^2 \mathcal{R}$ descends to a map 
\begin{equation} \label{uquo}
i^*\U_9  \to \wedge^2 \mathcal{R}.
\end{equation}
\begin{rem} \label{TtoN}
The tensor product of \eqref{uquo} with $i^*\O_{\pp(\wedge^2 \F)}(1)$ is the natural map from the restriction of the tangent bundle to the normal bundle, $i^* T_{\pp(\wedge^2 \F)} \to N_{G(2, \F)/\pp(\wedge^2 \F)}$.
\end{rem}

We define
\[\W := \H om(\O_{\p\E'}(-1),\gamma^*(\wedge^2\F)) = \O_{\pp \E^\vee}(1) \otimes \gamma^*(\wedge^2 \F) \otimes \det \E,\]
which is a rank $10$ vector bundle on $\pp \E'$. The composition
\[\O_{\pp \E'}(-1) \to \gamma^* \E' \xrightarrow{\gamma^* \eta} \gamma^*(\wedge^2 \F)\]
defines a section $\delta$ of $\W$.
Pulling back to
$\pp \E' \times_{\P} \pp(\wedge^2 \F)$,
consider the further composition
\begin{equation} \label{graph}
q_2^*\O_{\pp \E'}(-1) \to q_2^*\gamma^* \E' \xrightarrow{q_1^*\epsilon^* \eta} q_1^*\epsilon^*(\wedge^2 \F) \rightarrow q_1^*\U_9.
\end{equation}
The vanishing locus of this composition is precisely the graph of $\pp \eta$ inside $\pp \E' \times_{\P} \pp(\wedge^2 \F)$. Restricting \eqref{graph} to $\mathcal{Y}$, we obtain a section, which we call $\overline{\delta}$, of the rank $9$ vector bundle
\[\W' := \H om(p_2^*\O_{\pp \E'}(-1), p_1^*i^*\U_9).\]
The vanishing $V(\overline{\delta}) \subset \mathcal{Y}$ is the intersection of the graph of $\pp \eta$ with $\mathcal{Y}$ and is therefore identified with the intersection $G(2, \F) \cap \pp\eta (\pp \E')$. Viewed inside $\pp \E'$, this intersection is equal to the desired associated subscheme $D(\Phi(\eta)) \subset \pp \E'$. 

\begin{rem}
The subscheme $D(\Phi(\eta)) \subseteq \pp \E'$ is not in general the zero locus of a section of a vector bundle. However, we have found how to realize this scheme as
the zero locus of a section of a vector bundle on $\mathcal{Y}$, basically by using the fact that the graph of $\pp \eta$ is defined by the zero locus of a section of a vector bundle.
\end{rem}

Next, we are going to construct a certain restricted principal parts bundle from $\W'$ that will detect when fibers of $\C = V(\overline{\delta}) \to B$ have vertical tangent space of dimension $2$ or more. 
Before giving the construction, let us describe the geometric picture on a single fiber $\pp^1$ of $\P \to B$.
Let $E$ and $F$ be vector bundles on $\pp^1$ of ranks $4$ and $5$ respectively and suppose $\eta: E' \to \wedge^2 F$ is an injection of vector bundles with locally free cokernel. 
Let $p \in \pp E'$. The intersection $G(2, F) \cap \eta (\pp E')$ has a two dimensional tangent space at $\eta(p) \in G(2, F)$ if and only if there exists a two dimensional subspace $S \subset T_p\pp E'$ such that the differential of the projectivization of $\eta$ sends $S$ into $T_q G(2, F) \subset T_q \pp (\wedge^2 F)$.
Equivalently, the composition $S \subset T_p\pp E' \xrightarrow{\mathrm{d} \pp \eta} T_{\eta(p)} \pp(\wedge^2 F) \to N_{G(2, F)/\pp(\wedge^2 F)}|_{\eta(p)}$ is zero (see Figure \ref{detafig}).

\begin{figure}[h!]
\includegraphics[width=3.5in]{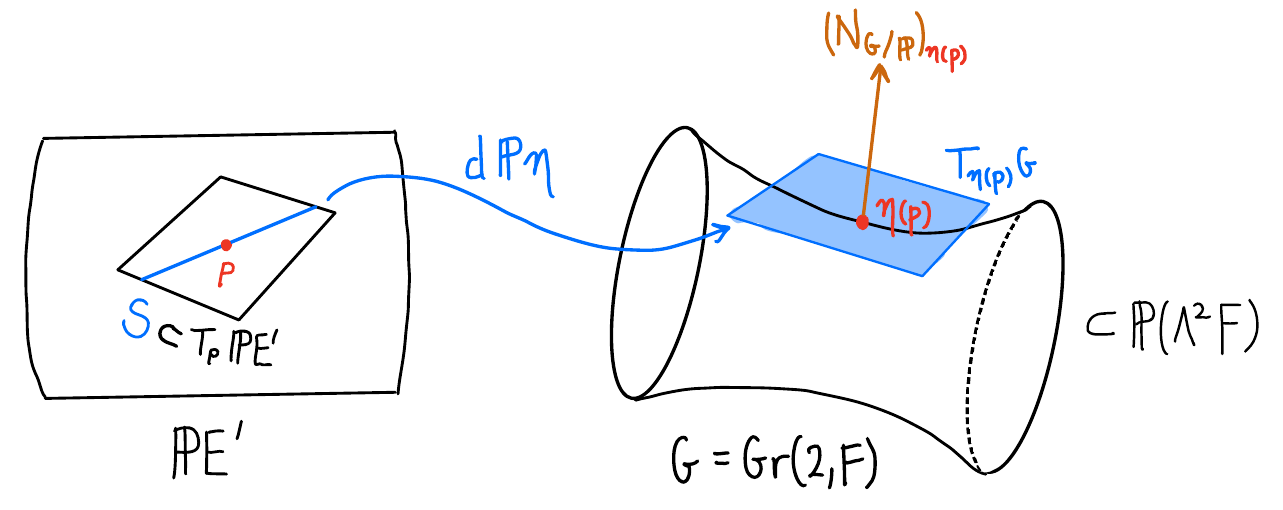}
\caption{Does $\mathrm{d} \pp \eta$ send $S$ into $T_{\eta(p)}G$?}
\label{detafig}
\end{figure}

First consider
$Q^1_{\pp \E'/B}(\W')$ (see Definition \ref{Qdef}), which comes equipped with a filtration
\begin{equation} \label{QW}
0 \rightarrow p_2^* \Omega_{\pp \E'/B} \otimes \W' \to Q^1_{\pp \E'/B}(\W') \rightarrow \W' \rightarrow 0.
\end{equation}
Given any section $\delta$ of $\W$, there is an induced section of $Q^1_{\pp \E'/B}(\W')$, which records the values and first order changes of the induced section $\overline{\delta}$ of $\W'$ as we move across $\pp \E'$.
Now let $\widetilde{\X} := G(2,p_2^*T_{\p\E'/B})\xrightarrow{a}\X$, which comes equipped with a tautological sequence
\[
0\rightarrow \Omega_x^\vee \rightarrow a^*p_2^*T_{\p\E'/B}\rightarrow \Omega_y^\vee \rightarrow 0,
\]
where $\Omega_x$ and $\Omega_y$ are both rank $2$.
Dualizing the left map gives
\begin{equation} \label{q1}
a^* p_2^* \Omega_{\pp \E'/B} \to \Omega_x.
\end{equation}
Meanwhile, tensoring the $p_1^*$ of \eqref{uquo} with $p_2^* \O_{\pp \E'}(1)$, we have a quotient 
\begin{equation} \label{wpq}
\W' \to p_2^* \O_{\pp \E'}(1) \otimes p_1^*(\wedge^2 \mathcal{R}).
\end{equation}
\begin{rem} \label{w2r}
If one has an injection $\eta: \E' \to \wedge^2 \F$, then one has an isomorphism of
 $p_2^*\O_{\pp \E'}(1)$ with $p_1^*i^*\O_{\pp(\wedge^2 \F)}(1)$ on $V(\overline{\delta})$ (coming from \eqref{graph}).
 By Remark \ref{TtoN}, the restriction of \eqref{wpq} to $V(\overline{\delta})$ then agrees with the restriction of
 $p_1^*i^* T_{\pp(\wedge^2 \F)} \to p_1^*N_{G(2, \F)/\pp(\wedge^2 \F)}$ to $V(\overline{\delta})$. This was the geometric intuition behind the definition we are about to make.
\end{rem}

Pulling back \eqref{wpq} to $\widetilde{\X}$ and tensoring with \eqref{q1}, we obtain a quotient
\begin{equation} \label{q2} a^*(p_2^*\Omega_{\pp \E'/B} \otimes \W') \to \Omega_{x}\otimes a^* (p_2^*\O_{\pp \E'}(1) \otimes p_1^*(\wedge^2 \mathcal{R})).
\end{equation}
Note that the term on the left of \eqref{q2} is the $a^*$ of the term on the left of \eqref{QW} (the ``derivatives part" of the principal parts bundle).
Let $RQ^1_{\pp\E'/B}(\W')$ be the quotient of $a^* Q^1_{\pp \E'/B}(\W')$ by the kernel of \eqref{q2}. This bundle comes equipped with a filtration
\begin{equation} \label{rqfilt}0 \rightarrow \Omega_x \otimes a^* (p_2^*\O_{\pp \E'}(1) \otimes p_1^* (\wedge^2 \mathcal{R})) \to RQ^1_{\pp \E'/B}(\W') \to \W' \rightarrow 0
\end{equation}
and has rank $15$. The bundle $RQ^1_{\pp \E'/B}(\W')$ remembers derivatives just in the ``$x$-directions" (i.e. along a distinguished $2$-plane) and remembers their values under the quotient \eqref{wpq}. Considering Remark \ref{w2r} and Figure \ref{detafig}, this is telling us to what extent vectors in the subspace $S$ corresponding to ``$x$-directions" leave $T_{\eta(p)}G(2, F)$. 
This will be spelled out in local coordinates in the lemma below.

The global section $\overline{\delta}$ of $\W'$ induces a global section $\overline{\delta}'$ of $Q^1_{\pp \E'/B}(\W')$, which in turn gives rise to a global section $\overline{\delta}''$ of $RQ^1_{\pp \E'/B}(\W')$. The following lemma describes the geometric condition for such an induced section to vanish at a geometric point of $\widetilde{\X}$.

\begin{lem} \label{geo}
Let
 $E$ and $F$ be vector bundles on $\pp^1$ of ranks $4$ and $5$ respectively. 
Let $Y = \pp E' \times_{\pp^1} G(2, F)$ and let $W$, $W'$, $R$, $Q^1_{\pp E'}(W')$ and $RQ^1_{\pp E'}(W')$ be defined analogously to the constructions above (working over a point instead of $B$). 
Suppose $\eta: E' \to \wedge^2 F$ is an injection of vector bundles.
Then the following are true:
\begin{enumerate}
    \item  The induced section $\overline{\delta}$ of $W'$ corresponding to $\eta$ vanishes at $(p, q) \in Y$ if and only if the projectivization of $\eta$ sends $p$ to $q$.
    \item The induced section $\overline{\delta}''$ of $RQ^1_{\pp E'}(W')$ corresponding to $\eta$ vanishes at $(p, q, S) \in \widetilde{Y}$ if and only if the differential of the projectivization of $\eta$ sends the subspace $S \subset T_p \pp E'$ into the subspace $T_q G(2, F) \subset T_q \pp(\wedge^2 F)$.
\end{enumerate} 
Hence, given any family $(\P \to B, \E, \F, \eta)$, the image of the vanishing of the induced section $\overline{\delta}''$ of $RQ^1_{\pp \E'}(\W')$ is the locus in $B$ over which fibers of $D(\Phi(\eta)) \to B$ fail to be smooth of relative dimension $1$.
\end{lem}
\begin{proof}
(1) Let $t$ be a coordinate on $\pp^1$, and let $p \in \pp E'$ and $q \in G(2, F)$ be points in the fiber over $0 \in \pp^1$.
To say $\eta$ sends $p$ to $q$ is to say that $\eta$ sends the subspace of $E'|_0$ corresponding to $p$ into the subspace of $\wedge^2 F|_0$ corresponding to $q$. 
Hence, by the definition of the tautological sequences, $\eta$ sends $p$ to $q$ if and only if the composition \[p_2^* \O_{\pp E'}(-1) \to p_2^* \gamma^* E' \to p_1^*i^*\epsilon^*(\wedge^2 F) \to p_1^*i^*U_9\]
vanishes at $(p, q)$, which is to say $\overline{\delta}$ vanishes.

(2) Trivializing $E$ and $F$ over an open $0 \in U \subset \pp^1$, we may choose a basis $e_1, \ldots, e_4$ for $E$ so that $p = \mathrm{span}(e_1)$ and a basis $f_1, \ldots, f_5$ for $F$ so that $q = \mathrm{span}(f_1 \wedge f_2)$.
 Let $\eta_{k, ij}$ be the coefficient of $f_i \wedge f_j$ in $\eta(e_k)$, so $\eta_{k, ij}$ is a polynomial in $t$.
In these local coordinates, to say $\eta$ sends $p$ to $q$ is to say that $\eta_{1, ij}|_{t=0} = 0$ for $ij \neq 12$.

The map $p_1^*\wedge^2 F \rightarrow \wedge^2 R$ corresponds to projection onto the span of $f_3 \wedge f_4, f_3 \wedge f_5,$ and $f_4 \wedge f_5$.
If $\eta$ sends $p$ to $q$, then the induced section $\overline{\delta}$ of $W'$ already vanishes. Therefore, the value of $\overline{\delta}''$ at $(p, q)$
lands in the subbundle 
$p_2^*\Omega_{\pp E'/B} \otimes p_2^* \O_{\pp E'}(1) \otimes p_1^*(\wedge^2 R) \subset Q^1_{\pp E'}(W')$. This ``value" of $\overline{\delta}''$ at $(p, q)$ records the first order information of $\eta_{1, ij}$ for $ij = 34, 35, 45$ as $p$ deforms.

First order deformations of $p$ are of the form $\mathrm{span}(e_1) \mapsto \mathrm{span}(e_1 + \epsilon(ae_2 + be_3 + c e_4))|_{t = \epsilon d}$, where $\epsilon^2 = 0$. Here, $a, b, c, d$ are coordinates on the tangent space at $p$ ($a, b, c$ are vertical coordinates and $d$ is the horizontal coordinate).
The coefficient of $f_i \wedge f_j$ in $\eta(e_1 + \epsilon(ae_2 + be_3 + ce_4))|_{t = \epsilon d}$ is
\begin{equation} \label{abg}
\eta_{1,ij} + \left(d \left. \left(\frac{d}{dt}\eta_{1,ij}\right)\right|_{t=0}  + a \eta_{2,ij}|_{t=0} + b \eta_{3,ij}|_{t=0} + c\eta_{4,ij}|_{t=0}\right)\epsilon \qquad \text{for $ij = 34, 35, 45$}.
\end{equation}
Locally, $a \epsilon, b\epsilon, c\epsilon, d\epsilon$ are our basis for $\Omega_{\pp E}$ and $f_i \wedge f_j$ for $ij = 34, 35, 45$ is our basis for $\wedge^2 R$. The ``value" we wish to extract in the fiber of $p_2^*\Omega_{\pp E'/B} \otimes p_2^* \O_{\pp E'}(1) \otimes p_1^*(\wedge^2 R)$ over $(p, q)$ is the coefficients of $a \epsilon, b \epsilon, c\epsilon,$ and $d\epsilon$ in \eqref{abg} for $ij = 34, 35, 45$.

Now suppose $\eta$ is injective on fibers, so $\pp \eta$ is well-defined. In particular, $\eta_{1,12}|_{t=0} \neq 0$. With respect to $a,b,c,d$ the differential of $\pp \eta$, from $T_p \pp E' \rightarrow T_q \pp (\wedge^2 F)$, is represented by a $9 \times 4$ matrix
\begin{equation} \label{diff}
   \frac{1}{\eta_{1,12}|_{t=0}} \left(\begin{matrix}
   \frac{d}{dt} \eta_{1,13}|_{t=0} &
    \eta_{2,13}|_{t=0} & \eta_{3,13}|_{t=0} & \eta_{4,13}|_{t=0} \\
    \frac{d}{dt} \eta_{1,14}|_{t=0} &
    \eta_{2,14}|_{t=0} & \eta_{3,14}|_{t=0} & \eta_{4,14}|_{t=0} \\
    \vdots & & & \vdots \\
    \frac{d}{dt} \eta_{1,45}|_{t=0} &
    \eta_{2,45}|_{t=0} & \eta_{3,45}|_{t=0} & \eta_{4,45}|_{t=0}
    \end{matrix}\right).
\end{equation}
The subspace $T_q G(2, F) \subset T_q\pp(\wedge^2 F)$ corresponds to the first $6$ coordinates. (A first order deformation of $f_1 \wedge f_2$ remains a pure wedge to first order if and only if the $f_i \wedge f_j$ with non-zero coefficient in the deformation have one of $i,j$ is equal to $1$ or $2$. See also Remark \ref{w2r}.)
Thus, $\pp \eta$ sends $T_p\pp E'$ into $T_qG(2, F)$ if and only if the bottom three rows of \eqref{diff} vanish, which occurs if and only if the coefficients of $a,b,c,d$ in \eqref{abg} vanish.
More generally, a tangent vector in $T_p \pp E'$
is sent into $T_q G(2, F)$ if and only if 
\eqref{abg} vanishes (for $ij = 34, 35, 45$) when the corresponding values of $a,b,c,d$ are plugged in.  Plugging in values for $a,b,c,d$ in a given two dimensional subspace $S$ of $T_p\pp E'$ then corresponds to the ``value" of $\eta$ in $S^\vee \otimes p_2^* \O_{\pp E'}(1) \otimes p_1^*(\wedge^2 R)$ over $(p, q)$. By the filtration \eqref{rqfilt}, this ``value" is zero if and only if $\overline{\delta}''$ vanishes at $(p, q, S) \in \widetilde{X}$.

Since the formation of these (refined) principal parts bundles commutes with base change, the claim regarding families follows.
\end{proof}

We now apply the above construction in the case $B=\H_{5,g}$ and $\eta = \eta^{\mathrm{univ}}$, the section associated to the universal cover $\C \to \P$. By Lemma \ref{geo} and the fact that the universal curve $\mathcal{C} = V(\overline{\delta}'')$ is smooth of relative dimension $1$ over $\H_{5,g}$, the global section $\overline{\delta}''$ of $RQ^1_{\pp \E'/\H_{5,g}}(\W')$ is nowhere vanishing. We therefore have the following lemma, which gives a source of relations among the CE classes on $\H_{5,g}$.
\begin{lem} \label{rels5}
Let $z = c_1(\O_{\P}(1)), \zeta = c_1(\O_{\pp \E'}(1)), \sigma_{i} = c_i(\mathcal{R}),$ and $s_i = c_i(\Omega_y^{\vee})$. All classes of the form (some pullbacks omitted for ease of notation):
\[a_* p_{2*} \gamma_* \pi_* (c_{15}(RQ^1_{\pp \E'/\H_{5,g}}(\W')) \cdot  s_1^{l_1} s_2^{l_2} \sigma_1^{k_1} \sigma_2^{k_2}\sigma_3^{k_3} \zeta^j z^i)\]
are zero in $R^*(\H_{5,g}) \subseteq A^*(\H_{5,g})$.
\end{lem}
\subsection{All relations in low codimension} \label{allrel}
We recall the construction of an open substack $\H_{5,g}^{\circ}\subset \H_{5,g}$ and what we already know about its Chow ring from \cite{part1}.
We start with $\B_{5,g}$, the moduli space of pairs of vector bundles $E$ of rank $4$, degree $g+4$ and $F$ of rank $5$, degree $g+5$ on $\pp^1$-bundles together with an isomorphism of $\det E^{\otimes 2}$ and $\det F$ (see \cite[Section 5.3]{part1}). Let $\E$ and $\F$ be the universal bundles on $\pi:\P\rightarrow \B_{5,g}$ and let $\gamma: \pp \E^\vee \to \P$ be the structure map. Define $\aV_{5,g}:=\H om(\E^{\vee}\otimes \det \E,\wedge^2 \F)$.
 We consider an open substack $\B^{\circ}_{5,g}\subset \B_{5,g}$, defined by a certain positivity condition for the bundle $\aV_{5,g}$
    \begin{equation} \label{defb5c}
    \B^{\circ}_{5,g}:=\B_{5,g}\smallsetminus \Supp R^1\pi_*(\aV_{5,g}(-2)).
    \end{equation}
    Let $\H^{\circ}_{5,g}$ denote the base change of $\H_{5,g}\rightarrow \B_{5,g}$ along the open embedding $\B_{5,g}^{\circ}\hookrightarrow \B_{5,g}$.
    
Over $\B^{\circ}_{5,g}$, we see that $\mathcal{X}_{5,g}^{\circ}:=\pi_*\aV_{5,g}|_{\B^{\circ}_{5,g}}$ is a vector bundle whose fibers correspond to sections of $\aV_{5,g}$. 
The open $\H_{5,g}^\circ$ is contained in the open $\H_{5,g}'$ of 
\cite[Lemma 5.11]{part1}, so that lemma implies $\H^{\circ}_{5,g} \to \B_{5,g}^\circ$ 
 factors through an open embedding in $\mathcal{X}^{\circ}_{5,g}$. We define
    \[
    \Delta_{5,g}:=\mathcal{X}^{\circ}_{5,g}\smallsetminus \H^{\circ}_{5,g},
    \]
represented in red in the middle column of Figure \ref{sumfig}.
Now we wish to use excision
to determine the Chow ring of $\H_{5,g}^\circ$ in degrees up to $\frac{g+4}{5} - 16$. We already understand $A^*(\mathcal{X}^\circ_{5,g}) \cong A^*(\B_{5,g}^\circ)$ in degrees up to $\frac{g+4}{5} - 16$ by \cite[Equation 5.11]{part1}.
Lemma \ref{whats_sing} says we need to remove the locus of non-injective maps and the locus of injective maps such that the induced intersection of $\p\E'$ and $G(2,\F)$ has a singular point. 


We begin by computing the relations obtained from removing the locus of non-injective maps $\E' \rightarrow \wedge^2 \F$, i.e. maps that drop rank along some point on $\P$. Consider the projective bundle $\gamma:\p \E'\rightarrow \P \to \B_{5,g}^\circ$, and let $\W := \O_{\p\E'}(1)\otimes \gamma^*(\wedge^2 \F)$. We have that $\gamma_* \W = \H om(\E', \wedge^2 \F) = \aV_{5,g}$, so by the definition of $\B_{5,g}^\circ$ (see \eqref{defb5c}) and Lemma \ref{famjet}, the map
\begin{equation} \label{wjets}\gamma^*\pi^* \aW_{5,g}^\circ = \gamma^* \pi^*\pi_* \gamma_* \W \to P^1_{\pp \E'/\B_{5,g}^\circ}(\W) \qquad \text{is surjective.}
\end{equation}
Composing with the surjection $P^1_{\pp \E'/\B_{5,g}^\circ}(\W) \to \W$, we obtain a surjection $\gamma^* \pi^* \aW_{5,g}^\circ \to \W$, whose kernel we define to be $\widetilde{\aW}^{\mathrm{ni}}$. The fiber of $\widetilde{\aW}^{\mathrm{ni}}$ at a point $p \in \pp \E'$ corresponds to maps of $\E' \to \wedge^2 \F$ (on the fiber over $\pi(\gamma(p))$) whose kernel contains the subspace referred to by $p$.

We then have the following trapezoid diagram:
\[
\begin{tikzcd}
\tilde{\aW}^{\mathrm{ni}} \arrow[r, hook] \arrow[rd] & {\gamma^*\pi^*\aW_{5,g}^\circ} \arrow[d] \arrow[r] & {\pi^*\aW_{5,g}^\circ} \arrow[d] \arrow[r] & {\aW_{5,g}^\circ} \arrow[d] \\
                                          & \p\E' \arrow[r, "\gamma"']                 & \P \arrow[r, "\pi"']               & {\B^{\circ}_{5,g}} 
\end{tikzcd}
\]
Thus, Lemma \ref{trap} yields:
\begin{prop}\label{niimage}
The image of the pushforward map $A_*(\tilde{\aW}^{\mathrm{ni}})\rightarrow A_*(\aW_{5,g})$ is equal to the ideal generated by
\begin{equation} \label{wni}
\rho^*\pi_*\gamma_*(c_{10}(\W))\cdot \zeta^jz^i), \qquad 0 \leq j \leq 3, \ 0 \leq i \leq 1.
\end{equation}
where $\zeta = c_1(\O_{\pp \E'}(1))$ and $z = c_1(\O_{\P}(1))$.
\end{prop}
Next, we excise the locus of injective maps such that the induced intersection of $\p\E'$ and $G(2,\F)$ has a singular point.
From the construction in Section \ref{construction} applied to the case $B=\B^{\circ}_{5,g}$, we have a rank $15$ vector bundle $RQ^1_{\pp \E'/\B^{\circ}_{5,g}}(\W')$ on $\widetilde{\X}$, which comes equipped with a series of surjections (see Lemma \ref{Qequip} for the first map; the second map comes from the construction of $RQ^1_{\pp \E'/\B^{\circ}_{5,g}}(\W')$, which was made just after \eqref{q2}):
\begin{equation} \label{ptor} a^*p_1^*P^1_{\pp \E'/\B_{5,g}^\circ}(\W) \to a^* Q^1_{\pp \E'/\B_{5,g}^\circ}(\W') \to RQ^1_{\pp \E'/\B_{5,g}^\circ}(\W').
\end{equation}
Applying $a^*p_2^*$ to \eqref{wjets} and composing the result with \eqref{ptor}, we obtain a surjection
\begin{equation} \label{to15}a^*p_2^*\gamma^*\pi^* \aW_{5,g}^\circ \to RQ^1_{\pp \E'/\B_{5,g}^\circ}(\W').
\end{equation}
Define $\widetilde{\Delta}_{5,g}$ to be the kernel of \eqref{to15}, so that we obtain a trapezoid diagram:
\[
\begin{tikzcd}
\widetilde{\Delta}_{5,g} \arrow[overlay,r, "i", hook] \arrow[overlay,rd, "\rho''"'] &  {\sigma^*p_2^*\gamma^*\pi^*\aW_{5,g}^\circ} \arrow[overlay,d, "\rho'"] \arrow[r] & {p_2^*\gamma^*\pi^*\aW_{5,g}^\circ} \arrow[overlay,d] \arrow[overlay,r] & {\gamma^*\pi^*\aW_{5,g}^\circ} \arrow[overlay,d] \arrow[overlay,r] & {\pi^*\aW_{5,g}^\circ} \arrow[d] \arrow[overlay,r] & {\aW_{5,g}^\circ} \arrow[overlay,d, "\rho"] \\
                                                                                                               & \tilde{\X} \arrow[overlay,r, "a"']                                         & \X \arrow[overlay,r, "p_2"']                             & \p\E' \arrow[overlay,r, "\gamma"']                 & \P \arrow[overlay,r, "\pi"']               & {\B^{\circ}_{5,g}}         
\end{tikzcd}
\]

\begin{lem} \label{dtlem}
Let $z = c_1(\O_{\P}(1)), \zeta = c_1(\O_{\pp \E'}(1)), \sigma_{i} = c_i(\mathcal{R}),$ and $s_i = c_i(\Omega_y^{\vee})$. The image of the push forward $A_*(\widetilde{\Delta}_{5,g}) \to A_*(\aW_{5,g}^\circ)$ is the ideal generated by
\begin{equation} \label{dtilde} \rho^*a_* p_{2*} \gamma_* \pi_* (c_{15}(RQ^1_{\pp \E'/\B^{\circ}_{5,g}}(\W')) \cdot  s_1^{l_1} s_2^{l_2} \sigma_1^{k_1} \sigma_2^{k_2} \sigma_3^{k_3} \zeta^j z^i)
\end{equation}
for $0 \leq j\leq 3$,  $ 0 \leq i\leq 1$,
$0 \leq l_1, l_2 \leq 2$ with $l_1+l_2\leq 2$, and $0\leq k_1,k_2,k_3\leq 2$ with $k_1+k_2+k_3\leq 2$.
\end{lem}
\begin{proof}
The monomials $s_1^{l_1} s_2^{l_2} \sigma_1^{k_1} \sigma_2^{k_2}\sigma_3^{k_3} \zeta^j z^i$ with exponents satisfying the inequalities in the statement of the lemma generate $A^*(\widetilde{\X})$ as an $A^*(\B_{5,g}^\circ)$ module (see the last paragraph of Section \ref{pandg}). The result now follows from the Trapezoid Lemma \ref{trap}. In codimension $1$, for example, since $\widetilde{\Delta}_{5,g} \to \Delta_{5,g}$ is generically one-to-one, we see
 \begin{equation} \label{5delta} [\Delta_{5,g}] = \rho^*(\pi \circ \gamma\circ p_2\circ a)_*(c_{15}(RQ^1_{\pp \E'/\B_{5,g}^{\circ}}(\W'))) = (10g+36)a_1-7a_2'-b_2'. \qedhere
\end{equation}
\end{proof}
\begin{lem} \label{asy5}
Let $I$ be the ideal generated by the classes in \eqref{wni} and \eqref{dtilde}. Then $A^*(\H_{5,g}^\circ) = A^*(\B_{5,g}^\circ)/I$. In fact, $I$ is generated by the classes in \eqref{dtilde}, so Lemma \ref{rels5} determines all relations among $CE$ classes in codimension up to $\frac{g+4}{5}-16$.
\end{lem}
\begin{proof} By Lemmas \ref{geo} and \ref{whats_sing}, we have that $\Delta_{5,g}$ is the
 union of the image of $\widetilde{\Delta}_{5,g}$ in $\aW_{5,g}^\circ$ with the image of
$\widetilde{\aW}^{\mathrm{ni}}$ in $\aW_{5,g}^\circ$. The first claim now follows from excision, the fact that push forward is surjective with rational coefficients, and Lemmas \ref{niimage} and \ref{dtlem}.

Meanwhile, direct computation \cite{github} shows that $I$ is generated by the classes in \eqref{dtilde}. Since $\rho$ is flat, the classes in \eqref{dtilde} equal the classes of Lemma \ref{rels5}.
Next, \cite[Equation 5.11]{part1} says that our generators on $\B_{5,g}^\circ$ satisfy no relations in codimension less than $\frac{g+4}{5}-16$.
Thus, we have determined all relations among CE classes in codimension up to $\frac{g+4}{5}-16$ 
\end{proof}

\subsection{Presentation of the ring and stabilization} \label{ach5}
Modulo the relations in Lemma \ref{asy5}, it turns out $R^*(\H_{5,g})$ is generated by $a_1, a_2' \in R^1(\H_{5,g})$ and $a_2, c_2 \in R^2(\H_{5,g})$, as we now explain. Let $I$ be the ideal generated by the classes in \eqref{wni} and \eqref{dtilde} in the $\qq$-algebra on the CE classes.
Using Macaulay, we determined a simplified presentation
\begin{equation} \label{simpres}
\qq[c_2, a_1, \ldots, a_4, a_2', \ldots, a_4', b_2, \ldots, b_5, b_2', \ldots, b_5']/I \cong \qq[a_1, a_2', a_2, c_2]/\langle r_1, r_2, r_3, r_4, r_5\rangle, 
\end{equation}
where
{\small
\begin{align*}
r_1 &= (1064g + 3610)a_1^3-1074a_1^2a_2' + (-2148g - 7272)a_1a_2 + 2160a_2a_2' + 
\\ &\qquad +(-1064g^3 - 10830g^2 - 36680g - 41360)a_1c_2 + (1074g^2 + 7272g + 12288)a_2'c_2 \\
r_2 &= (-6412g - 21255)a_1^3 + 6207a_1^2a_2' + (12414g + 40896)a_1a_2 + (-11880)a_2a_2'+\\ &\qquad 
+ (6412g^3 + 63765g^2 + 211540g + 234480)a_1c_2+ (-6207g^2 - 40896g - 68184)a_2'c_2 \\
r_3 &= (-22845g - 67763)a_1^4 + 18141a_1^3a_2' + (54423g + 146550)a_1^2a_2 - 35640a_1a_2a_2 \\
&\qquad  + (45690g^3 + 406578g^2 + 1184220g + 1123060)a_1^2c_2 \\
&\qquad - (54423g^2 + 293100g + 372648)a_1a_2'c_2 + (17820g + 24840)a_2'^2c_2 \\
&\qquad -(17820g + 24840)a_2^2 -(18141g^3 + 146550g^2 + 372648g + 283824)a_2c_2 \\
&\qquad  -(4569g^5 + 67763g^4 + 394740g^3 + 1123060g^2 + 1546176g + 810432)c_2^2 \\
r_4 &= 133a_1^4 -537a_1^2a_2 + (-798g^2 - 5415g - 9170)a_1^2c_2 + (1074g + 3636)a_1a_2'c_2 \\&\qquad -540a_2'^2c_2 + 540a_2^2 + (537g^2 + 3636g + 6144)a_2c_2 \\&\qquad + (133g^4 + 1805g^3 + 9170g^2 + 20680g + 17472)c_2^2 \\
r_5 &= (-18545g - 68407)a_1^4 + 15261a_1^3a_2' + (45783g + 175866)a_1^2a_2 -31320 a_1a_2a_2'\\&\qquad + (37090g^3 + 410442g^2 + 1499460g + 1811300)a_1^2c_2 \\&\qquad + (-45783g^2 - 351732g - 662976)a_1a_2'c_2 + (15660g + 72360)a_2'^2c_2 \\
&\qquad + (-15660g - 72360)a_2^2 + (-15261g^3 - 175866g^2 - 662976g - 822096)a_2c_2\\
&\qquad + (-3709g^5 - 68407g^4 - 499820g^3 - 1811300g^2 - 3260256g - 2334528)c_2^2.
\end{align*}
}

\par As a corollary of the above presentation, we can use Macaulay2 to determine a spanning set for each group $R^i(\H_{5,g})$, which is actually a basis when $g$ is sufficiently large relative to $i$. We will use these spanning sets in Section \ref{formulas5} to prove another collection of classes are additive generators.
\begin{cor}\label{5span}
Suppose $g\geq 2$.
\begin{enumerate}
    \item $R^1(\H_{5,g})$ is spanned by $\{a_1,a_2'\}$.
    \item $R^2(\H_{5,g})$ is spanned by $\{a_1^2,a_1a_2',a_2,a_2'^2,c_2\}$.
    \item $R^3(\H_{5,g})$ is spanned by $\{a_1^2a_2',a_1a_2'^2,a_1c_2,a_2a_2',a_2'c_2\}$.
    \item $R^4(\H_{5,g})$ is spanned by $\{a_1^2c_2,a_1a_2'^3,a_1a_2'c_2,a_2c_2,a_2'^4,a_2'^2,a_2'^2c_2,c_2^2\}$.
    \item $R^5(\H_{5,g})$ is spanned by $\{a_1a_2'^4,a_1c_2^2,a_2'^5,a_2'c_2^2\}$
    \item $R^6(\H_{5,g})$ is spanned by $\{a_1a_2'^5,a_2'^6,c_2^3\}$
    \item For $i\geq 7$, $R^7(\H_{5,g})$ is spanned by $\{a_1a_2'^{i-1},a_2'^i\}$.
\end{enumerate}
The above spanning set for $R^i(\H_{5,g})$ is a basis when $g > 5i + 76$.
\end{cor}
\begin{proof}
Let $S^i$ denote the degree $i$ group of the graded ring $\qq[a_1, a_2', a_2, c_2]/\langle r_1, r_2, r_3, r_4, r_5 \rangle$.
By Proposition \ref{asy5} and Equation \eqref{simpres}, $S^i$ surjects onto $R^i(\H_{5,g})$ and is an isomorphism in degrees $i < \frac{g+4}{5}-16$, equivalently when $g > 5i + 76$.

Using Macaulay, we check that the set listed in the lemma is a basis of $S^i$ for $i \leq 14$.
For $7 \leq i \leq 14$, in particular, we see that $a_2'^i$ and $a_2'^{i-1} a_1$ form a basis for the group $S^i$. For $i \geq 15$, every monomial of degree $i$ in $a_1, a_2', a_2, c_2$ is expressible as a product of two monomials, both of degree at least $7$. Then the product of two such monomials is in the span of $a_2'^i, a_2'^{i-1} a_1$ and $a_2'^{i-2} a_1^2 = a_2'^{i-7}(a_2'^5 a_1^2)$. The last monomial is already in the span of the first two because $S^7$ is spanned by $a_2'^7, a_2'^{6} a_1$.
It follows that $a_2'^i$ and $a_2'^{i-1} a_1$ span $S^i$ for all $i \geq 15$. Meanwhile, no monomial of the form $a_2'^i$ or  $a_2'^{i-1} a_1$ appears in the relations $r_1, \ldots, r_5$. Hence, no combination of $a_2'^i$ and $a_2'^{i-1} a_1$ lies in $\langle r_1, \ldots, r_5 \rangle$, so $a_1a_2'^{i-1}$ and $a_2'^i$ are independent for all $i$.
\end{proof}

\begin{proof}[Proof of Theorem \ref{main}(3)]
Consider the equation
\[\frac{\qq[a_1, a_2', a_2, c_2]}{\langle r_1, r_2, r_3, r_4, r_5 \rangle} \rightarrow R^*(\H_{5,g}) \rightarrow R^*(\H_{5,g}^\circ) = A^*(\H_{5,g}^\circ).\]
The first map exists and is surjective by Proposition \ref{rels5}. Meanwhile, Lemma \ref{asy5} establishes that the composition is an isomorphism in degrees less than $\frac{g+4}{5}-16$. Therefore, the first map can have no kernel in codimension less than $\frac{g+4}{5}-16$.
Finally, for $i < \frac{g+4}{5}-16$, we have $A^i(\H_{5,g}) = R^i(\H_{5,g})$ by \cite[Theorem 1.4]{part1}.
The dimension of $R^i(\H_{5,g})$ follows from Corollary \ref{5span}.
\end{proof}

\section{Applications to the moduli space of curves and a generalized Picard rank conjecture} \label{app-sec}
In this section, we express the Chow rings we have computed in terms of some natural classes associated to the Hurwitz spaces. We use those expressions to prove Theorems \ref{taut} and \ref{GPRC}. The natural classes we discuss can be defined on $\Hp_{k,g}$ for any $k$. They are the kappa classes and loci parametrizing covers with certain ramification profiles.
\begin{definition}
We define the following three closed loci in $\Hp_{k,g}$:
\begin{enumerate}
    \item $T:=\overline{\{[\alpha:C\rightarrow \p^1]: \alpha^{-1}(q)=3p_1+p_2\cdots+p_{k-2}, \text{ for some } q\text{ and distinct } p_i \}}$
    \item $D:=\overline{\{[\alpha:C\rightarrow \p^1]:\alpha^{-1}(q)=2p_1+2p_2\cdots+p_{k-2}, \text{ for some } q\text{ and distinct } p_i\}}$
    \item $U:=\overline{\{[\alpha:C\rightarrow \p^1]:\alpha^{-1}(q)=4p_1+p_2\cdots+p_{k-2}, \text{ for some } q\text{ and distinct } p_i\}}$
\end{enumerate}
\end{definition}

\begin{center}
\includegraphics[width=6in]{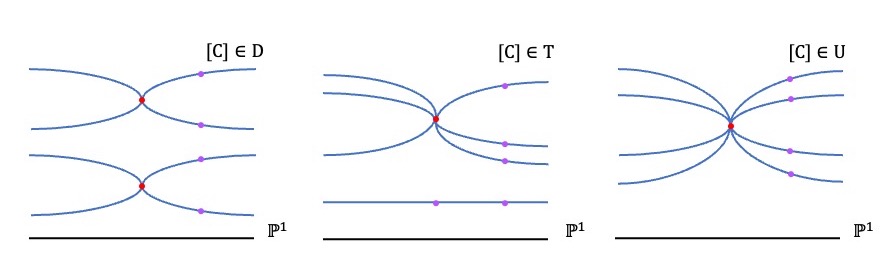}
\end{center}
The loci $T$ and $D$ have codimension $1$. The locus $U$ is one component of the intersection $T \cap D$, and $U$ has codimension $2$. 


Of course, one could consider other ramification behavior, but these three suffice for the applications in this paper. One benefit of these classes is that their push forwards to the moduli space of curves are known to be tautological. We make this precise in the next subsection. Then in the next two subsections, we rewrite the $\kappa$-classes and ramification loci in terms of CE classes to show that $[T], [D], [T]\cdot [D]$ and $[U]$ generate $R^*(\Hp_{k,g})$ as a module over $R^*(\M_g)$ in degrees $k = 4, 5$ respectively.

\subsection{Push forwards to $\M_g$}
To push forward cycles from the Hurwitz stack to $\Mg$, we first need to show that the relevant forgetful maps are proper. 
Consider the gonality stratification on the moduli space of curves:
\[
\M^d_{g}:=\{[C]\in \Mg: C \text{ has a } g^1_d \}.
\]
Because we don't require base point freeness in the equation above, we have the inclusions
$\M^d_{g}\subset \M^{d+1}_{g}$. Because gonality is lower semi-continuous,  $\Mg\setminus \M^d_{g} $ is open for any $d$. We have the map 
\[
\beta: \Hp_{k,g}\rightarrow \Mg
\]
obtained by forgetting the map to $\p^1$. After removing curves of lower gonality, we obtain a proper map
\[
\beta_k:\Hp_{k,g}\setminus \beta^{-1}(\Mg^{k-1})\rightarrow \Mg\setminus \Mg^{k-1},
\]
essentially by the same proof as \cite[Proposition 2.3]{BV}. 

\begin{rem}
If $k=3$ or $5$ and $g$ is sufficiently large, the maps $\beta_k$ are actually closed embeddings. See \cite[Proposition 2.3]{BV} for the $k=3$ case. On the other hand, the map $\beta_4$ is not injective on points because bielliptic curves admit infinitely many degree $4$ maps to $\p^1$.
\end{rem}

Because $\Mg\setminus \Mg^k$ is open in $\Mg$, there is a restriction map $A^*(\Mg)\rightarrow A^*(\Mg\setminus \Mg^k)$.
\begin{definition}\label{tautring}
The tautological ring $R^*(\Mg\setminus \Mg^k)$ of $\Mg\setminus \Mg^k$ is defined to be the image of the tautological ring $R^*(\Mg)$ under the restriction map 
$A^*(\Mg)\rightarrow A^*(\Mg\setminus \Mg^k)$.
\end{definition}

We need the following result of Faber-Pandharipande \cite{FP}, which concerns push forwards of classes of ramification loci quite generally. Let $\mu^1,\dots,\mu^m$ be $m$ partitions of equal size $k$ and length $\ell(\mu^i)$ that satisfy
\[
2g-2+2k=\sum_{i=1}^m (d-\ell(\mu^i)).
\]

Faber and Pandharipande use the Hurwitz space $\Hp_g(\mu^1,\dots,\mu^m)$ that parametrizes morphisms $\alpha:C\rightarrow \p^1$ that has marked ramification profiles $\mu^1,\dots,\mu^m$ over $m$ ordered points of the target and no ramification elsewhere. Two morphisms are equivalent if they are related by composition with an automorphism on $\pp^1$. By the Riemann-Hurwitz formula, these are covers of genus $g$ and degree $k$. They then consider the compactification by admissible covers $\overline{\Hp}_g(\mu^1,\dots,\mu^m)$. It admits a natural map to the moduli space of stable curves with marked points by forgetting the map to $\p^1$:
\[
\rho:\overline{\Hp}_g(\mu^1,\dots,\mu^m)\rightarrow \overline{\M}_{g,\sum_{i=1}^m\ell(\mu^i)}.
\]
\begin{thm}[Faber-Pandharipande \cite{FP}]
The pushforwards $\rho_*(\overline{\Hp}_g(\mu^1,\dots,\mu^m))$ are tautological classes in $A^*(\overline{\M}_{g,\sum_{i=1}^m\ell(\mu^i)})$.
\end{thm}
We then have the following diagram:
\[
\begin{tikzcd}
                                                              &                                       &                     & {\overline{\Hp}_{g}(\mu^1,\dots,\mu^m)} \arrow[d, "\rho"] \\
&&& \overline{\M}_{g,\sum_{i=1}^m\ell(\mu^i)} \arrow{d} \\
{\Hp_{k,g}\setminus\beta^{-1}(\Mg^{k-1})} \arrow[r, "\beta_k"] & \Mg\setminus\Mg^{k-1} \arrow[r, hook] & \Mg \arrow[r, hook] & \overline{\M}_g                               
\end{tikzcd}                  
\]
Because the tautological ring is closed under forgetting marked points and under the pullback from $\overline{\M}_g$ to $\Mg$, it follows that the image of $[\overline{\Hp}_g(\mu^1,\dots,\mu^m)]$ in $A^*(\Mg\setminus \Mg^{k-1})$ is a tautological class.

\begin{cor}\label{TDUpush}
Let $k\in\{3,4,5\}$. Then the classes $\beta_{k*}[T]$, $\beta_{k*}[D]$, $\beta_{k*}[U]$, and $\beta_{k*}([T]\cdot [D])$ lie in the tautological ring of $\Mg\setminus \Mg^{k-1}$.
\end{cor}
\begin{proof}
We explain the proof in the case $k=5$. The other cases are similar.
The image of $T$, $D$, and $U$ in $\Mg\setminus\Mg^{k-1}$ are the images of the corresponding spaces considered by Faber-Pandharipande. Indeed, for $T$, take $\mu_1=(3,1,1)$ and $\mu_i=(2,1,1,1)$ for all other $i$. For $D$, take $\mu_1=(2,2,1)$ and $\mu_i=(2,1,1,1)$ for all other $i$. For $U$, take $\mu_1=(4,1)$ and $\mu_i=(2,1,1,1)$ for all other $i$. 
\par One can see that the image of $T \cap D$ under $\beta_k$ is supported on the image of the following three spaces considered by Faber-Pandharipande:
\begin{enumerate}
    \item The image of the space with $\mu_1 = (4, 1)$ and all other $\mu_i = (2, 1, 1, 1)$
    \item The image of the space with $\mu_1 = (3, 2)$ and all other $\mu_i = (2, 1, 1, 1)$
    \item The image of the space with $\mu_1 = (3, 1, 1)$ and $\mu_2 = (2, 2, 1)$
\end{enumerate}
It follows that the pushforward of $[T]\cdot [D]$ is a linear combination of the restrictions of images of the above three spaces.
Hence, $\beta_{k*}([T] \cdot [D])$ is also tautological.
\end{proof}

\subsection{Formulas in degree $4$}
In this section, we compute formulas for the some of the natural classes on $\Hp_{4,g}$. We will do the computations in $A^*(\H_{4,g})$ in order to simplify the intersection theory calculation. This simplification is of no consequence to the end results because of the isomorphism $A^*(\H_{4,g})\cong A^*(\Hp_{4,g})$.

Deopurkar-Patel \cite[Proposition 2.8]{DP2} computed formulas for the classes of $T$ and $D$ in terms of $\kappa_1$ and $a_1$. In \cite[Example 3.12]{part1} we explained how to write the $\kappa$-classes in terms of CE classes, so we obtain the following.

\begin{lem} \label{4codim1}
The following identities hold in $A^1(\H_{4,g})$
\[\kappa_1 = (12g + 24)a_1 - 12a_2', \qquad [T] = (24g + 60)a_1 - 24a_2', \qquad [D] = (-32g - 80)a_1 + 36a_2'.\]
\end{lem}

Next, we compute the codimension two class $[U]$. In particular, we will see that $[U]$ is not in the span of products of codimension $1$ classes, from which it follows that the classes of $[T], [D], [U]$ generate $R^*(\H_{4,g})$ as a ring.

\begin{lem}\label{4quad}
The class of the quadruple ramification stratum $U$ on $\H_{4,g}$ is \[[U]= 36a_1a_2' -(32g + 80)a_1^2+(4g + 4)a_2-(4g+4)b_2.\]
Modulo the relations from Proposition \ref{4image}, we have $[U]=4a_3'$. 
\end{lem}
\begin{proof}
The fibers of a degree $4$ cover $\alpha:C\rightarrow \p^1$ are given by the base locus of a pencil of conics. A pencil of conics has base locus $4p$ if and only if every element of the pencil is tangent to a given line $L$ and $2L$ is a member of the pencil.
Equivalently, $4p$ is the base locus of a pencil of conics if and only if in some choice of local coordinates $x, y$ at $p$
\begin{enumerate}
    \item[(U1)] All members of the pencil are tangent to the line $y = 0$ at $p$, i.e. have vanishing coefficient of $x, 1$.
    \item[(U2)] Some member of the pencil is a multiple of $y^2$, i.e. has vanishing coefficient of $1, x, y, x^2, xy$.
\end{enumerate}
Note that the base locus of a pencil containing two double lines is not a \emph{curve-linear scheme} (i.e. a subscheme of smooth curve) since it has two dimensional tangent space at the intersection point. Therefore, if
If $p$ is a point of quadruple ramification on a smooth curve $C \xrightarrow{\alpha} \pp^1$, then the line $L \subset (\pp E_\alpha^\vee)_{\alpha^{-1}(\alpha(p))} \cong \pp^2$ is unique. That is, there is a unique direction and member of the pencil satisfying (U1) and (U2).

We will use the theory of restricted bundles of principal parts developed in Section 6 to characterize the covers satisfying these conditions. 
Let $X := \pp T_{\pp \E^\vee/\P} \times_{\P} \pp \F$. The first factor $\pp T_{\pp \E^\vee/\P}$ keeps track of a ``$x$-direction" and the second factor $\pp \F$ keeps track of a particular member of the pencil. We will apply the constructions of Section 6 to the tower
\[X \xrightarrow{a} \pp \E^\vee \xrightarrow{\gamma} \P.\]
In particular, pulling back the dual of the tautological sequence on the $\pp T_{\pp \E^\vee/\P}$ factor, we obtain a filtration on $X$
\[0 \rightarrow \Omega_y \rightarrow a^*\Omega_{\pp\E^\vee/\P} \rightarrow \Omega_x \rightarrow 0. \]
Meanwhile, pulling back the dual of the tautological sequence from the $\pp \F$ we obtain a quotient 
\[a^* \gamma^* \F^\vee \rightarrow \O_{\pp \F}(1) \rightarrow 0\]
Tensoring with $a^*\O_{\pp \E^\vee}(2)$, we obtain a filtration of $a^* \W = a^* (\gamma^*\F^\vee \otimes \O_{\pp \E^\vee}(2))$:
\[0 \rightarrow K \rightarrow a^*\W \rightarrow \O_{\pp \F}(1) \otimes \O_{\pp \E^\vee}(2) =: \W' \rightarrow 0.\]

To track the data in (U1) and (U2) we define $Q := P^{S \subset S'}_{\pp \E^\vee/\P}(\W \to \W')$
where
$S = \{1, x\}$ and $S' = \{1, x, y, x^2, xy\}$. This is represented by the diagram
      \begin{equation} \label{Qdiag}
   \begin{tikzpicture}[scale = .6]
     \draw[white] (3, 0) circle (5pt);
 \filldraw (0, 0) circle (5pt);
\draw (0, -1) circle (5pt);
 \filldraw (1, 0) circle (5pt);
 \draw (2, 0) circle (5pt);
 \begin{scope}[xshift=2cm]
 \fill[blue] (0,0) -- (90:1ex) arc (90:270:1ex) -- cycle;
 \end{scope}
  \begin{scope}[yshift=-1cm]
 \fill[blue] (0,0) -- (90:1ex) arc (90:270:1ex) -- cycle;
 \end{scope}
   \begin{scope}[yshift=-1cm]
 \fill[blue] (0,0) -- (90:1ex) arc (90:270:1ex) -- cycle;
 \end{scope}
 \draw (1, -1) circle (5pt);
    \begin{scope}[yshift=-1cm, xshift = 1cm]
 \fill[blue] (0,0) -- (90:1ex) arc (90:270:1ex) -- cycle;
 \end{scope}
\end{tikzpicture}
   \end{equation}
There is a natural quotient $a^*P^2_{\pp \E^\vee/\P}(\W) \to Q$, corresponding to the picture below.
\begin{center}
   \begin{tikzpicture}[scale = .6]
\begin{scope}[xshift=-5cm]
 \filldraw (0, 0) circle (5pt);
\filldraw (0, -1) circle (5pt);
 \filldraw (1, 0) circle (5pt);
\filldraw (0, -2) circle (5pt);
 \filldraw (2, 0) circle (5pt);
\filldraw (1, -1) circle (5pt);
\end{scope} 
\node at (-1.5, -.5) {$\longrightarrow$};
     \draw[white] (3, 0) circle (5pt);
 \filldraw (0, 0) circle (5pt);
\draw (0, -1) circle (5pt);
 \filldraw (1, 0) circle (5pt);
 \draw (2, 0) circle (5pt);
 \begin{scope}[xshift=2cm]
 \fill[blue] (0,0) -- (90:1ex) arc (90:270:1ex) -- cycle;
 \end{scope}
  \begin{scope}[yshift=-1cm]
 \fill[blue] (0,0) -- (90:1ex) arc (90:270:1ex) -- cycle;
 \end{scope}
   \begin{scope}[yshift=-1cm]
 \fill[blue] (0,0) -- (90:1ex) arc (90:270:1ex) -- cycle;
 \end{scope}
 \draw (1, -1) circle (5pt);
    \begin{scope}[yshift=-1cm, xshift = 1cm]
 \fill[blue] (0,0) -- (90:1ex) arc (90:270:1ex) -- cycle;
 \end{scope}
\end{tikzpicture}
\end{center}

As discussed in Section \ref{re4} the Casnati--Ekedahl theorem determines a global section $\delta^{\mathrm{univ}}$ of $\W$ whose vanishing is the universal curve. The induced section of $Q$
\begin{equation} \label{Qsec} \O_X \xrightarrow{a^* \delta^{\mathrm{univ}}{'}}  a^*P^2_{\pp \E^\vee/\P}(\W) \to Q\end{equation}
vanishes at a point of $X$ over $p$
precisely when conditions (U1) and (U2) above are satisfied at $p$ for the corresponding direction and member of the pencil.
Let $\widetilde{U}$ be the vanishing locus of the section in \eqref{Qsec}.

The map $a$ sends $\widetilde{U}$ one-to-one onto the universal quadruple ramification point. In turn, the universal quadruple ramification point maps generically one-to-one onto $U$, so 
\[[U] = \pi_*\gamma_*a_*[\widetilde{U}].\]
Since all fibers of the map $\widetilde{U} \to U$ are finite we have $\dim \widetilde{U} = \dim U$. Note that $X$ has relative dimension $2$ over $\pp \E^\vee$, which has relative dimension $3$ over $\H_{4,g}$. Thus, we have 
\[\codim(\widetilde{U} \subset X) = \codim(U \subset \H_{4,g}) + \text{relative dim of $X/\H_{4,g}$} = 2 + (2 + 3) = 7.\]
Meanwhile, $\rank \Omega_x = \rank \Omega_y = \rank \W' = \rank K = 1$. 
Each dot in the diagram \eqref{Qdiag} corresponds to a piece of a filtration of $Q$. The filled dots $\CIRCLE$ correspond to pieces of rank $2$ and half-filled dots ${\color{blue} \LEFTcircle }$ correspond to pieces of rank $1$.
Hence, $\rank Q = 7$. In particular, $\codim(\widetilde{U} \subset X) = \rank Q$, so $[\widetilde{U}] = c_7(Q).$ The top Chern class of $Q$ can be computed using its filtration, and its push forward to $\H_{4,g}$ is computed in Macaulay2 \cite{github}, which gives the expressions in the statement of the Lemma.
\end{proof}

In the example below, we provide expressions for some other codimension $2$ classes in terms of our preferred generators.
\begin{example}
Using the relations provided in the code, we can rewrite $c_2$ in terms of our preferred generators as
\begin{equation} \label{c2}
c_2=\frac{3}{g^2+4g+3}a_1^2-\frac{8}{g^3+6g^2+11g+6}a_3'
\end{equation}
Using \cite[Example 3.12]{part1}, we can compute

 \begin{align} 
\kappa_2 &=  
a_1b_2' -6a_1a_2' +(6g+6)a_1^2 -(6g-6)a_2+(g-3)b_2\\
& \qquad -(2g^3+6g^2+6g-14)c_2
+
4a_3' \notag \\
&= \frac{44g^2+200g+300}{g^2+4g+3}a_1^2-\frac{44}{g+1}a_1a_2'+\frac{2g^3-32g^2+138g-12}{3g^3+18g^2+33g+18}a_3'. \label{k2}
 \end{align}
Since the coefficient of $a_3'$ is non-zero in \eqref{c2} (resp. \eqref{k2}), we see that $c_2$ (resp. $\kappa_2$) may be used instead of $a_3'$ as the generator of $R^*(\H_{4,g})$ in codimension $2$. 
\end{example}

We can now prove Theorem \ref{GPRC} in when $k = 4$.
\begin{proof}[Proof of Theorem \ref{GPRC}, $k = 4$]
By Lemmas \ref{4codim1} and \ref{4quad} and Theorem \ref{main}, it follows that $[T], [D], [U]$ generate $R^*(\H_{4,g})$. Moreover, $R^i(\H_{4,g}) \to A^i(\H_{4,g}^{\nf})$ is surjective in degrees $i  \leq \frac{g+3}{4} - 4$ by Theorem \ref{main} (2). We have that $A^*(\H_{4,g}^{\nf}) \to A^*(\H_{4,g}^{s})$ is surjective and the ideal generated by $T, D, U$ is in the kernel. Hence, $A^i(\H_{4,g}^{s}) = 0$ for $i  \leq \frac{g+3}{4} - 4$.
\end{proof}

Above, we showed that $[T], [D], [U]$ generate $R^*(\H_{4,g})$ as a \emph{ring}. We now show that $[T], [D], [U], [T]\cdot [D]$ generate $R^*(\H_{4,g})$ as a \emph{module} over $\qq[\kappa_1]$.

\begin{lem} \label{first4}
The following are true
\begin{enumerate}
    \item $R^1(\H_{4,g})$ is spanned by $[T]$ and $[D]$. Alternatively, it is spanned by $[T]$ and $\kappa_1$.
    \item $R^2(\H_{4,g})$ is spanned by $[T]\kappa_1, [D]\kappa_1, [T] \cdot [D]$ and $[U]$.
    \item $R^3(\H_{4,g})$ is spanned by $\kappa_1^2 [T], \kappa_1^2 [D], \kappa_1 [U]$
    \item $R^4(\H_{4,g})$ is spanned by $\kappa_1^4$ and $\kappa_1^2 [U]$.
    \item For $i\geq 5$, $R^i(\H_{4,g})$ is spanned by $\kappa_1^i$. 
\end{enumerate}
\end{lem}
\begin{proof}
(1) By Lemma \ref{4codim1}, any pair of $[T], [D], \kappa_1$ span $R^1(\H_{4,g})$.

(2) By Corollary \ref{rcor}, we have that  $R^2(\H_{4,g})$ is spanned by $\{a_1^2, a_1a_2', a_2'^2, a_3'\}$. Hence, Lemma \ref{4quad} shows that $[U]$ and products of codimension $1$ classes span $R^2(\H_{4,g})$.

(3) Since $a_1, a_2', a_3'$ generate $R^*(\H_{4,g})$ as a ring, the classes
$\{a_1^3, a_1^2a_2', a_1a_2'^2, a_2'^3, a_1a_3', a_2'a_3'\}$ span $R^3(\H_{4,g})$. To show that $\kappa_1^2 [T], \kappa_1^2 [D],$ and $\kappa_1[U]$ span $R^3(\H_{4,g})$,
we first rewrite them in terms of CE classes.
It then suffices to see that these three classes, together with the codimension $3$ relations $r_1, r_2, r_3$ of Section \ref{ach}, span $\{a_1^3, a_1^2a_2', a_1a_2'^2, a_2'^3, a_1a_3', a_2'a_3'\}$.
One way to accomplish this is as follows. 
By Corollary \ref{rcor}, $\{a_1a_3',a_2'^3,a_2'a_3'\}$ is a spanning set modulo $r_1, r_2, r_3$ and one can readily rewrite
$\kappa_1^2[T], \kappa_1^2 [D],$ and $\kappa_1[U]$ in terms of $\{a_1a_3',a_2'^3,a_2'a_3'\}$ modulo the relations. We record the coefficients of these expressions in a $3 \times 3$ matrix. The determinant of this matrix has non-vanishing determinant for all $g$, so we conclude that $\kappa_1^2[T], \kappa_1^2 [D],$ and $\kappa_1[U]$ are also a spanning set modulo the relations. The calculation of the determinant is provided at \cite{github}.

(4) The proof is similar to the previous part.  By Corollary \ref{rcor}, $\{a_2'^4, a_3'^2\}$ spans the degree $4$ piece of $\qq[a_1, a_2', a_3']/\langle r_1, r_2, r_3, r_4 \rangle$.
We then write a $2 \times 2$ matrix of coefficients that expresses $\kappa_1^4$ and $\kappa_1^2[U]$ in terms of $\{a_2'^4, a_3'^2\}$ modulo the relations. We then check that the determinant is non-vanshing.

(5) From a direct calculation provided in the code, we see that $\kappa_1^i$ is a nonzero multiple of $a_2'^i$ for $5 \leq i \leq 10$. For all $i \geq 11$, a monomial of degree $i$ in the generators $a_1, a_2', a_3'$ can be written as a product of monomials having degrees between $5$ and $10$, so the claim follows.
\end{proof}

\begin{proof}[Proof of Theorem \ref{taut}, $k = 4$]
By Lemma \ref{first4}, we see that every class in $R^*(\H_{4,g})$ is expressible as a polynomial in $\kappa_1$ times $[T], [D], [T] \cdot [D],$ or $[U]$. By Corollary \ref{TDUpush}, the push forwards of $[T], [D], [T] \cdot [D], [U]$ are tautological, so by push-pull, the push forwards of all classes in $R^*(\H_{4,g})$ are tautological on $\M_g \smallsetminus \M_g^3$.
\end{proof}

\subsection{Formulas in degree 5} \label{formulas5}
As in the previous section, we will perform the calculations on the spaces $\H_{5,g}$ instead of $\Hp_{5,g}$. As in degree $4$, the codimension $1$ identities are easily converted from Deopukar-Patel \cite[Proposition 2.8]{DP2} and \cite[Example 3.12]{part1}, which computes $\kappa_1$ in terms of CE classes.

\begin{lem} \label{5codim1}
The following identities hold in $A^1(\H_{5,g})$
\[\kappa_1 = (12g + 36)a_1 -12a_2' \qquad [T] =(24g + 84)a_1-24a_2' \qquad [D] = -(32g + 112)a_1 + 36a_2'.\]
\end{lem}

Using the method explained in \cite[Example 3.12]{part1}, it is not difficult to compute $\kappa_2$ in terms of CE classes with our code \cite{github}.
\begin{lem} \label{5k2}
The following identities hold in $A^2(\H_{5,g})$
\[\kappa_2 = (6g^2+24g+40)c_2-6a_1^2+(-7g+2)a_2-7a_1a_2'+(2g+2)b_2+2a_1b_2'+5a_3'-b_3'.\]
Modulo the relations found in Lemma \ref{asy5},
\[\kappa_2 =(30g+66)a_1^2+(-21g+2)a_2-21a_1a_2'-(10g^3+66g^2+104g)c_2 .\]
\end{lem}

Next, we wish to compute $[U]$ in terms of CE classes, which will require more work and geometric input. Once we have $[U]$ in terms of CE classes, it will not be hard to see that $[T], [D], [U]$ and $[T] \cdot [U]$ generate $R^*(\H_{5,g})$ as a module over $\qq[\kappa_1, \kappa_2]$. However, in contrast with the case $k = 4$, the classes $[T], [D], [U]$ do not generate $R^*(\H_{5,g})$ as a ring, so additional work is needed to prove the vanishing results for $A^i(\H_{5,g}^s)$. We do this by constructing the universal triple ramification point and showing that an additional codimension $2$ class needed to generate $R^*(\H_{5,g})$ as a ring is supported on $T$.

For these last computations, we work with the realization of the universal curve $\C \subset G(2, \F)$ as the vanishing locus of a section of a rank $6$ vector bundle, as we now describe. On $\pi: \P \to \H_{5,g}$, Casnati's structure theorem in degree $5$ determines a universal injection $\eta^{\mathrm{univ}}: \E' \to \wedge^2 \F$. Let $\mathcal{Q}$ be the rank $6$ cokernel. Let $\mu: G := Gr(2, \F) \to \P$ be the Grassmann bundle.
Then $\C \subset G$ is defined by the vanishing of the composition 
\[\O_{G}(-1) := \O_{\pp(\wedge^2 \F)}(-1)|_G \to \mu^*(\wedge^2 \F) \to \mu^* \Q,\]
which we view as a section $\sigma$ of $\mu^*\mathcal{Q} \otimes \O_G(1) =: \W$. Studying appropriate principal parts of this section $\sigma$ of $\W$ on $G$ over $\P$ helps us describe when $\C \to \P$ has a point of higher order ramification.

Precisely, the universal curve has a triple (resp. quadruple) ramification point at $p \in \C \subset G$ if and only if there exists a direction $x$ in $(T_{G/\P})_p$ such that 
\begin{enumerate} 
\item the coefficient of $x$ vanishes in all equations. This implies that the universal curve has a vertical tangent vector in the $x$ direction, and so is ramified at $p$.
\item Let $y_1, \ldots, y_5$ be the remaining first order coordinates on $(T_{G/\P})_p$.
Locally $\sigma$ corresponds to $6$ equations on $G$. Since the universal curve is smooth, when we expand these equations to first order, the coefficients of $y_1, \ldots, y_5$ must span a five-dimensional space. That is, on $\C$ each $y_i$ may be solved for as a power series in $x$ with leading term order $2$. Moreover, there is also a ``distinguished equation" whose first order parts are all zero. This ``distinguished equation" will correspond to a particular quotient of $\W$. 
\item After substituting for $y_i$ as a power series in $x$ using (2), all equations vanish to order $2$ (resp. order 3). This is only a condition on the distinguished equation (the substitutions for $y_i$ were determined so that the other five are identically zero). For order $2$ vanishing, this condition is just that the coefficient of $x^2$ in the distinguished equation is zero. 
For order $3$ vanishing, this will involve expanding through the coefficients of $xy_i$ and $x^3$. 
\end{enumerate}
Note that because $\C$ is smooth over $\H$, the distinguished direction $x$ and distinguished equation of (2) are unique.

\par Let $X:=\p T_{G/\P}\times_{\P}\p \W^\vee$. The first factor keeps track of an ``$x$-direction" and the second factor keeps track of a ``distinguished equation" among the equations. We apply the constructions of Section 6 to the tower
\[
X\xrightarrow{a}G\xrightarrow{\mu}\P.
\]
The pullback to $X$ of the dual of the tautological sequence on $\p T_{G/\P}$ gives a filtration
\[
0\rightarrow \Omega_y\rightarrow a^*\Omega_{G/\P}\rightarrow \Omega_x\rightarrow 0.
\]
Meanwhile, the pullback of the dual of the tautological sequence on $\p \W^{\vee}$ gives a quotient
\[
a^*\mu^* \W \rightarrow \O_{\p \W^{\vee}}(1) =: \W' \rightarrow 0.
\]
Let $S = \{1, x\}$ and $S' = \{1, x, y, x^2\}$ and set $M := P_{G/\P}^{S \subset S'}(\W \to \W')$, which is a quotient $a^*P^2_{G/\P}(\W)$ corresponding to (\ref{6T}A), pictured again below. The bundles that appear in the filtration are listed in the corresponding location to the right.
\begin{equation} \label{t5diag}
\begin{tikzpicture}
 \filldraw (0, 0) circle (5pt);
\draw (0, -1) circle (5pt);
 \filldraw (1, 0) circle (5pt);
 \draw (2, 0) circle (5pt);
 \begin{scope}[xshift=2cm]
 \fill[blue] (0,0) -- (90:1ex) arc (90:270:1ex) -- cycle;
 \end{scope}
  \begin{scope}[yshift=-1cm]
 \fill[blue] (0,0) -- (90:1ex) arc (90:270:1ex) -- cycle;
 \end{scope}
 \draw[white] (-.5, -1.2) circle (5pt);
\end{tikzpicture}
\hspace{.8in}
\begin{tikzpicture}
\node at (0, 0) {$\W$};
\node at (3, 0) {$\W \otimes \Omega_x$};
\node at (6, 0) {\color{blue} $\W' \otimes \Omega_x^2$}; 
\node at (0, -1) {\color{blue} $\W' \otimes \Omega_y$};
\end{tikzpicture}
\end{equation}
The bundle $M$ measures the values and coefficients of $x$ in the equations, as well as the coefficients of the $y_i$ and $x^2$ in a distinguished equation. It has rank $18$.

A section of $a^*P^2_{G/\P}(\W)$ induces a section of $M$.
In particular, the global section $\sigma$ of $\W$ induces a section $\sigma'$ of $a^*P^2_{G/\P}(\W)$, which then gives a section $\sigma''$ of $M$.
We claim that this section $\sigma''$ vanishes at some point $\tilde{p} \in X$ lying over $p \in G$ if and only if conditions (1) -- (3) above are satisfied (to order $2$) for the distinguished direction and distinguished equation referred to by $\tilde{p}$. In more detail:
the left $\CIRCLE = \W$ corresponds to the condition $p \in \C$; the right $\CIRCLE = \W \otimes \Omega_x$ gives condition (1); the lower ${\color{blue} \LEFTcircle = \W' \otimes \Omega_y}$ corresponds to condition (2); and and the right ${\color{blue} \LEFTcircle = \W' \otimes \Omega_x^2}$ corresponds to condition (3).

Hence, the vanishing locus $\widetilde{T}$ of this induced section of $M$ maps isomorphically to the universal triple ramification point.
A computation similar to the one in Lemma \ref{4quad} shows that this vanishing occurs in the expected codimension, so $[\tilde{T}]=c_{18}(M)$.
The composition from $\tilde{T}\rightarrow \H_{5,g}$ is generically one-to-one onto its image, so we obtain an equality of classes
\[
[T]=\pi_*\mu_*a_*[\tilde{T}].
\]
This pushforward can be computed using a computer, and agrees with Lemma \ref{5codim1}.

The universal quadruple ramification point is cut out inside $\widetilde{T}$ by one more condition: namely, after replacing each $y_i$ with its power series in $x$ as in (2),
the coefficient of $x^3$ in the distinguished equation must vanish.

Since $y_i$ is of order $2$ in $x$, only the terms $xy_i$ can contribute to the coefficient of $x^3$. We already know that the coefficients of $1, y_1, \ldots, y_5, x, x^2$ vanish in the distinguished equation (corresponding to the shape \eqref{t5diag}).
We therefore wish to study the expansion of the distinguished equation through its coefficients of $xy_1, \ldots, xy_5$ and $x^3$. This will correspond to two new dots (represented below in red). Let $S'' = \{1, x, y, x^2, xy, y^2, x^3\}$.
The part of the Taylor expansion we need corresponds to the bundle $N:= P_{G/\P}^{S \subset S''}(\W \to \W')$ from (\ref{U5}C), pictured below. The bundles in the filtration are listed in the corresponding location on the right.
 \begin{center}
    \begin{tikzpicture}
     \draw[white] (0, -1.2) circle (5pt);
 \filldraw (0, 0) circle (5pt);
\draw (0, -1) circle (5pt);
 \filldraw (1, 0) circle (5pt);
 \draw (2, 0) circle (5pt);
 \begin{scope}[xshift=2cm]
 \fill[blue] (0,0) -- (90:1ex) arc (90:270:1ex) -- cycle;
 \end{scope}
  \begin{scope}[yshift=-1cm]
 \fill[blue] (0,0) -- (90:1ex) arc (90:270:1ex) -- cycle;
 \end{scope}
   \begin{scope}[yshift=-1cm]
 \fill[blue] (0,0) -- (90:1ex) arc (90:270:1ex) -- cycle;
 \end{scope}
 \draw (1, -1) circle (5pt);
    \begin{scope}[yshift=-1cm, xshift = 1cm]
 \fill[red] (0,0) -- (90:1ex) arc (90:270:1ex) -- cycle;
 \end{scope}
  \draw (3, 0) circle (5pt);
    \begin{scope}[xshift = 3cm]
 \fill[red] (0,0) -- (90:1ex) arc (90:270:1ex) -- cycle;
 \end{scope}
 \end{tikzpicture}
 \hspace{.5in}
\begin{tikzpicture}
\node at (0, 0) {$\W$};
\node at (2.8, 0) {$\W \otimes \Omega_x$};
\node at (2.8*2, 0) {\color{blue} $\W' \otimes \Omega_x^2$}; 
\node at (2.8*3, 0) { \color{red} $\W' \otimes \Omega_x^3$}; 
\node at (0, -1) {\color{blue} $\W' \otimes \Omega_y$};
\node at (2.8, -1) {\color{red} $\W' \otimes \Omega_x \otimes \Omega_y$};
\end{tikzpicture}
 \end{center}
 Let $N^{\color{red} \LEFTcircle} \subset N$ be the kernel of $N \to M$. Visually, $N^{\color{red} \LEFTcircle}$ is
subbundle corresponding to the right-most partially filled circles, which is filtered by $\W' \otimes \Omega_x \otimes \Omega_y$ and $\W' \otimes \Omega_x^3$.
 By the definition of $\widetilde{T}$, on $\widetilde{T} \subset X$, the section of $N$ induced by $\sigma$ factors through $N^{\color{red} \LEFTcircle}$. We call this section $\sigma^{\color{red} \LEFTcircle}$.

To get a quadruple point, it needs to be the case that when we sub in the power series of the $y_i$'s in terms of $x$
into the distinguished equation, the coefficient of $x^3$ vanishes. This is the same as saying that the expansion of the distinguished equation
lies in the span of ``$x$ times" the $\{y, x^2\}$ parts of the other equations. This will correspond to vanishing of evaluation in a rank $1$ quotient of $N^{\color{red} \LEFTcircle}$ that we define below. This quotient will be isomorphic to $\W' \otimes \Omega_x^3$.

\begin{rem}
The vanishing order filtration on $N^{{\color{red} \LEFTcircle}}$ provides a \emph{sub}bundle $\W' \otimes \Omega_x^3 \subset N^{{\color{red} \LEFTcircle}}$. The construction of our desired \emph{quotient} $N^{{\color{red} \LEFTcircle}} \to \W' \otimes \Omega_x^3$ on $\widetilde{T}$ will crucially use the fact that the subschemes in the fibers of $\C \to \P$ are curve-linear (in particular, have $1$ dimensional tangent space). This is equivalent to the statement in (2) that the other $y_i$'s may be solved for as power series in $x$.
\end{rem}

To make this precise,
let $V$ be the kernel of $P_{G/P}^{\{1, x, y, x^2\}}(\O) \to \O$, which comes equipped with a filtration
\[0 \rightarrow \Omega_{x}^2 \to V \to a^*\Omega_{G/P} \rightarrow 0.\] The bundle $V$ is like the tangent bundle but ``with a bit of second order information in the distinguished direction." Considering the triple point inside $G$ referred to by each point of $\widetilde{T}$ determines a rank $2$ quotient $Q_{\mathrm{trip}}$ of $V$ on $\widetilde{T}$ that fits in a diagram
\begin{center}
\begin{tikzcd}
0 \arrow{r} &\Omega_x^2 \arrow{r} & V \arrow{d} \arrow{r} & a^*\Omega_{G/P} \arrow{d} \arrow{r} & 0 \\
& & Q_{\mathrm{trip}} \arrow{r} & \Omega_x.
\end{tikzcd}
\end{center}
Just as having a distinguished quotient of $a^*\Omega_{G/P}$ allowed us to refine bundles of principal parts in Section \ref{dir-ref}, so too does having this rank $2$ quotient of $V$. Let $L$ be the kernel of $Q_{\mathrm{trip}} \to \Omega_x$, so $L$ corresponds to the second order data along a triple ramification point. The map from upper left to lower right, $\Omega_x^2 \to L$, is non-vanishing because the square of the first order coordinate is non-zero on the triple point (this uses curve-linearity), so $L \cong \Omega_x^2$. Equivalently, the quotient $V \to Q_{\mathrm{trip}}$ does factor through $a^*\Omega_{G/\P}$ on any fiber (which would mean the fiber through $p$ had two-dimensional tangent space).
Now, $\ker(V \to \Omega_x)$ corresponds to the $\{y, x^2\}$ parts of our expansions. Similarly, $\ker(V \to \Omega_x) \otimes \Omega_x$ corresponds to the $\{xy, x^3\}$ parts.
Tensoring $\ker(V \to \Omega_x) \to L$ with $\W' \otimes \Omega_x$, we get the desired quotient 
\[N^{\color{red} \LEFTcircle} = \W' \otimes \Omega_x \otimes \ker(V \to \Omega_x) \to \W' \otimes  \Omega_x \otimes L \cong \W' \otimes \Omega_{x}^3.\]
The evaluation of $\delta^{\color{red} \LEFTcircle}$ in this quotient is zero precisely when condition (3) above is satisfied to order $3$.

Hence, the universal quadruple ramification point is determined by the vanishing of a section of a line bundle $\W' \otimes \Omega_x^3$ on $\widetilde{T}$. In particular,
\[[U] = \pi_*\mu_*a_*([\widetilde{T}] \cdot c_1(\W' \otimes \Omega_x^3)), \]
which we computed in Macaulay.
\begin{lem} \label{classU5}
The class of the ramification locus $U$ on $\H_{5,g}$ is 
\[
[U]=(12g+48)a_1^2-(4g+16)b_2-(4g^3+48g^2+192g+256)c_2-4a_1b_2'+4b_3'.
\]
Modulo the relations from  Lemma \ref{asy5},
\[
[U]=\frac{156g+468}{5}a_1^2-\frac{108g+216}{5}a_2-\frac{108}{5}a_1a_2'-\frac{52g^3+468g^2+1352g+1248}{5}c_2.
\]
\end{lem}

We now give additive generators for $R^*(\H_{5,g})$.
\begin{lem}\label{5taut}
Suppose $g\geq 2$. Then,
\begin{enumerate}
    \item $R^1(\H_{5,g})$ is spanned by $[T]$ and $[D]$. Alternately, it is spanned by $[T]$ and $\kappa_1$.
    \item $R^2(\H_{5,g})$ is spanned by $[T]\kappa_1$, $[D]\kappa_1$, $[T]\cdot [D]$, $[U]$, $\kappa_2$.
    \item $R^3(\H_{5,g})$ is spanned by $[T]\kappa_1^2, [D]\kappa_1^2, [T]\cdot [D] \kappa_1, [U]\kappa_1, [T]\kappa_2, [D]\kappa_2$.
    \item $R^4(\H_{5,g})$ is spanned by $[T]\kappa_1^3, \kappa_1^4, [T]\kappa_1\kappa_2,  [T]\cdot [D] \kappa_2, \kappa_2^2, \kappa_1^2\kappa_2, [U]\kappa_2$.
    \item $R^5(\H_{5,g})$ is spanned by $[T]\kappa_1^4,[T]\kappa_2^2,\kappa_1^5,\kappa_1\kappa_2^2$.
    \item $R^6(\H_{5,g})$ is spanned by $[T]\kappa_1^5,\kappa_1^6,\kappa_1^4\kappa_2$.
    \item $R^i(\H_{5,g})$ is spanned by $[T]\kappa_1^{i-1},\kappa_1^i$ for $i\geq 7$.
\end{enumerate}
\end{lem}
\begin{proof}
Using Lemmas \ref{5codim1}, \ref{5k2}, and \ref{classU5},
we can write down expressions for each class in the statement of the Lemma in terms of Casnati--Ekedahl classes.
Modulo our relations in Section \ref{asy5}, Macaulay gives a formula for these classes in terms of the spanning sets of Corollary \ref{5span}.

For each $i$, we can then write down a matrix whose entries are the coefficients of the expression for the classes in the statement of the lemma in terms of the CE spanning set. We then check if the determinant of the matrix of coefficients, which is a polynomial in $g$, has no positive integer roots.  For example, in codimension $1$, we have that $\{a_1,a_2'\}$ is a spanning set, and we have
\[
[T] =(24g + 84)a_1-24a_2' \qquad [D] = -(32g + 112)a_1 + 36a_2'.
\]
The matrix of coefficients
\[
\begin{pmatrix}
24g+84 & -24 \\
-32g-112 & 36 
\end{pmatrix}
\]
has determinant $96g+336$, which has no integer roots, so $[T]$ and $[D]$ span $R^1(\H_{5,g})$. A similar calculation shows that $[T]$ and $\kappa_1$ span $R^1(\H_{5,g})$. For $2\leq i\leq 6$, we repeat the process, and the determinants are calculated at \cite{github}. None of them has roots at any integer $g\geq 2$.

When $i \geq 7$, we use an argument similar to Section \ref{ach5}.
For $7 \leq i \leq 14$, we check that 
$[T]\kappa_1^{i-1}$ and $\kappa_1^i$ span, by showing that the matrix of coefficients to express these in terms of $a_1 a_2'^{i-1}$ and $a_2'^i$ is invertible. Because it $R^*(\H_{5,g})$ is generated in degrees $1$ and $2$, for $i \geq 15$, every monomial class in $R^*(\H_{5,g})$ is expressible as a product of two monomials, both of degree at least $7$.
Then the product of two such monomials is in the span of $\kappa_1^i, \kappa_1^{i-1} [T]$ and $\kappa_1^{i-2} [T]^2 = \kappa_1^{i-7}(\kappa_1^5 [T]^2)$. The last monomial is already in the span of the first two because $R^7(\H_{5,g})$ is spanned by $\kappa_1^7, \kappa_1^{6} [T]$. The last part (7) now follows.
\end{proof}
As a consequence, we finish the proofs of Theorem \ref{taut} and Theorem \ref{GPRC}.
\begin{proof}[Proof of Theorem \ref{taut}, $k = 5$]
By Lemma \ref{5taut}, we see that every class in $R^*(\H_{5,g})$ is expressible as a polynomial in the kappa classes times $[T], [D],$ or $[U]$. By Corollary \ref{TDUpush}, the push forwards of $[T], [D], [U]$ are tautological, so by push-pull, the push forwards of all classes in $R^*(\H_{5,g})$ are tautological on $\M_g$.
\end{proof}

\begin{proof}[Proof of Theorem \ref{GPRC}, $k=5$]
For $i$ in the range of the statement, we have $A^i(\H_{5,g}) = R^i(\H_{5,g})$. Thus, it
suffices to produce generators for $R^*(\H_{5,g})$ as a ring that are supported on $T$ and $D$. We know from Theorem \ref{main} (3) that $R^*(\H_{5,g})$ is generated by two classes in degree $1$ and two classes in degree $2$. The classes $[T]$ and $[D]$ generate $R^1(\H_{5,g})$. Then, we computed $\pi_* (\mu_* a_*([\widetilde{T}]) \cdot z)$, which is supported on $T$, in the code \cite{github}. The result is that
\[
\pi_* (\mu_* a_*([\widetilde{T}]) \cdot z)=(3g^2+24g+48)c_2-3a_1^2-3a_2+3b_2.
\]
Modulo the relations from Lemma \ref{rels5}, this class is given by
\begin{equation} \label{extraclass}
\pi_* (\mu_* a_*([\widetilde{T}]) =12a_1^2-24a_2-(12g^2+84g-144)c_2.
\end{equation}
Using Lemma \ref{classU5}, we see that $\pi_* (\mu_* a_*([\widetilde{T}]) \cdot z)$ and $[U]$ are independent modulo products of codimension $1$ classes. 
Since $R^*(\H_{5,g})$ is generated in codimension $1$ and $2$, we conclude that $R^*(\H_{5,g})$ is generated by $[T], [D], [U]$ and the class in \eqref{extraclass}, which are all supported on $T$ and $D$.
\end{proof}

\bibliographystyle{amsplain}
\bibliography{refs}

\providecommand{\bysame}{\leavevmode\hbox to3em{\hrulefill}\thinspace}
\providecommand{\MR}{\relax\ifhmode\unskip\space\fi MR }
\providecommand{\MRhref}[2]{%
  \href{http://www.ams.org/mathscinet-getitem?mr=#1}{#2}
}
\providecommand{\href}[2]{#2}
\begin{thebibliography}{10}

\bibitem{B3}
Manjul Bhargava, \emph{Higher composition laws. {II}. {O}n cubic analogues of
  {G}auss composition}, Ann. of Math. (2) \textbf{159} (2004), no.~2, 865--886.
  \MR{2081442}

\bibitem{B4}
\bysame, \emph{Higher composition laws. {III}. {T}he parametrization of quartic
  rings}, Ann. of Math. (2) \textbf{159} (2004), no.~3, 1329--1360.
  \MR{2113024}

\bibitem{BD4}
\bysame, \emph{The density of discriminants of quartic rings and fields}, Ann.
  of Math. (2) \textbf{162} (2005), no.~2, 1031--1063. \MR{2183288}

\bibitem{B5}
\bysame, \emph{Higher composition laws. {IV}. {T}he parametrization of quintic
  rings}, Ann. of Math. (2) \textbf{167} (2008), no.~1, 53--94. \MR{2373152}

\bibitem{BD5}
\bysame, \emph{The density of discriminants of quintic rings and fields}, Ann.
  of Math. (2) \textbf{172} (2010), no.~3, 1559--1591. \MR{2745272}

\bibitem{B}
S\o ren~K. Boldsen, \emph{Improved homological stability for the mapping class
  group with integral or twisted coefficients}, Math. Z. \textbf{270} (2012),
  no.~1-2, 297--329. \MR{2875835}

\bibitem{BV}
Michele Bolognesi and Angelo Vistoli, \emph{Stacks of trigonal curves}, Trans.
  Amer. Math. Soc. \textbf{364} (2012), no.~7, 3365--3393. \MR{2901217}

\bibitem{CL}
Samir Canning and Hannah Larson, \emph{The {C}how rings of the moduli spaces of
  curves of genus $7,8$ and $9$}, arXiv:2104.05820 (2021).

\bibitem{github}
\bysame, \emph{Low-degree-hurwitz},
  \url{https://github.com/src2165/Low-Degree-Hurwitz/}, 2021.

\bibitem{part1}
\bysame, \emph{Tautological classes on low-degree {H}urwitz spaces},
  arXiv:2103.09902v2 (2021).

\bibitem{CE}
G.~Casnati and T.~Ekedahl, \emph{Covers of algebraic varieties. {I}. {A}
  general structure theorem, covers of degree {$3,4$} and {E}nriques surfaces},
  J. Algebraic Geom. \textbf{5} (1996), no.~3, 439--460. \MR{1382731}

\bibitem{C}
Gianfranco Casnati, \emph{Covers of algebraic varieties. {II}. {C}overs of
  degree {$5$} and construction of surfaces}, J. Algebraic Geom. \textbf{5}
  (1996), no.~3, 461--477. \MR{1382732}

\bibitem{DP}
Anand Deopurkar and Anand Patel, \emph{The {P}icard rank conjecture for the
  {H}urwitz spaces of degree up to five}, Algebra Number Theory \textbf{9}
  (2015), no.~2, 459--492. \MR{3320849}

\bibitem{DP2}
\bysame, \emph{Syzygy divisors on {H}urwitz spaces}, Higher genus curves in
  mathematical physics and arithmetic geometry, Contemp. Math., vol. 703, Amer.
  Math. Soc., Providence, RI, 2018, pp.~209--222. \MR{3782468}

\bibitem{EH}
David Eisenbud and Joe Harris, \emph{3264 and all that---a second course in
  algebraic geometry}, Cambridge University Press, Cambridge, 2016.
  \MR{3617981}

\bibitem{EVW}
Jordan~S. Ellenberg, Akshay Venkatesh, and Craig Westerland, \emph{Homological
  stability for {H}urwitz spaces and the {C}ohen-{L}enstra conjecture over
  function fields}, Ann. of Math. (2) \textbf{183} (2016), no.~3, 729--786.
  \MR{3488737}

\bibitem{FP}
C.~Faber and R.~Pandharipande, \emph{Relative maps and tautological classes},
  J. Eur. Math. Soc. (JEMS) \textbf{7} (2005), no.~1, 13--49. \MR{2120989}

\bibitem{F}
Carel Faber, \emph{A conjectural description of the tautological ring of the
  moduli space of curves}, Moduli of curves and abelian varieties, Aspects
  Math., E33, Friedr. Vieweg, Braunschweig, 1999, pp.~109--129. \MR{1722541}

\bibitem{GSS}
Daniel~R. Grayson, Alexandra Seceleanu, and Michael~E. Stillman,
  \emph{Computations in intersection rings of flag bundles}, arXiv:1205.4190
  (2012).

\bibitem{M2}
Daniel~R. Grayson and Michael~E. Stillman, \emph{Macaulay2, a software system
  for research in algebraic geometry}, Available at
  \url{http://www.math.uiuc.edu/Macaulay2/}.

\bibitem{S2}
Daniel~R. Grayson, Michael~E. Stillman, Stein~A. Str{\o}mme, David Eisenbud,
  and Charley Crissman, \emph{{Schubert2: characteristic classes for varieties
  without equations. Version~0.7}}, A \emph{Macaulay2} package available at
  "http://www.math.uiuc.edu/Macaulay2/".

\bibitem{H}
John~L. Harer, \emph{Stability of the homology of the mapping class groups of
  orientable surfaces}, Ann. of Math. (2) \textbf{121} (1985), no.~2, 215--249.
  \MR{786348}

\bibitem{HM}
Joe Harris and Ian Morrison, \emph{Moduli of curves}, Graduate Texts in
  Mathematics, vol. 187, Springer-Verlag, New York, 1998. \MR{1631825}

\bibitem{I}
Eleny-Nicoleta Ionel, \emph{Relations in the tautological ring of
  {$\mathscr{M}_g$}}, Duke Math. J. \textbf{129} (2005), no.~1, 157--186.
  \MR{2155060}

\bibitem{L}
Hannah~K. Larson, \emph{Universal degeneracy classes for vector bundles on
  {$\Bbb{P}^1$} bundles}, Adv. Math. \textbf{380} (2021), 107563, 20.
  \MR{4200467}

\bibitem{MW}
Ib~Madsen and Michael Weiss, \emph{The stable moduli space of {R}iemann
  surfaces: {M}umford's conjecture}, Ann. of Math. (2) \textbf{165} (2007),
  no.~3, 843--941. \MR{2335797}

\bibitem{M}
Rick Miranda, \emph{Triple covers in algebraic geometry}, Amer. J. Math.
  \textbf{107} (1985), no.~5, 1123--1158. \MR{805807}

\bibitem{Mu}
Scott Mullane, \emph{The {H}urwitz space {P}icard rank conjecture for $d>g-1$},
  arXiv:2009.10063 (2020).

\bibitem{Mum}
David Mumford, \emph{Towards an enumerative geometry of the moduli space of
  curves}, Arithmetic and geometry, {V}ol. {II}, Progr. Math., vol.~36,
  Birkh\"{a}user Boston, Boston, MA, 1983, pp.~271--328. \MR{717614}

\bibitem{PV}
Nikola Penev and Ravi Vakil, \emph{The {C}how ring of the moduli space of
  curves of genus six}, Algebr. Geom. \textbf{2} (2015), no.~1, 123--136.
  \MR{3322200}

\bibitem{VW}
Ravi Vakil and Melanie~Matchett Wood, \emph{Discriminants in the {G}rothendieck
  ring}, Duke Math. J. \textbf{164} (2015), no.~6, 1139--1185. \MR{3336842}

\bibitem{VZ}
Jason van Zelm, \emph{Nontautological bielliptic cycles}, Pacific J. Math.
  \textbf{294} (2018), no.~2, 495--504. \MR{3770123}

\bibitem{Z}
Angelina Zheng, \emph{Rational cohomology of the moduli space of trigonal
  curves of genus 5}, arXiv:2007.12150 (2020).

\bibitem{Z2}
Angelina Zheng, \emph{Stable cohomology of the moduli space of trigonal
  curves}, arXiv:2106.07245 (2021).

\end{thebibliography}
\end{document}